\documentclass[10pt, a4paper]{amsart}

\usepackage{lmodern}
\usepackage[utf8]{inputenc}
\usepackage{amsmath}
\usepackage{graphicx}
\usepackage{amssymb}
\usepackage{esint}
\usepackage[dvipsnames]{xcolor}
\usepackage{tikz}
\usepackage{xxcolor}
\usepackage{floatrow}
\usepackage{color}
\usepackage{amsthm}
\usepackage{epsfig}
\usepackage[english]{babel}
\usepackage{hyperref}
\usepackage[normalem]{ulem}
\usepackage{mathrsfs}
\usepackage{tikz}
\usetikzlibrary{decorations.markings,backgrounds}
\usetikzlibrary{arrows.meta}
\usepackage{pgfplots}
\pgfplotsset{compat=1.18}
\usepackage{subfigure}
\usepackage{caption}
\usepackage{bbm}

\usepackage{stmaryrd}

\usepackage{float}
\usepackage{pst-plot}

\hypersetup{
	colorlinks,
	linkcolor={red!80!black},
	citecolor={blue!50!black},
	urlcolor={blue!80!black}
}
 \usepackage{diagbox}

\usepackage[a4paper,top=3.5 cm,bottom=3 cm,left=2.5 cm,right=2.5 cm]{geometry}

\newtheorem{theorem}{Theorem}[section]
\newtheorem{lemma}{Lemma}[section]
\newtheorem{prop}{Proposition}[section]
\newtheorem{cor}{Corollary}[section]
\newtheorem{remark}{Remark}[section]
\newtheorem{defn}{Definition}[section]

\numberwithin{equation}{section}

\def\to{\rightarrow}

\usepackage{dsfont}

\newcommand{\nchi}{{\raise.3ex\hbox{$\chi$}}}

\def\XXint#1#2#3{{\setbox0=\hbox{$#1{#2#3}{\int}$ }
		\vcen{\hbox{$#2#3$ }}\kern-.6\wd0}}

\usepackage{mathtools}

\makeatletter
\newcommand{\justified}{%
	\rightskip\z@skip%
	\leftskip\z@skip}
\makeatother
\usepackage{amsfonts}
\usepackage{aurical}

\newcommand\restr[2]{{
		\left.\kern-\nulldelimiterspace 
		#1 
		\vphantom{\big|} 
		\right|_{#2} 
}}

\newcommand{\var}{\varepsilon}

\DeclareFontFamily{U}{mathx}{\hyphenchar\font45}
\DeclareFontShape{U}{mathx}{m}{n}{<-> mathx10}{}
\DeclareSymbolFont{mathx}{U}{mathx}{m}{n}
\DeclareMathAccent{\widebar}{0}{mathx}{"73}

\renewcommand{\i}{\ifmmode\mathit{\mathchar"7010 }\else\char"10 \fi}
\renewcommand{\j}{\ifmmode\mathit{\mathchar"7011 }\else\char"11 \fi}

\def\char{{1\!\mbox{\rm l}}}

\usepackage{dsfont}

\definecolor{orange}{rgb}{1,.549,0}
\definecolor{GreenYellow }{rgb}{ 0.15,   0.69, 0}
\definecolor{Yellowone}{rgb}{ 0, 1., 0} \definecolor{Goldenrod }{rgb}{  0, 0.10, 0.84}
\definecolor{Dandelion }{rgb}{ 0, 0.29, 0.84} 
\definecolor{Apricot }{rgb}{ 0, 0.32, 0.52}
\definecolor{Peach }{rgb}{ 0, 0.50, 0.70} 
\definecolor{GreenYellow}{cmyk}{0.15,0,0.69,0}
\definecolor{RoyalPurple}{cmyk}{0.75,0.90,0,0}
\definecolor{Yellow}{cmyk}{0,0,1,0}
\definecolor{BlueViolet}{cmyk}{0.86,0.91,0,0.04}
\definecolor{Goldenrod}{cmyk}{0,0.10,0.84,0}
\definecolor{Periwinkle}{cmyk}{0.57,0.55,0,0}
\definecolor{Dandelion}{cmyk}{0,0.29,0.84,0}
\definecolor{CadetBlue}{cmyk}{0.62,0.57,0.23,0}
\definecolor{Apricot}{cmyk}{0,0.32,0.52,0}
\definecolor{CornflowerBlue}{cmyk}{0.65,0.13,0,0}
\definecolor{Peach}{cmyk}{0,0.50,0.70,0}
\definecolor{MidnightBlue}{cmyk}{0.98,0.13,0,0.43}
\definecolor{Melon}{cmyk}{0,0.46,0.5,0}
\definecolor{NavyBlue}{cmyk}{0.94,0.54,0,0}
\definecolor{YellowOrange}{cmyk}{0,0.42,1,0}
\definecolor{RoyalBlue}{cmyk}{1,0.50,0,0}
\definecolor{Orange}{cmyk}{0,0.61,0.87,0}
\definecolor{Blue}{cmyk}{1,1,0,0}
\definecolor{BurntOrange}{cmyk}{0,0.51,1,0}
\definecolor{Cerulean}{cmyk}{0.94,0.11,0,0}
\definecolor{Bittersweet}{cmyk}{0,0.75,1,0.24}
\definecolor{Cyan}{cmyk}{1,0,0,0}
\definecolor{RedOrange}{cmyk}{0,0.77,0.87,0}
\definecolor{ProcessBlue}{cmyk}{0.96,0,0,0}
\definecolor{Mahogany}{cmyk}{0,0.85,0.87,0.35}
\definecolor{SkyBlue}{cmyk}{0.62,0,0.12,0}
\definecolor{Maroon}{cmyk}{0,0.87,0.68,0.32}
\definecolor{Turquoise}{cmyk}{0.85,0,0.20,0}
\definecolor{BrickRed}{cmyk}{0,0.89,0.94,0.28}
\definecolor{TealBlue}{cmyk}{0.86,0,0.34,0.02}
\definecolor{Red}{cmyk}{0,1,1,0}
\definecolor{Aquamarine}{cmyk}{0.82,0,0.30,0}
\definecolor{OrangeRed}{cmyk}{0,1,0.50,0}
\definecolor{BlueGreen}{cmyk}{0.85,0,0.33,0}
\definecolor{RubineRed}{cmyk}{0,1,0.13,0}
\definecolor{Emerald}{cmyk}{1,0,0.50,0}
\definecolor{WildStrawberry}{cmyk}{0,0.96,0.39,0}
\definecolor{JungleGreen}{cmyk}{0.99,0,0.52,0}
\definecolor{Salmon}{cmyk}{0,0.53,0.38,0}
\definecolor{SeaGreen}{cmyk}{0.69,0,0.50,0}
\definecolor{CarnationPink}{cmyk}{0,0.63,0,0}
\definecolor{Green}{cmyk}{1,0,1,0}
\definecolor{Magenta}{cmyk}{0,1,0,0}
\definecolor{ForestGreen}{cmyk}{0.91,0,0.88,0.12}
\definecolor{VioletRed}{cmyk}{0,0.81,0,0}
\definecolor{PineGreen}{cmyk}{0.92,0,0.59,0.25}
\definecolor{Rhodamine}{cmyk}{0,0.82,0,0}
\definecolor{LimeGreen}{cmyk}{0.50,0,1,0}
\definecolor{Mulberry}{cmyk}{0.34,0.90,0,0.02}
\definecolor{YellowGreen}{cmyk}{0.44,0,0.74,0}
\definecolor{RedViolet}{cmyk}{0.07,0.90,0,0.34}
\definecolor{SpringGreen}{cmyk}{0.26,0,0.76,0}
\definecolor{Fuchsia}{cmyk}{0.47,0.91,0,0.08}
\definecolor{OliveGreen}{cmyk}{0.64,0,0.95,0.40}
\definecolor{Lavender}{cmyk}{0,0.48,0,0}
\definecolor{RawSienna}{cmyk}{0,0.72,1,0.45}
\definecolor{Thistle}{cmyk}{0.12,0.59,0,0}
\definecolor{Sepia}{cmyk}{0,0.83,1,0.70}
\definecolor{Orchid}{cmyk}{0.32,0.64,0,0}
\definecolor{Brown}{cmyk}{0,0.81,1,0.60}
\definecolor{DarkOrchid}{cmyk}{0.40,0.80,0.20,0}
\definecolor{Tan}{cmyk}{0.14,0.42,0.56,0}
\definecolor{Purple}{cmyk}{0.45,0.86,0,0}
\definecolor{Gray}{cmyk}{0,0,0,0.50}
\definecolor{Plum}{cmyk}{0.50,1,0,0}
\definecolor{Black}{cmyk}{0,0,0,1}
\definecolor{Violet}{cmyk}{0.79,0.88,0,0}
\definecolor{White}{cmyk}{0,0,0,0}
\definecolor{rltred}{rgb}{0.75,0,0}
\definecolor{rltgreen}{rgb}{0,0.5,0}
\definecolor{oneblue}{rgb}{0,0,0.75}
\definecolor{marron}{rgb}{0.64,0.16,0.16}
\definecolor{forestgreen}{rgb}{0.13,0.54,0.13}
\definecolor{purple}{rgb}{0.62,0.12,0.94}
\definecolor{dockerblue}{rgb}{0.11,0.56,0.98}
\definecolor{freeblue}{rgb}{0.25,0.41,0.88}
\definecolor{myblue}{rgb}{0,0.2,0.4}
\definecolor{Melon}{rgb}{ 0.46, 0.50, 0}
\definecolor{Melone}{rgb}{ 0, 0.46, 0.50}

\allowdisplaybreaks
\usepackage[colorinlistoftodos]{todonotes}
\usepackage{url}

\usepackage{graphicx,tikz}



\begin{document}
\title[The Boltzmann equation in the homogeneous critical regularity framework]{The Boltzmann equation in the homogeneous critical regularity framework}

\author{Jing Liu, Ling-Yun Shou, and Jiang Xu}

\date{}

\linespread{1.2}

\begin{abstract}
We construct a unique global solution with critical regularity to the Cauchy problem for the 3D Boltzmann equation with initial data near the Maxwellian in the spatially critical Besov space $\widetilde{L}^2_{\xi}(\dot{B}_{2,1}^{1/2})\cap\widetilde{L}^2_{\xi}(\dot{B}_{2,1}^{3/2})$. Furthermore, under the condition that the low-frequency part of the initial perturbation is bounded in $\widetilde{L}^2_{\xi}(\dot{B}_{2,\infty}^{\sigma_{0}})$ with $-3/2\leq\sigma_{0}<1/2$, it is shown that the solution converges to equilibrium in large times with the optimal rate of $\mathcal{O}(t^{-(\sigma-\sigma_{0})/2})$ in $\widetilde{L}^2_{\xi}(\dot{B}_{2,1}^{\sigma})$ for $\sigma>\sigma_0$, and the microscopic part decays at an enhanced rate of $\mathcal{O}(t^{-(\sigma-\sigma_{0})/2-1/2})$. This is the first work to address the global existence and large-time behavior of solutions to the Boltzmann equation in the homogeneous critical setting. We develop the hypocoercivity theory, which indicates Lyapunov functionals with different dissipation rates in low and high frequencies. Moreover, the so-called time-weighted Lyapunov energy argument is employed to obtain the optimal time-decay estimates.
\end{abstract}

\maketitle

\vspace{3mm}





\section{Introduction}\label{1dlg-S1}

It is well-known that the Boltzmann equation was written down by Ludwig Boltzmann in 1872, which describes the time evolution of a dilute gas of identical particles (see \cite{Cercignani-1994,Glassey-1996,Villanireview}). Let $F=F(t,x,\xi)\geq0$ stand for the density function of particles on the position $x\in\mathbb{R}^{3}$ with the velocity $\xi\in\mathbb{R}^{3}$ at the time $t$. We are concerned with the Cauchy problem, which is given by
\begin{equation}\label{1dlg-E1.1}
\partial_{t}F+\xi\cdot\nabla_{x}F=\mathcal{Q}(F,F),
\end{equation}
supplemented with initial data
\begin{equation}\label{1dlg-E1.2}
F(0,x,\xi)=F_{0}(x,\xi).
\end{equation}
Here, the bilinear collision operator $\mathcal{Q}$ is of the form
\begin{equation}\label{1dlg-E1.3}
\mathcal{Q}(F,G)=\int_{\mathbb{R}^{3}}\int_{\mathbb{S}^{2}}B(\xi-\xi_{\ast},\omega)\left(F_{\ast}^{\prime}G^{\prime}-F_{\ast}G\right)\,d\omega d\xi_{\ast}
\end{equation}
with 
$$F_{\ast}^{\prime}=F(t,x,\xi_{\ast}^{\prime}),\quad G^{\prime}=G(t,x,\xi^{\prime}),\quad F_{\ast}=F(t,x,\xi_{\ast}),\quad G=G(t,x,\xi).$$
The collision kernel $B(\xi-\xi_{\ast},\omega)$ stands for the physical interaction between particles. Throughout this paper, we suppose that for technical reasons
$$
B(\xi-\xi_{\ast},\omega)=|\xi-\xi_{\ast}|^{\gamma}B_0(\cos\theta),\quad 0\leq\gamma\leq1,\quad 0\leq B_0(\cos\theta)\leq C|\cos\theta|,
$$
which corresponds to the case of hard potentials with angular cutoff. Pre-collisional and post-collisional velocities are related by the following involutive transformation
$$
\xi_{\ast}^{\prime}=\xi_{\ast}+\big((\xi-\xi_{\ast})\cdot\omega\big)\,\omega,~~~\xi^{\prime}=\xi-\big((\xi-\xi_{\ast})\cdot\omega\big)\,\omega
$$
for $\xi,\xi_{\ast}\in\mathbb{R}^{3}$ and fixed $\omega\in\mathbb{S}^{2}$, according to the momentum and energy laws of two particles for elastic collisions:
$$\xi_{\ast}+\xi=\xi_{\ast}^{\prime}+\xi^{\prime},~~~|\xi_{\ast}|^{2}+|\xi|^{2}=|\xi_{\ast}^{\prime}|^{2}+|\xi^{\prime}|^{2}.$$

Mathematically, the first result is due to T. S. Carleman \cite{Carleman-1933}, who established the global existence of solutions to the homogeneous Boltzmann equation. For the inhomogeneous case, Grad \cite{Grad-1965} constructed local-in-time mild solutions to the mixed initial boundary value problem with periodic conditions. Ukai \cite{Ukai-1974} established the corresponding global existence in the small perturbation framework of the Maxwellian (Gaussian) distribution. The most general known solutions of the Boltzmann equation are the renormalized solutions introduced by Diperna and Lions \cite{Diperna-1989}, which exist globally in time for arbitrary initial data with finite mass, second moments and entropy. However, several questions about renormalized solutions still remain open. For instance, they are neither known to be unique nor known to conserve energy. Following from the line of investigation in \cite{Ukai-1974}, so far there have been fruitful efforts where initial data are assumed to be sufficiently close to the Maxwellian, see, \textit{e.g.} 
Guo \cite{Guo-2003, Guo-2004, Guo101,Guo102}, Liu and Yu \cite{Liu-Yu}, Liu, Yang and Yu \cite{Liu-20041}, Yang and Zhao \cite{Yang-Zhao},   Alexander, Morimoto, Ukai, Xu
and Yang \cite{Alexandre-2011, Alexandre-2012}, Gressman and Strain \cite{Gressman1}. These solutions exist globally in time and enjoy uniqueness and continuous dependence in appropriate functional spaces. In fact, there is a fundamental problem in the mathematical analysis of the Boltzmann equation. \textit{Is that possible to find a functional space with minimal regularity in which the global well-posedness of solutions holds?} Several years ago, Alexander, Morimoto, Ukai, Xu and Yang \cite{Alexandre-2010,Alexander-2013} discussed the local well-posedness in $L^{\infty}(0,T;L^2(\mathbb{R}^3_{\xi};H^{s}(\mathbb{R}_{x}^3)))$ with $s>3/2$ that ensures the $L^\infty$ bound (in the spatial variable $x$) by Sobolev's embedding. Inspired by this consideration, Duan, Liu and the third author \cite{Duan-2016} introduced the Chemin-Lerner type space $\widetilde{L}^{\infty}(0,T;\widetilde{L}^2_{\xi}(B^{3/2}_{2,1}))$ in three dimensions and established the global well-posedness of solutions to the Cauchy problem \eqref{1dlg-E1.1}-\eqref{1dlg-E1.2} for angular cutoff hard potentials. The motivation for the adaptation of Besov spaces with spatially critical regularity originates from the study of the fluid dynamical equations, for example, the incompressible Navier-Stokes equations \cite{Chemin-1995}, the compressible Navier-Stokes equations \cite{Danchin-2000} and so on. Later, Morimoto and Sakamoto \cite{Morimoto-2016} 
developed that idea in \cite{Duan-2016} and investigated the global existence for the Boltzmann equation without angular cutoff. Recently, Duan and Sakamoto \cite{Duan-2018} obtained the global existence and time-decay rates of solutions for the Boltzmann equation in velocity-weighted Chemin-Lerner type spaces. Actually, the theory of well-posedness for Navier-Stokes equations in critical spaces is now very mature, and it is our hope that some analysis tools in the study of Navier-Stokes equations will turn out to be applicable to the corresponding problems for the Boltzmann equation. If the theory can be made precise enough, it may turn out to be robust for addressing problems such as the well-posedness or the time asymptotics of solutions, since the Boltzmann equation shares the similar dissipation structure at the kinetic level with fluid equations. 

In the present paper, we will prove a new global existence result for the Cauchy problem of \eqref{1dlg-E1.1}-\eqref{1dlg-E1.2}, inspired by the following chain of embeddings that holds for $s>3/2$:
\begin{equation*}
\begin{aligned}
&H^{s}(\mathbb{R}^3)\hookrightarrow B_{2,1}^{3/2}(\mathbb{R}^3)\hookrightarrow \dot{B}^{1/2}_{2,1}(\mathbb{R}^3)\cap\dot{B}^{3/2}_{2,1}(\mathbb{R}^3).
\end{aligned}
\end{equation*}
Note that the latter space is topologically weaker than the inhomogeneous space $B_{2,1}^{3/2}(\mathbb{R}^3)$. Besides providing a new approach to proving the global well-posedness for \eqref{1dlg-E1.1}-\eqref{1dlg-E1.2}, we will also deduce the optimal algebraic time-decay rates of solutions converging to the Maxwellian equilibrium. To the best of our knowledge, there have been extensive investigations on the rate of convergence to equilibrium for the Boltzmann equation. The spectral analysis and fluid dynamical limits of the linearized Boltzmann were first investigated by Ellis and Pinsky in \cite{Ellis-1975}. In the context of perturbations, the first decay result was given by Ukai \cite{Ukai-1974}, where the spectral analysis was employed to obtain the exponential rates for \eqref{1dlg-E1.1} with hard potentials on the torus. Desvillettes and Villani \cite{DV1} developed the hypocoercivity method and obtained the first almost exponential rate of convergence for large amplitude solutions to \eqref{1dlg-E1.1} on the torus with cut-off soft potentials. By the direct interpolation of smooth energy estimates, Guo and Strain \cite{Strain-2006} presented a simpler alternative proof of decay results in \cite{DV1} for soft potentials as well as the Coulombic interaction. The study of optimal convergence rates in the whole space has shown to be harder than the case of spatially periodic domains due to the additional dispersion effects of the transport term in \eqref{1dlg-E1.1}. Under the $L^1$-assumption, Kawashima, Matsumura and Nishida \cite{KMN79} proved that the solutions to the Boltzmann equation and the incompressible Navier-Stokes equations for small initial perturbations were asymptotically equivalent to that of the compressible Navier-Stokes equations at the rate of $\mathcal{O}(t^{-5/4})$, as $t\rightarrow\infty$. See also earlier works by Ukai and Asano \cite{Ukai-1986, Ukai-1982}, for the spectral analysis of the Boltzmann equation with or without angular cut-off soft potential. In particular, we mention Kawashima's method of thirteen moments \cite{Kawashima1}, where the author  employed the compensating function to get optimal linear decay analysis rather than the spectral analysis. Subsequently, that method was used to investigate the long-time behavior of the Cauchy problem \eqref{1dlg-E1.1}-\eqref{1dlg-E1.2}  in \cite{Glassey-1996}. Inspired by the hypocoercivity approach associated with Kawashima's argument for hyperbolic-parabolic systems \cite{SK}, Duan \cite{Duan-2011} constructed Lyapunov functionals to exploit the dissipation mechanism in the degenerate parts of the solution. By employing the spectral analysis as in \cite{Ellis-1975,UkaiYangBook}, Zhong \cite{Zhong-2014} considered the optimal time-convergence rates of the global solution to \eqref{1dlg-E1.1}-\eqref{1dlg-E1.2}, where the solution tends to the  Maxwellian at the optimal rate of $\mathcal{O}(t^{-3/4})$ and the microscopic part decays at the optimal rate of $\mathcal{O}(t^{-5/4})$. 
Strain \cite{Strain-2012} established the same rates for the soft potential Boltzmann equation in the whole space, with or without the angular cut-off assumption. However, it is often the case that propagating bounds on the  $L^1$-norm are difficult along the time evolution. Guo and Wang \cite{Guo-2012} developed an energy method for proving the optimal time-convergence rates of solutions to a class of dissipative equations (including the Boltzmann equation) in the whole space. The negative Sobolev spaces $H^{-s}(\mathbb{R}^3)$ for some $s\in[0,3/2)$ are shown to be preserved along the time evolution and enhance the rates. Owing to the fact that the space $L^1(\mathbb{R}^3)$ cannot be embedded in $H^{-s}(\mathbb{R}^3)$, Sohinger and Strain \cite{Vedran-2014} introduced the homogeneous Besov
spaces $\dot{B}^{-\sigma}_{2,\infty}(\mathbb{R}^3)$ for $\sigma=3(1/p-1/2)(1\leq p<2)$. It should be emphasized that the Besov space with negative order can be viewed as a physical choice due to the embeddings:
$$
L^1(\mathbb{R}^3)\hookrightarrow \dot{B}^{0}_{1,\infty}(\mathbb{R}^3)\hookrightarrow\dot{B}^{-3/2}_{2,\infty}(\mathbb{R}^3)\quad\text{and}\quad \dot{H}^{-3/2}(\mathbb{R}^3)\hookrightarrow \dot{B}^{-3/2}_{2,\infty}(\mathbb{R}^3).
$$ 
Although the latter space $\dot{B}^{-3/2}_{2,\infty}(\mathbb{R}^3)$ does not embed into $\dot{B}^{0}_{2,2}(\mathbb{R}^3)\sim L^2(\mathbb{R}^{3})$, any function belonging to this space and concentrated in low frequencies does also belong to $L^2(\mathbb{R}^3)$. This actually indicates that the $L^1$-regularity imposed is stronger than the $L^2$-regularity in low frequencies. Based on the simple observation, Danchin and the third author \cite{Danchin-2017, Xu-2019} claimed a more general low-frequency assumption $\dot{B}^{\sigma_1}_{2,\infty}(\mathbb{R}^3)$ with $-3/2\leq\sigma_1<1/2$ for viscous compressible fluids in the $L^p$ critical Besov framework, in order to get the time-decay estimates of $L^1$-$L^2$ type. Furthermore, inspired by the idea of energy methods in \cite{Guo-2012,Strain-2006}, Xin and the third author \cite{Xin-2021} removed the smallness condition on the low-frequency part of initial perturbations. As an interesting extension, we develop the energy argument (see \cite{Xin-2021}) to the kinetic level and establish the optimal time-convergence rates of the global solution to the Boltzmann equation \eqref{1dlg-E1.1}-\eqref{1dlg-E1.2}, provided that the additional low-frequency part of the initial perturbation is bounded in $\widetilde{L}^2_{\xi}(\dot{B}_{2,\infty}^{\sigma_{0}})$ with $-3/2\leq\sigma_{0}<1/2$. 

\section{Main results}
In the paper, we study the solution to the Cauchy problem \eqref{1dlg-E1.1}-\eqref{1dlg-E1.2} of the Boltzmann equation around the global Maxwellian 
$$\mu=\mu(\xi)=(2\pi)^{-3/2}e^{-|\xi|^{2}/2},
$$
which has been normalized to have zero bulk velocity and unit density and temperature. For that purpose, we define the perturbation $f(t,x,\xi)$ by
$$
F(t,x,\xi)=\mu(\xi)+\mu(\xi)^{1/2}f.
$$ Then, the Cauchy problem (\ref{1dlg-E1.1})-(\ref{1dlg-E1.2}) can be reformulated as
\begin{equation}\label{1dlg-E1.4}
 \left\{
\begin{aligned}
&\partial_{t}f+\xi\cdot\nabla_{x}f+Lf=\Gamma(f,f),\\
&f(0,x,\xi)=f_{0}(x,\xi)
\end{aligned}
 \right.
\end{equation}
with $f_{0}(x,\xi)\triangleq \mu^{-1/2}(F_{0}(x,\xi)-\mu)$. Here, the linearized term $Lf$ and nonlinear collision term $\Gamma(f,f)$, respectively, are defined by
\begin{equation}\label{1dlg-E1.5}
Lf\triangleq -\mu^{-1/2}\Big(\mathcal{Q}(\mu,\mu^{1/2}f)+\mathcal{Q}(\mu^{1/2}f,\mu)\Big)
\end{equation}
and
\begin{equation}\label{1dlg-E1.51}
\Gamma(f,f)\triangleq \mu^{-1/2}\mathcal{Q}(\mu^{1/2}f,\mu^{1/2}f).
\end{equation}
Note that the linearized collision operator $L$ is nonnegative and its null space is defined by
$$\mathcal{N}\triangleq\rm{span}\Big\{\sqrt{\mu},~\xi_{1}\sqrt{\mu},~\xi_{2}\sqrt{\mu},~\xi_{3}\sqrt{\mu},~|\xi|^{2}\sqrt{\mu}\Big\}.$$
For later use, we denote by $\mathbf{P}$ the orthogonal projection from the $L^{2}_{\xi}$ to the null space $\mathcal{N}$ and write
\begin{equation}\label{1dlg-E1.6}
\mathbf{P}f(t,x,\xi)\triangleq \Big\{a(t,x)+b(t,x)\cdot\xi+c(t,x)(|\xi|^{2}-3)\Big\}\mu(\xi)^{1/2},
\end{equation}
where the moments $a$, $b$, $c$ are given by
\begin{equation}\label{1dlg-E1.7}
 \left\{
\begin{aligned}
&a(t,x)\triangleq\int_{\mathbb{R}^3_{\xi}}\sqrt{\mu} f(t,x,\xi)\,d\xi,\\
&b(t,x)\triangleq \int_{\mathbb{R}^3_{\xi}}\xi  \sqrt{\mu} f(t,x,\xi)\,d\xi,\\
&c(t,x)\triangleq \frac{1}{6}\int_{\mathbb{R}^3_{\xi}}(|\xi|^{2}-3) \sqrt{\mu}  f(t,x,\xi)\,d\xi.
\end{aligned}
\right.
\end{equation}
Consequently, we are led to the following macro-micro decomposition
$$f(t,x,\xi)=\mathbf{P}f(t,x,\xi)+\{\mathbf{I-P}\}f(t,x,\xi),$$
where $\mathbf{P}f$ and $\{\mathbf{I-P}\}f$ denote the macroscopic and microscopic parts of $f$, respectively. 

In this position, we state the global existence of solutions to the Cauchy problem $\rm(\ref{1dlg-E1.4})$ in the spatially homogeneous Besov space. We denote the energy functional $\mathcal{E}_{t}(f)$ and the energy dissipation functional $\mathcal{D}_{t}(f)$, respectively, as 
\begin{equation}\label{1dlg-E1.16}
\mathcal{E}_{t}(f)=\|f\|^{\ell}_{\widetilde{L}^{\infty}_{t}\widetilde{L}^{2}_{\xi}(\dot{B}^{1/2}_{2,1})}+\|f\|^{h}_{\widetilde{L}^{\infty}_{t}\widetilde{L}^{2}_{\xi}(\dot{B}^{3/2}_{2,1})},
\end{equation}
and
\begin{equation}\label{1dlg-E1.17}
\begin{aligned}
\mathcal{D}_{t}(f)&=\|\mathbf{P}f\|^{\ell}_{\widetilde{L}^{2}_{t}\widetilde{L}^{2}_{\xi}(\dot{B}^{3/2}_{2,1})}+\|\{\mathbf{I-P}\}f\|^{\ell}_{\widetilde{L}^{2}_{t}\widetilde{L}^{2}_{\xi,\nu}(\dot{B}^{1/2}_{2,1})}+\|f\|^{h}_{\widetilde{L}^{2}_{t}\widetilde{L}^{2}_{\xi,\nu}(\dot{B}^{3/2}_{2,1})}.\\
\end{aligned}
\end{equation}
The reader is referred to Section \ref{1dlg-S2}
for those definitions
of Besov norms restricted in the low-frequency regime or the high-frequency regime. 

\begin{theorem}\label{1dlg-T1.2}
There exists a positive constant $\varepsilon_0>0$ such that if the initial data $f_0\in\widetilde{L}^{2}_{\xi}(\dot{B}^{1/2}_{2,1})\cap\widetilde{L}^{2}_{\xi}(\dot{B}^{3/2}_{2,1})$ and
\begin{equation}\label{1dlg-E1.18}
\|f_{0}\|^{\ell}_{\widetilde{L}^{2}_{\xi}(\dot{B}^{1/2}_{2,1})}+\|f_{0}\|^{h}_{\widetilde{L}^{2}_{\xi}(\dot{B}^{3/2}_{2,1})}\leq\varepsilon_0,
\end{equation}
then there exists a unique global strong solution $f(t,x,\xi)$ to the Cauchy problem $\rm(\ref{1dlg-E1.4})$, satisfying
\begin{equation}\label{1dlg-E1.19}
\mathcal{E}_{t}(f)+\mathcal{D}_{t}(f)\leq C\Big(\|f_{0}\|^{\ell}_{\widetilde{L}^{2}_{\xi}(\dot{B}^{1/2}_{2,1})}+\|f_{0}\|^{h}_{\widetilde{L}^{2}_{\xi}(\dot{B}^{3/2}_{2,1})}\Big)
\end{equation}
for any $t\in\mathbb{R}_{+}$, where $C>0$ is a generic constant independent of time. Moreover, if $F_{0}(x,\xi)=\mu+\mu^{1/2}f_{0}(x,\xi)\geq0$, then $F(t,x,\xi)=\mu+\mu^{1/2}f(t,x,\xi)\geq0$.
\end{theorem}

\begin{remark} As shown by \cite{Duan-2016, Morimoto-2016}, the Boltzmann equation admits a unique global strong solution near Maxwellian in the inhomogeneous Chemin-Lerner space $\widetilde{L}^{\infty}(\mathbb{R}_{+};\widetilde{L}^{2}_{\xi}(B^{3/2}_{2,1}))$. Theorem \ref{1dlg-T1.2} indicates that the global-in-time existence remains true in the homogeneous Chemin-Lerner space $\widetilde{L}^{\infty}(\mathbb{R}_{+};\widetilde{L}^{2}_{\xi}(\dot{B}^{1/2}_{2,1})\cap\widetilde{L}^{2}_{\xi}(\dot{B}^{3/2}_{2,1}))$. In fact, the $L^2$ energy estimates are not necessary in our approach. We develop the hypocoercivity theory {\rm(}see \cite{DV0,DV1,Villani2}{\rm)} and construct  Lyapunov functionals with different dissipation rates in low and high frequencies. It is observed from $\rm{(\ref{1dlg-E1.19})}$ that the Boltzmann equation has a parabolic smoothing effect on $\mathbf{P}f$ and a damping effect on $\{\mathbf{I-P}\}f$ in low frequencies, whereas the overall solution $f$ behaves as damping in high frequencies.
\end{remark}

\begin{remark}
An interesting question is how to justify 
the hydrodynamic limit of the Boltzmann equation in the homogeneous critical setting, which will be investigated in the coming work. Indeed, the Besov space $\dot{B}^{1/2}_{2,1}$ is the scaling critical one for the global well-posedness of 3D incompressible Navier-Stokes equations {\rm{(}}see \cite{Chemin-1995}{\rm{)}}. 
\end{remark}
 
Furthermore,  we have time decay rates of $L^1$-$L^2$ type, if the additional low-frequency regularity $\widetilde{L}^2_{\xi}(\dot{B}_{2,\infty}^{\sigma_{0}})$ is
imposed.

\begin{theorem}\label{1dlg-T1.3}
Let $f(t,x,\xi)$ be the corresponding global solution to $\rm(\ref{1dlg-E1.4})$ constructed in Theorem \ref{1dlg-T1.2}. If 
the low-frequency part of $f_{0}$ additionally fulfills $f_{0}^{\ell}\in \widetilde{L}^{2}_{\xi}(\dot{B}^{\sigma_{0}}_{2,\infty})$ with $\sigma_{0}\in\big[-3/2,1/2\big) $ such that
\begin{equation}\label{1dlg-E1.20}
\|f_{0}\|^\ell_{\widetilde{L}^{2}_{\xi}(\dot{B}^{\sigma_{0}}_{2,\infty})}<\infty,
\end{equation}
then it holds for any $t\in \mathbb{R}^+$ that
\begin{align}
\|f(t)\|_{\widetilde{L}^{2}_{\xi}(\dot{B}^{\sigma}_{2,1})}&\leq C\delta_{0}(1+t)^{-\frac{1}{2}(\sigma-\sigma_{0})},\quad~~  \sigma\in\big(\sigma_{0},3/2\big],\label{1dlg-E1.21}\\
\|\{\mathbf{I-P}\}f(t)\|_{\widetilde{L}^{2}_{\xi}(\dot{B}^{\sigma}_{2,1})}&\leq C\delta_{0}(1+t)^{-\frac{1}{2}(\sigma-\sigma_{0}+1)},\quad  \sigma\in\big(\sigma_{0},1/2\big],\label{1dlg-E1.22}
\end{align}
where $C>0$ is a generic constant independent of time and $$\delta_{0}\triangleq \|f_{0}\|^{\ell}_{\widetilde{L}^{2}_{\xi}(\dot{B}^{\sigma_{0}}_{2,\infty})}+\|f_{0}\|^{h}_{\widetilde{L}^{2}_{\xi}(\dot{B}^{3/2}_{2,1})}.$$
\end{theorem}

\begin{remark}
Notice that
$$ 
\|f_{0}\|^{\ell}_{\widetilde{L}^{2}_{\xi}(\dot{B}^{1/2}_{2,1})} \leq C\|f_{0}\|^{\ell}_{\widetilde{L}^{2}_{\xi}(\dot{B}^{\sigma_{0}}_{2,\infty})}.
$$
The decay estimates \eqref{1dlg-E1.21}-\eqref{1dlg-E1.22} are available provided that the regularity on low frequencies is suitably strengthened in Theorem \ref{1dlg-T1.2}. Moreover, the smallness of the norm $\|f_{0}\|^{\ell}_{\widetilde{L}^{2}_{\xi}(\dot{B}^{\sigma_{0}}_{2,\infty})}$ is not required.
In particular, if taking $\sigma=3/2$ and $\sigma_{0}=-3/2$, then the solution decays in  $L^2_{\xi}L^{\infty}_x$ at the rate of $\mathcal{O}(t^{-3/2})$. If 
taking $\sigma=0$ and $\sigma_{0}=-3/2$, then
the solution
 decays in $L^2_{\xi}L^2_{x}$ at the rate of $\mathcal{O}(t^{-3/4})$ and the decay of the microscopic part is at the enhanced rate of $\mathcal{O}(t^{-5/4})$.
 \end{remark}

In what follows, we explain technical difficulties and strategies to prove Theorems \ref{1dlg-T1.2}-\ref{1dlg-T1.3}. Regarding the global existence, the major difficulty lies in deriving the {\emph{a priori}} estimates without the usual $L^2$ energy bounds. Note that the linearized operator $L$ is degenerate on a nontrivial five-dimensional space. However, the interaction between the transport term $\xi\cdot\nabla f$ and the degenerate dissipative term $Lf$ can produce the complete dissipation and lead to the convergence to equilibrium. The systematic study of hypocoercivity for the Boltzmann equation was presented by Villani \cite{Villani2}. 
As explained in \cite[Remark 17]{Villani2} and \cite{BZ}, the hypocoercivity is closely linked with the Shizuta-Kawashima condition \cite{SK} that is extensively applied to study the hyperbolic-parabolic composite systems or partially hyperbolic systems. In the spirit of hypocoercivity for the Boltzmann equation, we construct those spectral-localized Lyapunov functionals with different dissipation rates in low and high frequencies (see \eqref{1dlg-E5.11}-\eqref{1dlg-E5.110} and  \eqref{1dlg-E5.29}-\eqref{1dlg-E5.2900}), which allows us to establish the global {\emph{a priori}} estimates in the homogeneous critical Chemin-Lerner spaces. On the other hand, we establish nonstandard trilinear estimates in spatially homogeneous Besov spaces (see Lemma \ref{1dlg-L2.7}) to control the nonlinear collision terms. 

In the proof of Theorem \ref{1dlg-T1.3}, we develop the recent energy argument of Lyapunov type \cite{Xin-2021} that can be
adapted to suit the Boltzmann equation. However, the generalization is highly nontrivial compared to the case of fluid dynamical equations, not only because of the additional velocity variable $\xi$, but also due to the bilinear nonlocal collision operator. In the critical framework of $\widetilde{L}^{2}_{\xi}(\dot{B}^{3/2}_{2,1})$, we develop the so-called weighted Lyapunov energy approach in contrast to \cite{Liu-2021}. Actually, 
we derive the following time-weighted inequality:
\begin{equation*}
\begin{aligned}
&\|(1+\tau)^{M}f\|_{\widetilde{L}^{\infty}_{t}\widetilde{L}^{2}_{\xi}(\dot{B}^{3/2}_{2,1})}+\|(1+\tau)^{M}\mathbf{P}f\|_{\widetilde{L}^{2}_{t}\widetilde{L}^{2}_{\xi}(\dot{B}^{5/2}_{2,1})}^{\ell}\\
&\quad\quad+\|(1+\tau)^{M}\{\mathbf{I-P}\}f\|_{\widetilde{L}^{2}_{t}\widetilde{L}^{2}_{\xi,\nu}(\dot{B}^{3/2}_{2,1})}^{\ell}+\|(1+\tau)^{M}f\|^{h}_{\widetilde{L}^{2}_{t}\widetilde{L}^{2}_{\xi,\nu}(\dot{B}^{3/2}_{2,1})}\\
&\quad\lesssim \Big(\|f^{\ell}\|_{\widetilde{L}^2_{\xi}(\dot{B}^{\sigma_0}_{2,\infty})}+\|f^h\|_{\widetilde{L}^2_{\xi}(\dot{B}^{3/2}_{2,1})}\Big)(1+t)^{M-\frac{1}{2}(\frac{3}{2}-\sigma_{0})},
\end{aligned}
\end{equation*}
where $M$ is chosen large enough if necessary. This leads to the desired decay estimate \eqref{1dlg-E1.21} with the aid of interpolation tricks. Here, the crucial part of the decay proof lies in the nonlinear evolution of the $\widetilde{L}^{2}_{\xi}(\dot{B}^{\sigma_{0}}_{2,\infty})$-norm in low frequencies, which depends on different trilinear estimates (see Lemma \ref{1dlg-L2.8}). It still remains a challenging problem to extend the current approach to more singular situations, such as the soft potential and the angular non-cutoff cases.  
\medbreak

\vspace{1mm}

Finally, we would like to mention recent works that showed the close connection between the Boltzmann equation and the nonlinear Schr\"odinger equation. 
Chen, Denlinger and Pavlovi\'c \cite{CMP-2019} employed the dispersive property of the linear Schr\"odinger equation and proved the local well-posedness of the Boltzmann equation near vacuum in weighted Sobolev spaces $L^2(\mathbb{R}^{d}_{\xi};H^{s}(\mathbb{R}^{d}_{x}))$ with $d\geq2$ and $s>(d-1)/2$ for both Maxwellian molecules and hard potentials.  The authors in \cite{Chen-2021} also established the global well-posedness for the Boltzmann equation with a constant collision
kernel in two space dimensions, only assuming the scaling-critical $L^2(\mathbb{R}^{2}_{\xi};L^{2}(\mathbb{R}^{2}_{x}))$ norm 
to be sufficiently small. Chen, Shen and Zhang \cite{chenshenzhang2}
proved the 3D sharp global existence result for the Boltzmann equation with Maxwellian molecules and soft potentials under the scaling-critical weighted $L^2(\mathbb{R}^{3}_{\xi};H^{1/2}(\mathbb{R}^{3}_{x}))$ smallness condition. 

\vspace{2mm}

The rest of the paper is organized as follows. In Section \ref{1dlg-S2}, we briefly recall the Littlewood-Paley decomposition and Besov spaces and provide several lemmas which will be used in the subsequent analysis. In Section \ref{1dlg-S4}, we deduce the pointwise estimates for the linear Boltzmann equation, which leads to the sharp time-decay properties. As technical preparation, in Section \ref{1dlg-SA}, we develop new trilinear estimates for the collision operator $\Gamma(f,g)$ by using Bony's paraproduct decomposition. In Section \ref{1dlg-S5}, we establish global {\emph{a priori}} estimates, which lead to the proof of Theorem \ref{1dlg-T1.2}. Section \ref{1dlg-S6} is devoted to proving the time-decay estimates in Theorem \ref{1dlg-T1.3}. 
\vspace{2mm}

\section{Preliminary}\label{1dlg-S2}
Throughout the paper, $C$ stands for a generic constant independent of time $t$. For brevity, we write $A\lesssim B$ instead of $A\leq CB$. $A\sim B$ implies $A\lesssim B$ and $B\lesssim A$ simultaneously. We use $(~\cdot ~)_{x}$ (resp. $(~\cdot ~)_{\xi}$) to denote the standard $L^2$ inner product in $\mathbb{R}^3_{x}$ (resp. $\mathbb{R}^3_{\xi}$) with its corresponding $L^2$-norm $\|\cdot\|_{L^2_{x}}$ (resp. $\|\cdot \|_{L^2_{\xi}}$). Similarly, $( ~\cdot ~)_{\xi,x}$ denotes the $L^2$ inner product in $\mathbb{R}^3_{x}\times\mathbb{R}^3_{\xi}$ with its $L^2$-norm $\|\cdot\|_{L^2_{\xi}L^2_{x}}$. We use $L^{\varrho}_{\xi}L^{p}_{x}=L^{\varrho}(\mathbb{R}^3_{\xi};L^{p}(\mathbb{R}^3_{x}))$ as the mixed velocity-space Lebesgue space endowed with the norm $\|\cdot\|_{L^{\varrho}_{\xi}L^p_{x}}$. For $X$ a Banach space, the notation $L^p(0, T; X)$ or $L^p_T(X) (p\in[1, \infty]$, $T>0)$ designates the set of measurable functions $f: [0, T]\to X$ with $t\mapsto\|f(t)\|_X$ in $L^p(0, T)$, endowed with the norm $\|\cdot\|_{L^p_{T}(X)} \triangleq \|\|\cdot\|_X\|_{L^p(0, T)}$. We denote $\mathcal{S}(\mathbb{R}^{3})$ as the Schwartz function space and $\mathcal{S}^{\prime}(\mathbb{R}^{3})$ as its dual space. Given a Schwartz function $u(t,x,\xi)\in\mathcal{S}(\mathbb{R}^{3})$, the Fourier transform $\widehat{u}(t,k,\xi)=\mathcal{F}(u)(t,k,\xi)$ with respect to $x$ is given by
$$\widehat{u}(t,k,\xi)=\mathcal{F}(u)(t,k,\xi)\triangleq \int_{\mathbb{R}^{3}}e^{-ix\cdot k}u(t,x,\xi)\,dx,$$
and  $\mathcal{F}^{-1}(u)(t,k,\xi)$ denotes the inverse Fourier transform. The Fourier transform and its inverse transform of a tempered function $u(t,x,\xi)\in\mathcal{S}^{\prime}(\mathbb{R}^{3})$ are defined by the dual argument in the standard way. 

We recall the Littlewood-Paley decomposition theory and functional spaces, such as Besov spaces and Chemin-Lerner spaces. 
The reader can refer to Chapter 2 in \cite{Bahouri-2011} for more details. Choose a smooth radial non-increasing function $\chi(k)$ compactly supported in the ball $B(0,4/3)$ and satisfying $\chi(k)\equiv1$ in $B(0,3/4)$. Then $\varphi(k)\triangleq \chi(k/2)-\chi(k)$ satisfies
$$\sum_{q\in\mathbb{Z}}\varphi(2^{-q}k)=1,~~{\rm Supp}~\varphi\subset\big\{k\in\mathbb{R}^{3}\mid3/4\leq|k|\leq8/3\big\}.$$
For any $q\in\mathbb{Z}$, define the homogeneous dyadic blocks $\dot{\Delta}_{q} (q\in\mathbb{Z})$ by
$$
\dot{\Delta}_{q}u\triangleq \varphi(2^{-q}D)u=\mathcal{F}^{-1}(\varphi(2^{-q}\cdot)\mathcal{F}u)=2^{3q}h(2^{q}\cdot)\ast u\quad \mbox{with}\quad h=\mathcal{F}^{-1}\varphi.
$$
Let $\mathcal{P}$ be the class of all polynomials on $\mathbb{R}^{3}$ and $\mathcal{S}_{h}^{\prime}(\mathbb{R}^{3})=\mathcal{S}^{\prime}/\mathcal{P}(\mathbb{R}^{3})$ stands for the tempered distribution on $\mathbb{R}^{3}$ modulo polynomials. For $\forall u\in\mathcal{S}_{h}^{\prime}(\mathbb{R}^{3})$, we have the Littlewood-Paley decomposition
$$u=\sum_{q\in\mathbb{Z}}\dot{\Delta}_{q}u\quad \mbox{in}\quad \mathcal{S}_{h}^{\prime}(\mathbb{R}^{3})\quad \text{with}\quad  \dot{\Delta}_{j}\dot{\Delta}_{q}u\equiv0\quad \mbox{if}\quad |j-q|\geq2.$$

With the help of those dyadic blocks, we present the definition of homogeneous Besov spaces as follows.
\begin{defn}\label{1dlg-D2.1}
For $s\in\mathbb{R}$ and $1\leq p,r\leq\infty$, the homogeneous Besov space $\dot{B}^{s}_{p,r}$ is defined by
$$\dot{B}^{s}_{p,r}\triangleq \Big\{u\in\mathcal{S}_{h}^{\prime}(\mathbb{R}^{3})~:~\|u\|_{\dot{B}^{s}_{p,r}}\triangleq \|\{2^{qs}\|\dot{\Delta}_{q}u\|_{L^{p}}\}_{q\in\mathbb{Z}}\|_{l^{r}}<+\infty\Big\}.$$
\end{defn}
Next, we present a class of mixed
space-velocity Besov spaces and mixed
space-velocity-time Besov spaces, that is, Chemin-Lerner type spaces, which were initiated by Chemin and Lerner \cite{Chemin-1995}.
\begin{defn}\label{1dlg-D2.00001}
For $s\in\mathbb{R}$ and $1\leq\varrho, p, r\leq\infty$, the homogeneous Chemin-Lerner space $\widetilde{L}^{\varrho}_{\xi}(\dot{B}^{s}_{p,r})$ is defined by
\begin{equation}\nonumber
\widetilde{L}^{\varrho}_{\xi}(\dot{B}^{s}_{p,r})\triangleq \Big\{u\in L^{\varrho}(\mathbb{R}^{3}_{\xi};\mathcal{S}_{h}^{\prime}(\mathbb{R}^{3}_{x}))~:~\|u\|_{\widetilde{L}^{\varrho}_{\xi}(\dot{B}^{s}_{p,r})}\triangleq 
    \Big\|\{2^{qs}\|\dot{\Delta}_{q}u\|_{L^{\varrho}_{\xi}L^{p}_{x}}\}_{q\in\mathbb{Z}}\Big\|_{l^{r}}<\infty \Big\}.
\end{equation}
Moreover, in order to characterize the Boltzmann dissipation rate, we also introduce the following velocity-weighted norm
\begin{equation}\nonumber
\|u\|_{\widetilde{L}^{\varrho_{2}}_{\xi,\nu}(\dot{B}^{s}_{p,r})}\triangleq \|\sqrt{\nu(\xi)}u\|_{\widetilde{L}^{\varrho_{2}}_{\xi}(\dot{B}^{s}_{p,r})},
\end{equation}
where the multiplier $\nu=\nu(\xi)$, called the collision frequency, is given by {\rm(}cf. {\rm\cite{Cercignani-1994,Glassey-1996}}{\rm)}
\begin{align}
\nu(\xi)=\int_{\mathbb{R}^{3}}\int_{\mathbb{S}^{2}}|\xi-\xi_{\ast}|^{\gamma}B_{0}(\theta)\mu(\xi_{\ast})\,d\omega d\xi_{\ast}\sim(1+|\xi|)^{\gamma},\quad\quad 0\leq \gamma\leq 1.\label{nu}
\end{align}
\end{defn}

\begin{defn}
For $T>0, s\in\mathbb{R}$ and $1\leq\varrho_{1}, \varrho_{2} ,p , r\leq\infty$, the homogeneous Chemin-Lerner space $\widetilde{L}^{\varrho_{1}}_{T}\widetilde{L}^{\varrho_{2}}_{\xi}(\dot{B}^{s}_{p,r})$ is defined by
\begin{equation}\nonumber
\widetilde{L}^{\varrho_{1}}_{T}\widetilde{L}^{\varrho_{2}}_{\xi}(\dot{B}^{s}_{p,r})\triangleq \Big\{u\in L^{\varrho_{1}}(0,T;L^{\varrho_2}(\mathbb{R}^{3}_{\xi};S_{h}^{\prime}(\mathbb{R}^{3}_{x})))~:~\|u\|_{\widetilde{L}^{\varrho_{1}}_{T}\widetilde{L}^{\varrho_{2}}_{\xi}(\dot{B}^{s}_{p,r})}<\infty\Big\},
\end{equation}
where 
\begin{equation}\nonumber
    \begin{aligned}
\|u\|_{\widetilde{L}^{\varrho_{1}}_{T}\widetilde{L}^{\varrho_{2}}_{\xi}(\dot{B}^{s}_{p,r})}\triangleq 
    \begin{cases}
    \Big\|\{2^{qs}\|\dot{\Delta}_{q}u\|_{L^{\varrho_{1}}_{T}L^{\varrho_{2}}_{\xi}L^{p}_{x}}\}_{q\in\mathbb{Z}}\Big\|_{l^{r}},
    & \mbox{if ~$1\leq \varrho_1<\infty$},\\
    \bigg\|\{2^{qs}\sup\limits_{t\in[0,T]}\|\dot{\Delta}_{q}u\|_{L^{\varrho_{2}}_{\xi}L^{p}_{x}}\}_{q\in\mathbb{Z}}\bigg\|_{l^{r}},
     & \mbox{if ~$\varrho_1=\infty$}.
    \end{cases}
    \end{aligned}
    \end{equation}
Furthermore, we define
\begin{equation}\nonumber
\|u\|_{\widetilde{L}^{\varrho_{1}}_{T}\widetilde{L}^{\varrho_{2}}_{\xi,\nu}(\dot{B}^{s}_{p,r})}\triangleq \|\sqrt{\nu(\xi)}u\|_{\widetilde{L}^{\varrho_{1}}_{T}\widetilde{L}^{\varrho_{2}}_{\xi}(\dot{B}^{s}_{p,r})}.
\end{equation}
\end{defn}

Thanks to Minkowski's inequality, the Chemin-Lerner type spaces may be linked with the standard mixed spaces $L^{\varrho_{1}}_{T}L^{\varrho_{2}}_{\xi}(\dot{B}^{s}_{p,r})$ in the following way (see \cite{Duan-2016}).

\begin{lemma}
\label{1dlg-R2.1}
Let $1\leq\varrho_{1}, \varrho_{2}, p, r\leq\infty$ and $s\in\mathbb{R}$.

\item{\rm(1)} If $r\geq \max\{\varrho_{1},\varrho_{2}\}$, then
\begin{equation}\nonumber
\|u\|_{\widetilde{L}^{\varrho_{1}}_{T}\widetilde{L}^{\varrho_{2}}_{\xi}(\dot{B}^{s}_{p,r})}\leq\|u\|_{L^{\varrho_{1}}_{T}L^{\varrho_{2}}_{\xi}(\dot{B}^{s}_{p,r})};
\end{equation}
\item{\rm(2)} If $r\leq \min\{\varrho_{1},\varrho_{2}\}$, then
\begin{equation}\nonumber
\|u\|_{\widetilde{L}^{\varrho_{1}}_{T}\widetilde{L}^{\varrho_{2}}_{\xi}(\dot{B}^{s}_{p,r})}\geq\|u\|_{L^{\varrho_{1}}_{T}L^{\varrho_{2}}_{\xi}(\dot{B}^{s}_{p,r})}.
\end{equation}
\end{lemma}
To describe precise dissipation structures in different frequencies, we restrict Besov norms to the low-frequency regime and the high-frequency regime:
\begin{equation}\nonumber
 \left\{
\begin{aligned}
&\|u\|^{\ell}_{\widetilde{L}^{\varrho}_{\xi}(\dot{B}^{s}_{p,r})}\triangleq \|\{\|2^{qs}\dot{\Delta}_{q}u\|_{L^{\varrho}_{\xi}L^{p}_{x}}\}_{q\leq0}\|_{l^{r}},\\
&\|u\|^{h}_{\widetilde{L}^{\varrho}_{\xi}(\dot{B}^{s}_{p,r})}\triangleq \|\{\|2^{qs}\dot{\Delta}_{q}u\|_{L^{\varrho}_{\xi}L^{p}_{x}}\}_{q\geq-1}\|_{l^{r}},\\
&\|u\|^{\ell}_{\widetilde{L}^{\varrho_{1}}_{T}\widetilde{L}^{\varrho_2}_{\xi}(\dot{B}^{s}_{p,r})}\triangleq \|\{\|2^{qs}\dot{\Delta}_{q}u\|_{L^{\varrho_{1}}_{T}L^{\varrho_{2}}_{\xi}L^{p}_{x}}\}_{q\leq0}\|_{l^{r}},\\
&\|u\|^{h}_{\widetilde{L}^{\varrho_{1}}_{T}\widetilde{L}^{\varrho_{2}}_{\xi}(\dot{B}^{s}_{p,r})}\triangleq \|\{\|2^{qs}\dot{\Delta}_{q}u\|_{L^{\varrho_{1}}_{T}L^{\varrho_{2}}_{\xi}L^{p}_{x}}\}_{q\geq-1}\|_{l^{r}}
\end{aligned}
 \right.
\end{equation}
for any $s\in\mathbb{R}$ and $1\leq \varrho, \varrho_{1},\varrho_{2}, p\leq \infty$.

When bounding nonlinear terms, we often decompose $u$ into its low-frequency part $u^\ell$ and high-frequency part $u^h$, which are given by
$$
u^{\ell}\triangleq \sum_{q\leq-1}\dot{\Delta}_{q}u,\quad u^{h}\triangleq u-u^{\ell}=\sum_{q\geq0}\dot{\Delta}_{q}u.
$$
It is easy to check for any $s^{\prime}>0$ that
\begin{equation}\label{1dlg-E2.6}
 \left\{
\begin{aligned}
&\|u^{\ell}\|_{\widetilde{L}^{\varrho}_{\xi}(\dot{B}^{s}_{p,r})}\lesssim\|u\|^{\ell}_{\widetilde{L}^{\varrho}_{\xi}(\dot{B}^{s}_{p,r})}\lesssim\|u\|^{\ell}_{\widetilde{L}^{\varrho}_{\xi}(\dot{B}^{s-s^{\prime}}_{p,r})},\\
&\|u^{h}\|_{\widetilde{L}^{\varrho}_{\xi}(\dot{B}^{s}_{p,r})}\lesssim\|u\|^{h}_{\widetilde{L}^{\varrho}_{\xi}(\dot{B}^{s}_{p,r})}\lesssim\|u\|^{h}_{\widetilde{L}^{\varrho}_{\xi}(\dot{B}^{s+s^{\prime}}_{p,r})},\\
&\|u^{\ell}\|_{\widetilde{L}^{\varrho_{1}}_{T}\widetilde{L}^{\varrho_{2}}_{\xi}(\dot{B}^{s}_{p,r})}\lesssim\|u\|^{\ell}_{\widetilde{L}^{\varrho_{1}}_{T}\widetilde{L}^{\varrho_{2}}_{\xi}(\dot{B}^{s}_{p,r})}\lesssim\|u\|^{\ell}_{\widetilde{L}^{\varrho_{1}}_{T}\widetilde{L}^{\varrho_{2}}_{\xi}(\dot{B}^{s-s^{\prime}}_{p,r})},\\
&\|u^{h}\|_{\widetilde{L}^{\varrho_{1}}_{T}\widetilde{L}^{\varrho_{2}}_{\xi}(\dot{B}^{s}_{p,r})}\lesssim\|u\|^{h}_{\widetilde{L}^{\varrho_{1}}_{T}\widetilde{L}^{\varrho_{2}}_{\xi}(\dot{B}^{s}_{p,r})}\lesssim\|u\|^{h}_{\widetilde{L}^{\varrho_{1}}_{T}\widetilde{L}^{\varrho_{2}}_{\xi}(\dot{B}^{s+s^{\prime}}_{p,r})}.
\end{aligned}
 \right.
\end{equation}

\vspace{2mm}

Next, we recall  basic properties of Besov spaces, which will be used repeatedly in this paper. The first lemma is devoted to Bernstein's inequalities, which in particular imply that $\dot{\Delta}_{q}u$ is smooth in $x$ for every $u$ in any Besov space, so that we perform direct calculations on linear equations after applying the operator $\dot{\Delta}_{q}(q\in \mathbb{Z})$.
\begin{lemma}\label{1dlg-L2.2}{\rm(}Bernstein lemma{\rm)}
Let $0<r<R$, $1\leq p\leq q\leq\infty$, and $k\in\mathbb{N}$. For any $u\in L^{p}$ and $\lambda>0$, it holds that
\begin{align*}
 \left\{
\begin{aligned}
&\mathrm{Supp}~\mathcal{F}(u)\subset\{\xi\in\mathbb{R}^{3}\mid|\xi|\leq\lambda R\}\Rightarrow\|D^{k}u\|_{L^{q}}\leq\lambda^{k+3/p-3/q}\|u\|_{L^{p}},\\
&\mathrm{Supp}~\mathcal{F}(u)\subset\{\xi\in\mathbb{R}^{3}\mid\lambda r\leq|\xi|\leq\lambda R\}\Rightarrow\|D^{k}u\|_{L^{p}}\sim\lambda^{k}\|u\|_{L^{p}}.
\end{aligned}
 \right.
\end{align*}
\end{lemma}

Due to Lemma \ref{1dlg-L2.2}, the space-velocity Besov spaces have the following properties, which can be proved similarly as in \cite[Chapter 2]{Bahouri-2011}.
\begin{lemma}\label{1dlg-L2.3}
The following properties hold:
\begin{itemize}
\item For $1\leq \varrho\leq \infty$ and $1\leq p\leq q\leq\infty$, we have the following chain of continuous embedding:
$$ \widetilde{L}^{\varrho}_{\xi}(\dot{B}^{0}_{p,1})\hookrightarrow L^{\varrho}_{\xi}L^{p}_{x}\hookrightarrow\widetilde{L}^{\varrho}_{\xi}(\dot{B}^{0}_{p,\infty})\hookrightarrow\widetilde{L}^{\varrho}_{\xi}(\dot{B}^{\sigma}_{q,\infty}),\quad \sigma=-3/p +3/q\leq 0;$$
\end{itemize}
\begin{itemize}
\item If $p<\infty$, then $\widetilde{L}^{\varrho}_{\xi}(\dot{B}^{3/p}_{p,1})$ is continuously embedded in the set of functions in $L^\varrho(\mathbb{R}^3_{\xi};C^0(\mathbb{R}^3_x))$ decaying to 0 at infinity;
\item The Besov space $\widetilde{L}^{\varrho}_{\xi}(\dot{B}^{s}_{2,2})$ is consistent with the homogeneous Sobolev space $L^\varrho(\mathbb{R}^3_{\xi};\dot{H}^s(\mathbb{R}^3_x))$;
\end{itemize}
\begin{itemize}
\item For $1\leq \varrho, p,r\leq\infty$ and $s\in\mathbb{R}$ satisfying
$$s<3/p,\quad {\rm{or}}\quad s=3/p\quad {\rm{and}}\quad r=1,
$$
 $\widetilde{L}^{\varrho}_{\xi}(\dot{B}^{s}_{p,r})$ is a Banach space.
\end{itemize}
\end{lemma}

The following real interpolation will be useful in decay estimates. For the proof, one can refer to \cite{Bahouri-2011}[Proposition 2.22](or see  \cite{Liu-2021}).
\begin{lemma}\label{1dlg-L2.4} 
Let $s<\widetilde{s}$, $\theta\in(0,1)$ and $1\leq \varrho, p\leq\infty$. It holds that
$$\|u\|_{\widetilde{L}^{\varrho}_{\xi}(\dot{B}^{\theta s+(1-\theta)\tilde{s}}_{p,1})}\lesssim \frac{1}{\theta(1-\theta)(\tilde{s}-s)}\|u\|^{\theta}_{\widetilde{L}^{\varrho}_{\xi}(\dot{B}^{s}_{p,\infty})}\|u\|^{1-\theta}_{\widetilde{L}^{\varrho}_{\xi}(\dot{B}^{\tilde{s}}_{p,\infty})}.$$
\end{lemma}

To compare the topological relations between 
homogeneous Chemin-Lerner type spaces 
and inhomogeneous spaces, we employ the following lemma, whose proof follows a similar argument to that of \cite[Proposition 7.1]{Xu-2014}.

\begin{lemma}
For any $s\in\mathbb{R}$ and $1\leq\varrho,p,r\leq\infty$, let $\widetilde{L}^{\varrho}_{\xi}(B^{s}_{p,r})$ be the inhomogeneous Besov space defined in {\rm{\cite{Duan-2016}}}. Then, we have
$$
\widetilde{L}^{\varrho}_{\xi}(B^{s}_{p,r})=\widetilde{L}^{\varrho}_{\xi}(\dot{B}^{s}_{p,r})\cap L^\varrho_{\xi}L^p_{x}\hookrightarrow\widetilde{L}^{\varrho}_{\xi}(\dot{B}^{s'}_{p,r}),\quad\quad 0<s'\leq s.
$$
\end{lemma}

\begin{lemma}\label{1dlg-L2.06}
Let $1\leq\varrho,p,r\leq\infty$ and $s_{1}\leq s_{2}$. It holds that
\begin{equation}\label{1dlg-E2.006}
\|\cdot\|_{\widetilde{L}^{\varrho}_{\xi}(\dot{B}^{s_1}_{p,r})}^{\ell}+\|\cdot\|_{\widetilde{L}^{\varrho}_{\xi}(\dot{B}^{s_2}_{p,r})}^{h}\sim \|\cdot\|_{\widetilde{L}^{\varrho}_{\xi}(\dot{B}^{s_1}_{p,r})\cap\widetilde{L}^{\varrho}_{\xi}(\dot{B}^{s_2}_{p,r})}.
\end{equation}
\end{lemma}
\begin{proof}
If $u\in\widetilde{L}^{\varrho}_{\xi}(\dot{B}^{s_1}_{p,r})\cap \widetilde{L}^{\varrho}_{\xi}(\dot{B}^{s_2}_{p,r})$, then we obtain
$$
\|u\|_{\widetilde{L}^{\varrho}_{\xi}(\dot{B}^{s_1}_{p,r})}^{\ell}\lesssim \|u\|_{\widetilde{L}^{\varrho}_{\xi}(\dot{B}^{s_1}_{p,r})}\quad\text{and}\quad  \|u\|_{\widetilde{L}^{\varrho}_{\xi}(\dot{B}^{s_2}_{p,r})}^{h}\lesssim \|u\|_{\widetilde{L}^{\varrho}_{\xi}(\dot{B}^{s_2}_{p,r})}.
$$
Conversely, when $\|u\|_{\widetilde{L}^{\varrho}_{\xi}(\dot{B}^{s_1}_{p,r})}^{\ell}+\|u\|_{\widetilde{L}^{\varrho}_{\xi}(\dot{B}^{s_2}_{p,r})}^{h}<\infty$ with $s_{1}\leq s_{2}$, it follows from \eqref{1dlg-E2.6} that
\begin{align*}
\begin{split}
\|u\|_{\widetilde{L}^{\varrho}_{\xi}(\dot{B}^{s_1}_{p,r})}
&\leq \|u\|_{\widetilde{L}^{\varrho}_{\xi}(\dot{B}^{s_1}_{p,r})}^\ell+\|u\|_{\widetilde{L}^{\varrho}_{\xi}(\dot{B}^{s_1}_{p,r})}^h\lesssim \|u\|_{\widetilde{L}^{\varrho}_{\xi}(\dot{B}^{s_1}_{p,r})}^{\ell}+\|u\|_{\widetilde{L}^{\varrho}_{\xi}(\dot{B}^{s_2}_{p,r})}^{h}.
\end{split}
\end{align*}
Similarly, we have the corresponding estimate for $u\in \widetilde{L}^{\varrho}_{\xi}(\dot{B}^{s_2}_{p,r})$. Hence, \eqref{1dlg-E2.006} is easily proved.
\end{proof}

\vspace{2mm}
We now recall some properties of the linearized collision operator. The operator $L$ given by (\ref{1dlg-E1.5}) can be written as (\cite{Cercignani-1994,Glassey-1996}):
$$
L=\nu-K,
$$
where the collision frequency $\nu$ is given by \eqref{nu} and the integral operator $K=K_{2}-K_{1}$ is defined as
\begin{equation}\nonumber
\begin{aligned}
K_{1}(f)(\xi)&\triangleq \int_{\mathbb{R}^{3}}d\xi_{\ast}\int_{\mathbb{S}^{2}}d\omega\,|\xi-\xi_{\ast}|^{\gamma}B_{0}(\theta)\mu^{1/2}(\xi_{\ast})\mu^{1/2}(\xi)f(\xi_{\ast}),\\
K_{2}(f)(\xi)&\triangleq \int_{\mathbb{R}^{3}}d\xi_{\ast}\int_{\mathbb{S}^{2}}d\omega\,|\xi-\xi_{\ast}|^{\gamma}B_{0}(\theta)\mu^{1/2}(\xi_{\ast}) \Big(\mu^{1/2}(\xi_{\ast}^{\prime})f(\xi^{\prime})+\mu^{1/2}(\xi^{\prime})f(\xi_{\ast}^{\prime})\Big).
\end{aligned}
\end{equation}
It is well-known that $K=K_{2}-K_{1}$ is a compact and self-adjoint operator on $L^{2}_{\xi}$ (cf. \cite{Cercignani-1994}) and has the following property.
\begin{lemma}\label{1dlg-L2.5}
It holds that
\begin{equation}\nonumber
(\dot{\Delta}_{q}Kf,\dot{\Delta}_{q}g)_{\xi,x}\leq C\|\dot{\Delta}_{q}f\|_{L^{2}_{\xi}L^{2}_{x}}\|\dot{\Delta}_{q}g\|_{L^{2}_{\xi}L^{2}_{x}}
\end{equation}
for each $q\in\mathbb{Z}$, where $C$ is a constant independent of $q, f$ and $g$.
\end{lemma}
\begin{lemma}\label{1dlg-L2.6}
There exists a generic constant $\lambda_0>0$ such that it holds that
$$(\dot{\Delta}_{q}Lf,\dot{\Delta}_{q}f)_{\xi,x}\geq\lambda_{0}\|\sqrt{\nu(\xi)}\dot{\Delta}_{q}\{\mathbf{I-P}\}f\|^{2}_{L^{2}_{\xi}L^{2}_{x}}$$
for each $q\in\mathbb{Z}$. Furthermore, for $s\in\mathbb{R}$ and  $T>0$, we have
$$
\sum_{q\in\mathbb{Z}}2^{qs}\bigg(\int_{0}^{T}(\dot{\Delta}_{q}Lf,\dot{\Delta}_{q}f)_{\xi,x}dt\bigg)^{1/2}\geq\sqrt{\lambda_{0}}\|\{\mathbf{I-P}\}f\|_{\widetilde{L}^{2}_{T}\widetilde{L}^{2}_{\xi,\nu}(\dot{B}^{s}_{2,1})}.
$$

\end{lemma}

\vspace{2ex}

\section{Linear analysis}\label{1dlg-S4}
In this section, we consider the Cauchy problem of the linear Boltzmann equation
\begin{equation}\label{1dlg-E4.1}
\left\{
\begin{aligned}
&\partial_{t}f+\xi\cdot\nabla_{x}f+Lf=0,\\
&f|_{t=0}=f_{0}(x,\xi).
\end{aligned}
\right.
\end{equation}
See \cite{Ellis-1975,UkaiYangBook}
for the classical spectral analysis of \eqref{1dlg-E4.1}, where the eigenvalues of solutions to \eqref{1dlg-E4.1}  exhibit the diffusion effect in low frequencies and the spectrum gap in high frequencies. 
Based on the macro-micro decomposition and Kawashima's dissipation argument in \cite{SK}, we can establish the pointwise estimates of solutions to \eqref{1dlg-E4.1}.

\begin{prop}\label{1dlg-P4.1}
For any $t>0$ and $k\in\mathbb{R}^{3}$, the solution to $\rm(\ref{1dlg-E4.1})$ satisfies 
\begin{equation}\label{1dlg-E4.3}
\|\widehat{f}(k,t)\|_{L^{2}_{\xi}}\lesssim \|\widehat{f}_{0}(k)\|_{L^{2}_{\xi}}e^{-\lambda_{1}\min(1,|k|^{2})t} ,
\end{equation}
where $\lambda_{1}>0$ is a uniform constant.
Moreover, there exists a constant $k_{0}$ such that if $|k|\leq k_{0}$, then 
\begin{equation}\label{1dlg-E4.4}
\|\{\mathbf{I-P}\}\widehat{f}(k,t)\|_{L^{2}_{\xi}}\lesssim \|\widehat{f}_{0}(k)\|_{L^{2}_{\xi}} |k| e^{-\frac{\lambda_{1}}{2}|k|^{2}t}+\|\{\mathbf{I-P}\}\widehat{f}_{0}(k)\|_{L^{2}_{\xi}} e^{-\lambda_{0}t}
\end{equation}
 for any $t>0$, where  $\lambda_{0}>0$ is given by Lemma \ref{1dlg-L2.6}.
\end{prop}
\begin{proof}
It is sufficient to prove \eqref{1dlg-E4.4}, since \eqref{1dlg-E4.3} has already been established in \cite{Duan-2011}. 
By applying the Fourier transform (with respect to $x$) to the first equation in $(\ref{1dlg-E4.1})$ and then using the orthogonal projection operator $\{\mathbf{I-P}\}$ to the resulting equality, we obtain
\begin{equation}\label{1dlg-E4.5}
\partial_{t}\{\mathbf{I-P}\}\widehat{f}+ik\cdot\xi\{\mathbf{I-P}\}\widehat{f}+L\{\mathbf{I-P}\}\widehat{f}=-ik\cdot\xi\mathbf{P}\widehat{f}+\mathbf{P}(ik\cdot\xi\widehat{f}).
\end{equation}
Then, we take the Hermitian inner product with respect to $k$ of (\ref{1dlg-E4.5}) and $\{\mathbf{I-P}\}\widehat{f}$, take the real part and then integrate the resulting equation over $\mathbb{R}^3_{\xi}$. Using Lemma \ref{1dlg-L2.6} and the Cauchy-Schwarz inequality, we deduce that
\begin{equation}\label{1dlg-E4.51}
\begin{aligned}
&\frac{1}{2}\frac{d}{dt}\|\{\mathbf{I-P}\}\widehat{f}\|^{2}_{L^{2}_{\xi}}+\lambda_0\|\sqrt{\nu(\xi)}\{\mathbf{I-P}\}\widehat{f}\|^{2}_{L^{2}_{\xi}}\\
&\quad\lesssim  |k|( \|\xi\mathbf{P}\widehat{f}\|_{L^2_{\xi}}+\|\mathbf{P}(\xi \widehat{f})\|_{L^2_{\xi}})\|\{\mathbf{I-P}\}\widehat{f}\|_{L^{2}_{\xi}}.\\
\end{aligned}
\end{equation}
By \eqref{1dlg-E1.6}, we have $\|\xi\mathbf{P}\widehat{f}\|_{L^2_{\xi}}\lesssim |(\widehat{a},\widehat{b},\widehat{c})|\sim \|\mathbf{P}\widehat{f}\|_{L^2_{\xi}}$. On the other hand, $\mathbf{P}(\xi f)$ can be written as
\begin{flalign*}
\begin{aligned}
&\mathbf{P}(\xi f)=\{a_1+b_1\cdot \xi+c_1(|\xi|^2-3)\}\mu^{1/2}
\end{aligned}
\end{flalign*}
with
\begin{flalign*}
\begin{aligned}
&a_1=\int_{\mathbb{R}^3}\mu^{1/2}\xi fd\xi,\quad b_1=\int_{\mathbb{R}^3}\xi\mu^{1/2}\xi fd\xi,\quad c_1=\frac{1}{6}\int_{\mathbb{R}^3}(|\xi|^2-3)\}\mu^{1/2}\xi fd\xi.
\end{aligned}
\end{flalign*}
So one can get
$$\|\mathbf{P}(\xi \widehat{f})\|_{L^2_{\xi}}\lesssim|(\widehat{a_1},\widehat{b_1},\widehat{c_1})|\lesssim \|\widehat{f}\|_{L^2_{\xi}}.
$$
Consequently, applying Gr\"onwall's inequality leads to 
\begin{align}
\|\{\mathbf{I-P}\}\widehat{f}\|_{L^{2}_{\xi}}\leq \|\{\mathbf{I-P}\}\widehat{f}_{0}\|_{L^{2}_{\xi}} e^{-\lambda_{0}t}+|k| \int_{0}^{t}e^{-\lambda_{0}(t-\tau)}\|\widehat{f}\|_{L^{2}_{\xi}}\,d\tau.\label{1dlg-E4.511}
\end{align}
It follows from 
\eqref{1dlg-E4.3} that
\begin{align*}
\int_{0}^{t}e^{-\lambda_{0}(t-\tau)}\|\widehat{f}\|_{L^{2}_{\xi}}\,d\tau\lesssim\int_{0}^{t}e^{-\lambda_{0}(t-\tau)}e^{-\lambda_1\min\{1,|k|^2\}\tau}\|\widehat{f}_{0}\|_{L^{2}_{\xi}}\,d\tau,
\end{align*}
where for $|k|\leq \sqrt{\lambda_0/\lambda_1}$, it holds that
\begin{equation}\label{1dlg-E4.6}
\begin{aligned}
\int_{0}^{t}e^{-\lambda_{0}(t-\tau)}e^{-\lambda_1 |k|^2\tau}\,d\tau&\leq
\int_{0}^{t}e^{-\frac{\lambda_{0}}{2}(t-\tau)}e^{-\frac{\lambda_1}{2}|k|^{2}(t-\tau)}e^{-\frac{\lambda_1}{2}|k|^{2}\tau}\,d\tau\\
&\lesssim e^{-\frac{\lambda_1}{2}|k|^{2}t}\int_{0}^{t}e^{-\frac{\lambda_0}{2}(t-\tau)}\,d\tau\\
&\lesssim e^{-\frac{\lambda_1}{2}|k|^{2}t}.
\end{aligned}
\end{equation}
Therefore, together with \eqref{1dlg-E4.511} and (\ref{1dlg-E4.6}), we obtain (\ref{1dlg-E4.4}) immediately.
\end{proof}

Consequently, we get the following time-decay properties in the framework of Besov spaces. 

\begin{cor}\label{1dlg-C4.2}Assume $f_0\in \widetilde{L}^2_{\xi}(\dot{B}^{\sigma_0}_{2,\infty})\cap \widetilde{L}^2_{\xi}(\dot{B}^{\sigma}_{2,1})$ for $\sigma, \sigma_{0}\in\mathbb{R}$ with $\sigma>\sigma_{0}$. Then for all $t>0$, the solution $f(t,x,\xi)$ to $(\ref{1dlg-E4.1})$ fulfills
\begin{equation}\label{1dlg-E4.8}
\|f(t)\|_{\widetilde{L}^{2}_{\xi}(\dot{B}^{\sigma}_{2,1})}\lesssim \|f_{0}\|_{\widetilde{L}^{2}_{\xi}(\dot{B}^{\sigma_{0}}_{2,\infty})} (1+t)^{-\frac{1}{2}(\sigma-\sigma_{0})}+\|f_{0}\|_{\widetilde{L}^{2}_{\xi}(\dot{B}^{\sigma}_{2,1})} e^{-\lambda_2 t}.
\end{equation}
Moreover, the microscopic part $\{\mathbf{I-P}\}f$
decays faster at the
half rate among all the components of the solution. Precisely,
\begin{equation}\label{1dlg-E4.9}
\|\{\mathbf{I-P}\}f(t)\|_{\widetilde{L}^{2}_{\xi}(\dot{B}^{\sigma}_{2,1})}\lesssim \|f_{0}\|_{\widetilde{L}^{2}_{\xi}(\dot{B}^{\sigma_{0}}_{2,\infty})} (1+t)^{-\frac{1}{2}(\sigma-\sigma_{0}+1)}+\|f_{0}\|_{\widetilde{L}^{2}_{\xi}(\dot{B}^{\sigma}_{2,1})} e^{-\lambda_2 t}.
\end{equation}
Here $\lambda_2>0$ is a generic constant.
\end{cor}
\begin{proof}
The proof of Corollary \ref{1dlg-C4.2} is similar to that in \cite{Xu-2015}, so we give the sketch for brevity. To show \eqref{1dlg-E4.8}, we apply the operator $\dot{\Delta}_{q}$
to \eqref{1dlg-E4.1} and perform similar steps leading to 
\eqref{1dlg-E4.3}-\eqref{1dlg-E4.4}. Hence,  
the optimal information can be obtained, for  instance,  
\begin{equation}\label{1dlg-E4.10}
\|\widehat{\dot{\Delta}_{q}f}(k,t)\|_{L^{2}_{\xi}}\lesssim \|\widehat{\dot{\Delta}_{q}f_{0}}(k)\|_{L^{2}_{\xi}}e^{-\lambda_{1}\min(1,2^{2q})t}.
\end{equation}
This implies that for any $q\geq-1$, 
there exists a uniform constant $\lambda_2>0$ such that 
\begin{equation}\label{1dlg-E4.11}
\|\dot{\Delta}_{q}f\|_{L^{2}_{\xi}L^{2}_{x}}\lesssim \|\dot{\Delta}_{q}f_{0}\|_{L^{2}_{\xi}L^{2}_{x}} e^{-\lambda_{2}t},
\end{equation}
where Plancherel's theorem was used. Then, multiplying (\ref{1dlg-E4.11}) by $2^{q\sigma}$ and summing the resulting inequality over $q\geq-1$, we arrive at
\begin{equation}
\|f\|_{\widetilde{L}^{2}_{\xi}(\dot{B}^{\sigma}_{2,1})}^{h}=\sum_{q\geq -1}2^{q\sigma}\|\dot{\Delta}_{q}f\|_{L^{2}_{\xi}L^{2}_{x}}\lesssim \|f_{0}\|_{\widetilde{L}^{2}_{\xi}(\dot{B}^{\sigma}_{2,1})}^{h} e^{-\lambda_{2}t}.\label{highe}
\end{equation}
On the other hand, for any $q\leq0$, it follows from \eqref{1dlg-E4.10} that 
\begin{equation}\label{localhigh}
\|\dot{\Delta}_{q}f\|_{L^{2}_{\xi}L^{2}_{x}}\lesssim   \|\dot{\Delta}_{q}f_{0}\|_{L^{2}_{\xi}L^{2}_{x}} e^{-\lambda_1 2^{2q}t}.
\end{equation}
This leads to
\begin{flalign*}
\begin{split}
\|f\|^{\ell}_{\widetilde{L}^{2}_{\xi}(\dot{B}^{\sigma}_{2,1})}&\lesssim \|f_{0}\|^{\ell}_{L^{2}_{\xi}(\dot{B}^{\sigma_{0}}_{2,\infty})}t^{-\frac{1}{2}(\sigma-\sigma_{0})}\sum_{q\in\mathbb{Z}}\left( (2^{q}\sqrt{t})^{\sigma-\sigma_{0}}e^{-\lambda_1 2^{2q}t}\right)\\
&\lesssim \|f_{0}\|^{\ell}_{L^{2}_{\xi}(\dot{B}^{\sigma_{0}}_{2,\infty})} t^{-\frac{\sigma-\sigma_0}{2}}.
\end{split}
\end{flalign*}
In addition, as $\sigma>\sigma_{0}$, \eqref{localhigh} directly implies that
\begin{align*}
\|f\|^{\ell}_{\widetilde{L}^{2}_{\xi}(\dot{B}^{\sigma}_{2,1})}\lesssim\|f_{0}\|^{\ell}_{\widetilde{L}^{2}_{\xi}(\dot{B}^{\sigma}_{2,1})}\lesssim\|f_{0}\|^{\ell}_{\widetilde{L}^{2}_{\xi}(\dot{B}^{\sigma_{0}}_{2,\infty})}.
\end{align*}
Consequently, we obtain
\begin{align}\label{lowf}
\|f\|^{\ell}_{\widetilde{L}^{2}_{\xi}(\dot{B}^{\sigma}_{2,1})}\lesssim \|f_{0}\|^{\ell}_{\widetilde{L}^{2}_{\xi}(\dot{B}^{\sigma_{0}}_{2,\infty})}(1+t)^{-\frac{\sigma-\sigma_0}{2}}.
\end{align}
Thus, adding \eqref{highe} and \eqref{lowf} together yields \eqref{1dlg-E4.8} for all $t>0$. For the decay of $\{\mathbf{I-P}\}f$,
there exist $\lambda_3>0$ and $q_0\in \mathbb{Z}$ sufficiently small such that 
\begin{align}
\|\dot{\Delta}_{q}\{\mathbf{I-P}\}f\|_{L^{2}_{\xi}L^{2}_{x}}\lesssim \|\dot{\Delta}_{q}f_{0}\|_{L^{2}_{\xi}L^{2}_{x}} 2^{q}e^{-\lambda_3(2^{q}\sqrt{t})^{2}}+\|\dot{\Delta}_{q}\{\mathbf{I-P}\}f_{0}\|_{L^{2}_{\xi}L^{2}_{x}} e^{-\lambda_{0}t} \nonumber
\end{align}
for $q\leq q_0$. It is not difficult to deduce that, for any $\sigma>\sigma_0$,  
\begin{equation}\label{IPlow}
\begin{aligned}
\sum_{q\leq q_0}2^{q\sigma}\|\dot{\Delta}_{q}\{\mathbf{I-P}\}f\|_{L^{2}_{\xi}L^{2}_{x}}&\lesssim\sum_{q\leq q_0}2^{q\sigma_0}\|\dot{\Delta}_{q}f_{0}\|_{L^{2}_{\xi}L^{2}_{x}}  \, (1+t)^{-\frac{1}{2}(\sigma-\sigma_{0}+1)}\\
&\quad+\sum_{q\leq q_0}2^{q\sigma}\|\dot{\Delta}_{q}\{\mathbf{I-P}\}f_{0}\|_{L^{2}_{\xi}L^{2}_{x}}\, e^{-\lambda_{0}t}.
\end{aligned}
\end{equation}
On the other hand, with the aid of 
\eqref{1dlg-E4.10} with $q\geq q_0+1$, we readily have
\begin{equation}\label{IPhigh}
\begin{aligned}
\sum_{q\geq q_0+1}2^{q\sigma}\|\dot{\Delta}_{q}\{\mathbf{I-P}\}f\|_{L^{2}_{\xi}L^{2}_{x}}&\lesssim  \sum_{q\geq q_0+1}2^{q\sigma}\|\dot{\Delta}_{q}f\|_{L^{2}_{\xi}L^{2}_{x}}\\
&\lesssim e^{-\lambda_1 \min\{1,2^{2(q_0+1)}\}t}  \sum_{q\geq q_0+1}2^{q\sigma} \|\dot{\Delta}_{q}f_0\|_{L^{2}_{\xi}L^{2}_{x}}.
\end{aligned}
\end{equation}
Combining \eqref{IPlow} and \eqref{IPhigh}, we end up with \eqref{1dlg-E4.9}.
\end{proof}

\section{Trilinear estimates in homogeneous Besov spaces}\label{1dlg-SA}

In this section, we establish several trilinear estimates for the Boltzmann collision operator $\Gamma(f,g)$ in the framework of homogeneous Besov spaces, which can be employed to obtain the global existence and large-time behavior of solutions to $\rm(\ref{1dlg-E1.4})$. Recall that the collision operator $\Gamma(f,g)$ can be written as
\begin{equation}\label{1dlg-2.11}
\Gamma(f,g)\triangleq\mu^{-1/2}\mathcal{Q}(\mu^{1/2}f,\mu^{1/2}g)=\Gamma_{{\rm gain}}(f,g)-\Gamma_{{\rm loss}}(f,g),
\end{equation}
where
\begin{equation}\nonumber
\begin{aligned}
&\Gamma_{{\rm gain}}(f,g)=\int_{\mathbb{R}^{3}}\int_{\mathbb{S}^{2}}|\xi-\xi_{\ast}|^{\gamma}B_{0}(\theta)\mu^{1/2}(\xi_{\ast})f(\xi_{\ast}^{\prime})g(\xi^{\prime})\,d\omega \,d\xi_{\ast},\\
&\Gamma_{{\rm loss}}(f,g)=\int_{\mathbb{R}^{3}}\int_{\mathbb{S}^{2}}|\xi-\xi_{\ast}|^{\gamma}B_{0}(\theta)\mu^{1/2}(\xi_{\ast})f(\xi_{\ast})g(\xi)\,d\omega \,d\xi_{\ast}.
\end{aligned}
\end{equation}

First, we prove the key trilinear estimate of  $\widetilde{L}_{\xi}^{2}(\dot{B}^{s}_{2,1})$ type Besov norms. It should be pointed out that the lemma is different from the inhomogeneous version in \cite{Duan-2016}. Without the velocity variable $\xi$, similar product laws have been widely applied in the study of fluid dynamical equations (see, for example, \cite{Danchin-2000}).

\begin{lemma}\label{1dlg-L2.7}
Let $s_1, s_2\in\mathbb{R}$ be such that $-3/2<s_1,s_2\leq 3/2$ and $s_{1}+s_{2}>0$. Let $f=f(t,x,\xi), g=g(t,x,\xi)$ and $h=h(t,x,\xi)$ be suitably smooth distribution functions such that those norms on the right of the following inequality are well defined, then it holds that
\begin{equation}\label{1dlg-E2.13}
\begin{aligned}
&\sum_{q\in\mathbb{Z}}2^{q(s_{1}+s_{2}-3/2)}\bigg(\int^{T}_{0}\big|\big(\dot{\Delta}_{q}\Gamma(f,g),\dot{\Delta}_{q}h\big)_{\xi,x}\big|\,dt\bigg)^{1/2}\\
&\quad\leq C\|h\|^{1/2}_{\widetilde{L}^{2}_{T}\widetilde{L}^{2}_{\xi,\nu}(\dot{B}^{s_{1}+s_{2}-3/2}_{2,1})}\\
&\quad\quad\times \bigg( \|f\|^{1/2}_{\widetilde{L}^{\infty}_{T}\widetilde{L}^{2}_{\xi}(\dot{B}^{s_{1}}_{2,1})}\|g\|^{1/2}_{\widetilde{L}^{2}_{T}\widetilde{L}^{2}_{\xi,\nu}(\dot{B}^{s_{2}}_{2,1})}+\|f\|^{1/2}_{\widetilde{L}^{2}_{T}\widetilde{L}^{2}_{\xi,\nu}(\dot{B}^{s_{2}}_{2,1})}\|g\|^{1/2}_{\widetilde{L}^{\infty}_{T}\widetilde{L}^{2}_{\xi}(\dot{B}^{s_{1}}_{2,1})}\bigg) ,
\end{aligned}
\end{equation}
where $C>0$ is a constant depending only on $s_{1}$ and $s_{2}$.
\end{lemma}
\begin{proof}
It follows from (\ref{1dlg-2.11}) that
\begin{equation}\nonumber
\begin{aligned}
&\bigg(\int^{T}_{0}|\big(\dot{\Delta}_{q}\Gamma(f,g),\dot{\Delta}_{q}h\big)_{\xi,x}|dt\bigg)^{1/2}\\
&\quad \leq\bigg(\int^{T}_{0}|\big(\dot{\Delta}_{q}\Gamma_{{\rm gain}}(f,g),\dot{\Delta}_{q}h\big)_{\xi,x}|dt\bigg)^{1/2}+\bigg(\int^{T}_{0}|\big(\dot{\Delta}_{q}\Gamma_{{\rm loss}}(f,g),\dot{\Delta}_{q}h\big)_{\xi,x}|dt\bigg)^{1/2}.
\end{aligned}
\end{equation}
Here, we write $f_*\triangleq f(t,x,\xi_*)$ for simplicity. By 
applying the Cauchy-Schwarz inequality with respect to $(t,x,\xi,\xi_{\ast},\omega)$, making the change of variables $(\xi,\xi_{\ast})\rightarrow(\xi^{\prime},\xi^{\prime}_{\ast})$ in the gain term and then taking the summation on $q\in\mathbb{Z}$ after multiplying it by $2^{q(s_{1}+s_{2}-3/2)}$, we have
\begin{flalign}\nonumber
\begin{aligned}
&\sum_{q\in\mathbb{Z}}2^{q(s_{1}+s_{2}-3/2)}\bigg(\int^{T}_{0}|\big(\dot{\Delta}_{q}\Gamma(f,g),\dot{\Delta}_{q}h\big)_{\xi,x}|dt\bigg)^{1/2}\\
&\quad\lesssim \sum_{q\in\mathbb{Z}}2^{q(s_{1}+s_{2}-3/2)}\\
&\quad\quad\times\Bigg(\bigg(\int_{0}^{T}\int_{\mathbb{R}^{9}\times\mathbb{S}^{2}}|\xi^{\prime}-\xi^{\prime}_{\ast}|^{\gamma}\mu^{1/2}(\xi_{\ast}^{\prime})|\dot{\Delta}_{q}(f_{\ast}g)|^{2}dxd\xi d\xi_{\ast}d\omega dt\bigg)^{1/2}\Bigg)^{1/2}\\
&\quad\quad\times\Bigg(\bigg(\int_{0}^{T}\int_{\mathbb{R}^{9}\times\mathbb{S}^{2}}|\xi-\xi_{\ast}|^{\gamma}\mu^{1/2}(\xi_{\ast})|\dot{\Delta}_{q}h|^{2}dxd\xi d\xi_{\ast}d\omega dt\bigg)^{1/2}\Bigg)^{1/2}\\
&\quad\quad+\sum_{q\in\mathbb{Z}}2^{q(s_{1}+s_{2}-3/2)}\\
&\quad\quad\times\Bigg(\bigg(\int_{0}^{T}\int_{\mathbb{R}^{9}\times\mathbb{S}^{2}}|\xi-\xi_{\ast}|^{\gamma}\mu^{1/2}(\xi_{\ast})|\dot{\Delta}_{q}(f_{\ast}g)|^{2} dxd\xi d\xi_{\ast}d\omega dt\bigg)^{1/2}\Bigg)^{1/2}\\
&\quad\quad\times\bigg(\Big(\int_{0}^{T}\int_{\mathbb{R}^{9}\times\mathbb{S}^{2}}|\xi-\xi_{\ast}|^{\gamma}\mu^{1/2}(\xi_{\ast})|\dot{\Delta}_{q}h|^{2}dxd\xi d\xi_{\ast}d\omega dt\bigg)^{1/2}\Bigg)^{1/2},
\end{aligned}
\end{flalign}
where $0\leq B_{0}(\cos\theta)\leq C|\cos\theta|\leq C$ has been used.  Thanks to $\int_{\mathbb{S}^2} d\omega=4\pi$ and $|\xi'-\xi'_*|=|\xi-\xi_*|$, it holds that
\begin{flalign}
\begin{aligned}
&\sum_{q\in\mathbb{Z}}2^{q(s_{1}+s_{2}-3/2)}\bigg(\int_{0}^{T}|(\dot{\Delta}_{q}\Gamma(f,g),\dot{\Delta}_{q}h)_{\xi,x}|dt\bigg)^{1/2}\lesssim I^{1/2}  II^{1/2}
\nonumber
\end{aligned}
\end{flalign}
with
\begin{flalign}\nonumber
\begin{aligned}
I&\triangleq\sum_{q\in\mathbb{Z}}2^{q(s_{1}+s_{2}-3/2)}\bigg(\int_{0}^{T}\int_{\mathbb{R}^{9}}|\xi-\xi_{\ast}|^{\gamma}|\dot{\Delta}_{q}(f_{\ast}g)|^{2}dxd\xi d\xi_{\ast} dt\bigg)^{1/2},\\
II&\triangleq\sum_{q\in\mathbb{Z}}2^{q(s_{1}+s_{2}-3/2)}\bigg(\int_{0}^{T}\int_{\mathbb{R}^{9}}|\xi-\xi_{\ast}|^{\gamma}\mu^{1/2}(\xi_{\ast})|\dot{\Delta}_{q}h|^{2} dxd\xi d\xi_{\ast} dt\bigg)^{1/2}.
\end{aligned}
\end{flalign}
Due to the fact $\int_{\mathbb{R}^3}|\xi-\xi_*|^{\gamma}\mu^{1/2}(\xi_*)d\xi_*\sim \nu(\xi)$, it is easy to check that
$$
II\leq\|h\|_{\widetilde{L}^{2}_{T}\widetilde{L}^{2}_{\xi,\nu}(\dot{B}^{s_{1}+s_{2}-3/2}_{2,1})}.
$$
In the following, we deal with the difficult term $I$. According to the homogeneous Bony's decomposition, we write $f_{\ast}g$ as
$$
f_{\ast}g=\mathcal{\dot{T}}_{f_{\ast}}g+\mathcal{\dot{T}}_{g}f_{\ast}+\mathcal{\dot{R}}(f_{\ast} ,g),
$$
where the paraproduct $\mathcal{\dot{T}}_{u}v$ is defined by
$$
\mathcal{\dot{T}}_{u}v\triangleq \sum_{j\in\mathbb{Z}}\dot{S}_{j-1}u\dot{\Delta}_{j}v\quad\mbox{with}\quad \dot{S}_{j-1}\triangleq\chi(2^{-(j-1)}D),$$
and the remainder $\mathcal{\dot{R}}(u,v)$ is given by
$$\mathcal{\dot{R}}(u,v)\triangleq \sum_{|j^{\prime}-j|\leq1}\dot{\Delta}_{j^{\prime}}u\dot{\Delta}_{j}v.$$
By Minkowski's inequality, it holds that
\begin{flalign}
\begin{aligned}
I&\leq\sum_{q\in\mathbb{Z}}2^{q(s_{1}+s_{2}-3/2)}\bigg(\int_{0}^{T}\int_{\mathbb{R}^{9}}|\xi-\xi_{\ast}|^{\gamma}\Big|\sum_{j\in\mathbb{Z}}\dot{\Delta}_{q}(\dot{S}_{j-1}f_{\ast}\dot{\Delta}_{j}g)\Big|^{2}\,dxd\xi d\xi_{\ast}dt\bigg)^{1/2}\\
&\quad+\sum_{q\in\mathbb{Z}}2^{q(s_{1}+s_{2}-3/2)}\bigg(\int_{0}^{T}\int_{\mathbb{R}^{9}}|\xi-\xi_{\ast}|^{\gamma}\Big|\sum_{j\in\mathbb{Z}}\dot{\Delta}_{q}(\dot{S}_{j-1}g\dot{\Delta}_{j}f_{\ast})\Big|^{2}\,dxd\xi d\xi_{\ast}dt\bigg)^{1/2}\\
&\quad+\sum_{q\in\mathbb{Z}}2^{q(s_{1}+s_{2}-3/2)}\bigg(\int_{0}^{T}\int_{\mathbb{R}^{9}}|\xi-\xi_{\ast}|^{\gamma}\Big|\sum_{|j-j^{\prime}|\leq1}\dot{\Delta}_{q}(\dot{\Delta}_{j}f_{\ast}\dot{\Delta}_{j^{\prime}}g)\Big|^{2}\,dxd\xi d\xi_{\ast}dt\bigg)^{1/2}\\
&\triangleq I_{1}+I_{2}+I_{3}.
\nonumber
\end{aligned}
\end{flalign}
Those terms $I_{i}$ ($i=1,2,3$) are handled in order. In view of the compact support for the Fourier transform of $\dot{S}_{j-1}f_{\ast}\dot{\Delta}_{j}g$, 
we deduce 
\begin{equation}\nonumber
\begin{aligned}
I_{1}&\leq\sum_{q\in\mathbb{Z}}\sum_{|j-q|\leq4}2^{q(s_{1}+s_{2}-3/2)}\bigg(\int_{0}^{T}\int_{\mathbb{R}^{9}}|\xi|^{\gamma}|\dot{\Delta}_{q}(\dot{S}_{j-1}f_{\ast}\dot{\Delta}_{j}g)|^{2}dxd\xi d\xi_{\ast} dt\bigg)^{1/2}\\
&\quad+\sum_{q\in\mathbb{Z}}\sum_{|j-q|\leq4}2^{q(s_{1}+s_{2}-3/2)}\bigg(\int_{0}^{T}\int_{\mathbb{R}^{9}}|\xi_{\ast}|^{\gamma}|\dot{\Delta}_{q}(\dot{S}_{j-1}f_{\ast}\dot{\Delta}_{j}g)|^{2}dxd\xi d\xi_{\ast}dt\bigg)^{1/2}\\
&\triangleq I_{1,1}+I_{1,2},
\end{aligned}
\end{equation}
where $S_{j-1}f_{\ast}$ is equal to $\sum\limits_{k\leq j-2}\dot{\Delta}_{k}f_{\ast}$. Hence, we have 
\begin{flalign}\nonumber
\begin{aligned}
I_{1,1}&\leq\sum_{q\in\mathbb{Z}}\sum_{|j-q|\leq4}\sum_{k\leq j-2}2^{q(s_{1}+s_{2}-3/2)}\\
&\quad\times\bigg(\int_{0}^{T}\int_{\mathbb{R}^{3}}|\xi|^{\gamma}\|\dot{\Delta}_{j}g\|^{2}_{L^{2}_{x}}d\xi dt \sup_{0\leq t\leq T}\int_{\mathbb{R}^{3}}\|\dot{\Delta}_{k}f_{\ast}\|^{2}_{L^{\infty}_{x}}d\xi_{\ast}\bigg)^{1/2}\\
&\leq\sum_{q\in\mathbb{Z}}\sum_{|j-q|\leq4}2^{(q-j)s_2} 2^{js_2}(\int_{0}^{T}dt\int_{\mathbb{R}^{3}}|\xi|^{\gamma}\|\dot{\Delta}_{j}g\|^{2}_{L^{2}_{x}}d\xi\Big)^{1/2} 2^{(s_1-3/2)(q-j)}\\
&\quad\times \sum_{k\leq j-2}2^{(s_{1}-3/2)(j-k)}2^{k(s_{1}-3/2)}\bigg(\sup_{0\leq t\leq T}\int_{\mathbb{R}^{3}}\|\dot{\Delta}_{k}f_*\|^{2}_{L^{\infty}_{x}}d\xi_*\bigg)^{1/2},
\end{aligned}
\end{flalign}
which, together with Lemma \ref{1dlg-L2.2}, Young's inequality for convolutions and $s_1\leq 3/2$, gives rise to
\begin{flalign}\nonumber
\begin{aligned}
I_{1,1}\lesssim & \bigg\| \bigg\{ \Big( 2^{s_2q}\mathbf{I}_{|q|\leq 4} \Big)\ast \Big( 2^{s_2q} \|\sqrt{\nu(\xi)}\dot{\Delta}_{q}g\|_{L^2_TL^2_{\xi}L^2_x} \Big) \bigg\}_{q\in\mathbb{Z}} \bigg\|_{l^1} \\
&\times \bigg\| \bigg\{ \Big( 2^{(s_1-3/2)q}\mathbf{I}_{q\geq2} \Big)\ast \Big( 2^{(s_1-3/2)q}\sup_{0\leq t\leq T}\|\dot{\Delta}_q f\|_{L^2_{\xi}L^{\infty}_x} \Big) \bigg\}_{q\in\mathbb{Z}} \bigg\|_{l^{\infty}}\\
\lesssim & \|\{2^{s_2q}\}_{|q|\leq 4}\|_{l^1} \|g\|_{\widetilde{L}^{2}_{T}\widetilde{L}^{2}_{\xi,\nu}(\dot{B}^{s_{2}}_{2,1})}  \|\{2^{(s_1-3/2)q}\}_{q\geq 2}\|_{l^{\infty}} \|f\|_{\widetilde{L}^{\infty}_{T}\widetilde{L}^{2}_{\xi}(\dot{B}^{s_{1}-3/2}_{\infty,1})} \\
\lesssim&\|g\|_{\widetilde{L}^{2}_{T}\widetilde{L}^{2}_{\xi,\nu}(\dot{B}^{s_{2}}_{2,1})}\|f\|_{\widetilde{L}^{\infty}_{T}\widetilde{L}^{2}_{\xi}(\dot{B}^{s_{1}}_{2,1})}.
\end{aligned}
\end{flalign}
Regarding $I_{1,2}$, since $s_2\leq 3/2$, a similar calculation shows that
\begin{flalign}\nonumber
\begin{aligned}
I_{1,2}\leq&\sum_{q\in\mathbb{Z}}\sum_{|j-q|\leq4}2^{(q-j)s_1} 2^{js_1}\sup_{0\leq t\leq T}\bigg(\int_{\mathbb{R}^{3}}\|\dot{\Delta}_{j}g\|^{2}_{L^{2}_{x}}d\xi\bigg)^{1/2} 2^{(s_2-3/2)(q-j)}\\
&\times \sum_{k\leq j-2}2^{(s_{2}-3/2)(j-k)}2^{k(s_{2}-3/2)}\bigg(\int_{0}^{T}dt\int_{\mathbb{R}^{3}}|\xi_*|^{\gamma}\|\dot{\Delta}_{k}f_*\|^{2}_{L^{\infty}_{x}}d\xi_*\bigg)^{1/2} \\
\lesssim &  \|2^{s_1q}\mathbf{I}_{|q|\leq 4}\|_{l^1} \|g\|_{\widetilde{L}^{\infty}_{T}\widetilde{L}^{2}_{\xi}(\dot{B}^{s_{1}}_{2,1})} \|2^{(s_2-3/2)q}\mathbf{I}_{q\geq 2}\|_{l^{\infty}} \|f\|_{\widetilde{L}^{2}_{T}\widetilde{L}^{2}_{\xi,\nu}(\dot{B}^{s_{2}-3/2}_{\infty,1})}  \\
\lesssim&\|g\|_{\widetilde{L}^{\infty}_{T}\widetilde{L}^{2}_{\xi}(\dot{B}^{s_{1}}_{2,1})}\|f\|_{\widetilde{L}^{2}_{T}\widetilde{L}^{2}_{\xi,\nu}(\dot{B}^{s_{2}}_{2,1})}.
\end{aligned}
\end{flalign}
Thus, we get
$$I_{1}\lesssim\|g\|_{\widetilde{L}^{2}_{T}\widetilde{L}^{2}_{\xi,\nu}(\dot{B}^{s_{2}}_{2,1})}\|f\|_{\widetilde{L}^{\infty}_{T}\widetilde{L}^{2}_{\xi}(\dot{B}^{s_{1}}_{2,1})}+\|g\|_{\widetilde{L}^{\infty}_{T}\widetilde{L}^{2}_{\xi}(\dot{B}^{s_{1}}_{2,1})}\|f\|_{\widetilde{L}^{2}_{T}\widetilde{L}^{2}_{\xi,\nu}(\dot{B}^{s_{2}}_{2,1})}.$$

Next, we turn to bound $I_{2}$. Arguing similarly as for $I_{1}$, we have
\begin{equation}\label{1dlg-E2.20}
\begin{aligned}
I_{2}&\leq\sum_{q\in\mathbb{Z}}\sum_{|j-q|\leq4}2^{q(s_{1}+s_{2}-3/2)}\bigg(\int_{0}^{T}\int_{\mathbb{R}^{9}}|\xi|^{\gamma}|\dot{\Delta}_{q}(\dot{S}_{j-1}g\dot{\Delta}_{j}f_{\ast})|^{2} dxd\xi d\xi_{\ast} dt\bigg)^{1/2}\\
&\quad+\sum_{q\in\mathbb{Z}}\sum_{|j-q|\leq4}2^{q(s_{1}+s_{2}-3/2)}\bigg(\int_{0}^{T}\int_{\mathbb{R}^{9}}|\xi_{\ast}|^{\gamma}|\dot{\Delta}_{q}(\dot{S}_{j-1}g\dot{\Delta}_{j}f_{\ast})|^{2}dxd\xi d\xi_{\ast} dt\bigg)^{1/2}\\
&\triangleq I_{2,1}+I_{2,2},
\end{aligned}
\end{equation}
where $I_{2,1}$ and $I_{2,2}$ can be estimated as
\begin{flalign*}
\begin{aligned}
I_{2,1}&\leq\sum_{q\in\mathbb{Z}}\sum_{|j-q|\leq4}\sum_{k\leq j-2}2^{q(s_{1}+s_{2}-3/2)}\\
&\quad\times\bigg(\int_{0}^{T}dt\int_{\mathbb{R}^{3}}|\xi|^{\gamma}\|\dot{\Delta}_{k}g\|^{2}_{L^{\infty}_{x}}d\xi\int_{\mathbb{R}^{3}}\|\dot{\Delta}_{j}f_*\|^{2}_{L^{2}_{x}}d\xi_*\bigg)^{1/2}\\
&\leq\sum_{q\in\mathbb{Z}}\sum_{|j-q|\leq4}2^{(q-j)s_1} 2^{js_1}\sup_{0\leq t\leq T}\bigg(\int_{\mathbb{R}^{3}}\|\dot{\Delta}_{j}f_*\|^{2}_{L^{2}_{x}}d\xi_*\bigg)^{1/2} 2^{(s_{2}-3/2)(q-j)}\\
&\quad\times \sum_{k\leq j-2}2^{(s_{2}-3/2)(j-k)}2^{k(s_{2}-3/2)}\bigg(\int_{0}^{T}dt\int_{\mathbb{R}^{3}}|\xi|^{\gamma}\|\dot{\Delta}_{k}g\|^{2}_{L^{\infty}_{x}}d\xi\bigg)^{1/2} \\
&\lesssim\|f\|_{\widetilde{L}^{\infty}_{T}\widetilde{L}^{2}_{\xi}(\dot{B}^{s_{1}}_{2,1})}\|g\|_{\widetilde{L}^{2}_{T}\widetilde{L}^{2}_{\xi,\nu}(\dot{B}^{s_{2}}_{2,1})},
\end{aligned}
\end{flalign*}
and similarly,
\begin{equation}\nonumber
I_{2,2}\lesssim\|f\|_{\widetilde{L}^{2}_{T}\widetilde{L}^{2}_{\xi,\nu}(\dot{B}^{s_{2}}_{2,1})}\|g\|_{\widetilde{L}^{\infty}_{T}\widetilde{L}^{2}_{\xi}(\dot{B}^{s_{1}}_{2,1})}.
\end{equation}
Consequently, it holds that
$$
I_{2}\lesssim \|f\|_{\widetilde{L}^{\infty}_{T}\widetilde{L}^{2}_{\xi}(\dot{B}^{s_{1}}_{2,1})}\|g\|_{\widetilde{L}^{2}_{T}\widetilde{L}^{2}_{\xi,\nu}(\dot{B}^{s_{2}}_{2,1})}+\|f\|_{\widetilde{L}^{2}_{T}\widetilde{L}^{2}_{\xi,\nu}(\dot{B}^{s_{2}}_{2,1})}\|g\|_{\widetilde{L}^{\infty}_{T}\widetilde{L}^{2}_{\xi}(\dot{B}^{s_{1}}_{2,1})}.
$$

Regarding the reminder term $I_{3}$, due to the fact that
$$
\sum_{j\in\mathbb{Z}}\sum_{|j-j^{\prime}|\leq1}\dot{\Delta}_{q}(\dot{\Delta}_{j^{\prime}}f_{\ast}\dot{\Delta}_{j}g)=\sum_{\max\{j,j^{\prime}\}\geq q-2}\sum_{|j-j^{\prime}|\leq1}\dot{\Delta}_{q}(\dot{\Delta}_{j^{\prime}}f_{\ast}\dot{\Delta}_{j}g),
$$
one has
\begin{flalign}
\begin{aligned}
I_{3}&\leq\sum_{q\in\mathbb{Z}}\sum_{\max\{j,j^{\prime}\}\geq q-2}\sum_{|j-j^{\prime}|\leq1}2^{q(s_{1}+s_{2}-3/2)}\\
&\quad\times\bigg(\int_{0}^{T}\int_{\mathbb{R}^{9}}|\xi|^{\gamma}|\dot{\Delta}_{q}(\dot{\Delta}_{j^{\prime}}f_{\ast}\dot{\Delta}_{j}g)|^{2} dxd\xi d\xi_{\ast} dt\bigg)^{1/2}\\
&\quad+\sum_{q\in\mathbb{Z}}\sum_{\max\{j,j^{\prime}\}\geq q-2}\sum_{|j-j^{\prime}|\leq1}2^{q(s_{1}+s_{2}-3/2)}\\
&\quad\times\bigg(\int_{0}^{T}\int_{\mathbb{R}^{9}}|\xi_{\ast}|^{\gamma}|\dot{\Delta}_{q}(\dot{\Delta}_{j^{\prime}}f_{\ast}\dot{\Delta}_{j}g)|^{2} dxd\xi d\xi_{\ast} dt\bigg)^{1/2}\\
&\triangleq I_{3,1}+I_{3,2}.
\nonumber
\end{aligned}
\end{flalign}
As the Bernstein lemma (Lemma \ref{1dlg-L2.2}) and Young's inequality for convolutions, one can handle $I_{3,1}$ as follows:
\begin{flalign*}
\begin{aligned}
I_{3,1}&
\lesssim\sum_{q\in\mathbb{Z}}\sum_{j\geq q-3}2^{q(s_{1}+s_{2})}\bigg(\int_{0}^{T}\int_{\mathbb{R}^{6}}|\xi|^{\gamma}\|\dot{\Delta}_{j}f_{\ast}\dot{\Delta}_{j}g\|^{2}_{L^{1}_{x}}d\xi d\xi_{\ast}dt\bigg)^{1/2}\\
&\lesssim\sum_{q\in\mathbb{Z}}\sum_{j\geq q-3}2^{(q-j)(s_{1}+s_{2})}2^{js_{2}}\bigg(\int_{0}^{T}\int_{\mathbb{R}^{3}}|\xi|^{\gamma}\|\dot{\Delta}_{j}g\|^{2}_{L^{2}_{x}}d\xi dt\bigg)^{1/2} \\
&\quad\times2^{js_{1}}\bigg(\sup_{0\leq t\leq T}\int_{\mathbb{R}^{3}}\|\dot{\Delta}_{j}f_{\ast}\|^{2}_{L^{2}_{x}}d\xi_{\ast}dt\bigg)^{1/2}\\
&\lesssim \bigg\|\bigg\{  \Big(2^{(s_1+s_2)q}\mathbf{I}_{q\leq 3}\Big)\ast \Big(2^{qs_2}\|\sqrt{\nu(\xi)}\dot{\Delta}_{q}g\|_{L^2_TL^2_{\xi}L^2_x}\Big)         \bigg\}_{q\in\mathbb{Z}} \bigg\|_{l^1} \|f\|_{L^{\infty}_{T}L^{2}_{\xi}(\dot{B}^{s_{1}}_{2,\infty})}\\
&\lesssim \|g\|_{\widetilde{L}^{2}_{T}\widetilde{L}^{2}_{\xi,\nu}(\dot{B}^{s_{2}}_{2,1})}\|f\|_{\widetilde{L}^{\infty}_{T}\widetilde{L}^{2}_{\xi,\nu}(\dot{B}^{s_{1}}_{2,1})}.
\end{aligned}
\end{flalign*}
where we used $\|\{2^{(s_1+s_2)q}\}_{q\leq 3}\|_{l^1}\leq C$ due to $s_1+s_2>0$. Similarly, we see that 
$$
I_{3,2}\lesssim\|f\|_{\widetilde{L}^{2}_{T}\widetilde{L}^{2}_{\xi,\nu}(\dot{B}^{s_{2}}_{2,1})}\|g\|_{\widetilde{L}^{\infty}_{T}\widetilde{L}^{2}_{\xi}(\dot{B}^{s_{1}}_{2,1})}.
$$
Combining all the above estimates for $I_{1},~I_{2}$ and $I_{3}$, we obtain the desired inequality (\ref{1dlg-E2.13}).
\end{proof}

\vspace{2mm}

Next, we give the trilinear estimate of the $\widetilde{L}^2_{\xi}(\dot{B}^{s}_{2,\infty})$-type Besov norms, which plays a key role in deriving optimal decay estimates. Compared with \eqref{1dlg-E2.13}, this estimate can cover the case $s_1+s_2=0$. In particular, when $s_1=-3/2$ and $s_2=3/2$, we are able to deal with the $\widetilde{L}^2_{\xi}(\dot{B}^{-3/2}_{2,\infty})$-norm associated with the embedding in $L^2_{\xi}(L^1_{x})$.

\begin{lemma}\label{1dlg-L2.8}
Assume $s_1, s_2\in\mathbb{R}$ such that $-3/2\leq s_1<3/2, -3/2<s_2\leq 3/2$ and $s_{1}+s_{2}\geq0$. For suitably smooth distribution functions $f=f(t,x,\xi), g=g(t,x,\xi)$ and $h=h(t,x,\xi)$, we have
\begin{equation}\label{1dlg-E2.23}
\begin{aligned}                 
&\sup_{q\in\mathbb{Z}}2^{q(s_{1}+s_{2}-3/2)}\bigg(\int^{T}_{0}\big|\big(\dot{\Delta}_{q}\Gamma(f,g),\dot{\Delta}_{q}h\big)_{\xi,x}\big|\,dt\bigg)^{1/2}\\
&\quad\leq C\|h\|^{1/2}_{\widetilde{L}^{2}_{T}\widetilde{L}^{2}_{\xi,\nu}(\dot{B}^{s_{1}+s_{2}-3/2}_{2,\infty})}\\
&\quad\quad\times\bigg(\|f\|^{1/2}_{\widetilde{L}^{\infty}_{T}\widetilde{L}^{2}_{\xi}(\dot{B}^{s_{1}}_{2,\infty})}\|g\|^{1/2}_{\widetilde{L}^{2}_{T}\widetilde{L}^{2}_{\xi,\nu}(\dot{B}^{s_{2}}_{2,1})}+\|f\|^{1/2}_{\widetilde{L}^{2}_{T}\widetilde{L}^{2}_{\xi,\nu}(\dot{B}^{s_{2}}_{2,1})}\|g\|^{1/2}_{\widetilde{L}^{\infty}_{T}\widetilde{L}^{2}_{\xi}(\dot{B}^{s_{1}}_{2,\infty})}\bigg) ,
\end{aligned}
\end{equation}
where $C$ is a constant depending only on $s_{1}$ and $s_{2}$.
\end{lemma}
\begin{proof}
We have
\begin{flalign*}
\begin{aligned}
&\sup_{q\in\mathbb{Z}}2^{q(s_{1}+s_{2}-3/2)}\bigg(\int_{0}^{T}|(\dot{\Delta}_{q}\Gamma(f,g),\dot{\Delta}_{q}h)_{\xi,x}|dt\bigg)^{1/2}\lesssim I_*^{1/2} II_{*}^{1/2}
\end{aligned}
\end{flalign*}
with
\begin{flalign*}
\begin{aligned}
I_*&\triangleq\sup_{q\in\mathbb{Z}}2^{q(s_{1}+s_{2}-3/2)}\bigg(\int_{0}^{T}\int_{\mathbb{R}^{9}\times\mathbb{S}^{2}}|\xi-\xi_{\ast}|^{\gamma}|\dot{\Delta}_{q}(f_{\ast}g)|^{2} 
 dxd\xi d\xi_{\ast}d\omega dt\bigg)^{1/2},\\
II_*&\triangleq\sup_{q\in\mathbb{Z}}2^{q(s_{1}+s_{2}-3/2)}\bigg(\int_{0}^{T}\int_{\mathbb{R}^{9}\times\mathbb{S}^{2}} |\xi-\xi_{\ast}|^{\gamma}\mu^{1/2}(\xi_{\ast})|\dot{\Delta}_{q}h|^{2} dxd\xi d\xi_{\ast}d\omega dt\bigg)^{1/2}.
\end{aligned}
\end{flalign*}
It is straightforward to see that  
$$
II_{*}\leq\|h\|_{\widetilde{L}^{2}_{T}\widetilde{L}^{2}_{\xi,\nu}(\dot{B}^{s_{1}+s_{2}-3/2}_{2,\infty})}.
$$
By using Bony's decomposition, we bound $I_{*}$ as follows
\begin{flalign}
\begin{aligned}
I_{*}&\leq \sup_{q\in\mathbb{Z}} 2^{q(s_{1}+s_{2}-3/2)}\int_{0}^{T}\bigg(\int_{\mathbb{R}^{9}}|\xi-\xi_{\ast}|^{\gamma}\Big|\sum_{j\in\mathbb{Z}}\dot{\Delta}_{q}(\dot{S}_{j-1}f_{\ast}\dot{\Delta}_{j}g)\Big|^{2}\,dxd\xi d\xi_{\ast}dt\bigg)^{1/2}\\
&\quad+\sup_{q\in\mathbb{Z}} 2^{q(s_{1}+s_{2}-3/2)}\int_{0}^{T}\bigg(\int_{\mathbb{R}^{9}}|\xi-\xi_{\ast}|^{\gamma}\Big|\sum_{j\in\mathbb{Z}}\dot{\Delta}_{q}(\dot{S}_{j-1}g\dot{\Delta}_{j}f_{\ast})\Big|^{2}\,dxd\xi d\xi_{\ast}dt\bigg)^{1/2}\\
&\quad+\sup_{q\in\mathbb{Z}} 2^{q(s_{1}+s_{2}-3/2)}\int_{0}^{T}\bigg(\int_{\mathbb{R}^{9}}|\xi-\xi_{\ast}|^{\gamma}\Big|\sum_{|j-j^{\prime}|\leq1}\dot{\Delta}_{q}(\dot{\Delta}_{j}f_{\ast}\dot{\Delta}_{j^{\prime}}g)\Big|^{2}\,dxd\xi d\xi_{\ast}dt\bigg)^{1/2}\\
&\triangleq J_{1}+J_{2}+J_{3}.
\nonumber
\end{aligned}
\end{flalign}
We estimate $J_1$ as
\begin{equation}\nonumber
\begin{aligned}
J_{1}&\leq\sup_{q\in\mathbb{Z}}\sum_{|j-q|\leq4}2^{q(s_{1}+s_{2}-3/2)}\bigg(\int_{0}^{T}\int_{\mathbb{R}^{9}}|\xi|^{\gamma}|\dot{\Delta}_{q}(\dot{S}_{j-1}f_{\ast}\dot{\Delta}_{j}g)|^{2}dxd\xi d\xi_{\ast}dt\bigg)^{1/2}\\
&\quad+\sup_{q\in\mathbb{Z}}\sum_{|j-q|\leq4}2^{q(s_{1}+s_{2}-3/2)}\bigg(\int_{0}^{T}\int_{\mathbb{R}^{9}}|\xi_{\ast}|^{\gamma}|\dot{\Delta}_{q}(\dot{S}_{j-1}f_{\ast}\dot{\Delta}_{j}g)|^{2}dxd\xi d\xi_{\ast}dt\bigg)^{1/2}\\
&\triangleq J_{1,1}+J_{1,2}.
\end{aligned}
\end{equation}
As $s_1<3/2$, the Bernstein lemma (Lemma \ref{1dlg-L2.2}) and Young's inequality enable us to obtain
\begin{flalign*}
\begin{aligned}
J_{1,1}&\leq\sup_{q\in\mathbb{Z}}\sum_{|j-q|\leq4}\sum_{k\leq j-2}2^{q(s_{1}+s_{2}-3/2)}\\
&\quad\times\bigg(\sup_{0\leq t\leq T}\int_{\mathbb{R}^{3}}\|\dot{\Delta}_{k}f_{\ast}\|^{2}_{L^{\infty}_{x}}d\xi_{\ast}\int_{0}^{T}dt\int_{\mathbb{R}^{3}}|\xi|^{\gamma}\|\dot{\Delta}_{j}g\|^{2}_{L^{2}_{x}}d\xi\bigg)^{1/2}\\
&\leq\sup_{q\in\mathbb{Z}}\sum_{|j-q|\leq4}2^{(q-j)s_2} 2^{js_2}\bigg(\int_{0}^{T}dt\int_{\mathbb{R}^{3}}|\xi|^{\gamma}\|\dot{\Delta}_{j}g\|^{2}_{L^{2}_{x}}d\xi\bigg)^{1/2} 2^{(s_1-3/2)(q-j)}\\
&\quad\times \sum_{k\leq j-2}2^{(s_{1}-3/2)(j-k)}2^{k(s_{1}-3/2)}\bigg(\sup_{0\leq t\leq T}\int_{\mathbb{R}^{3}}\|\dot{\Delta}_{k}f_*\|^{2}_{L^{\infty}_{x}}d\xi_*\bigg)^{1/2}\\
 &\lesssim \bigg\|\bigg\{ \Big(2^{s_2q}\mathbf{I}_{|q|\leq 4}\Big)\ast \Big( 2^{s_2q} \|\sqrt{\nu(\xi)}\dot{\Delta}_{q}g\|_{L^2_TL^2_{\xi}L^2_x}\Big) \bigg\}_{q\in\mathbb{Z}}\bigg\|_{l^\infty} \\
&\quad\times \bigg\| \bigg\{ \Big( 2^{(s_1-3/2)q}\mathbf{I}_{q\geq2} \big)\ast \Big( 2^{(s_1-3/2)q}\sup_{0\leq t\leq T}\|\dot{\Delta}_q f\|_{L^2_{\xi}L^{\infty}_x} \big) \bigg\}_{q\in\mathbb{Z}}\bigg\|_{l^{\infty}}\\
 &\lesssim \|\{2^{s_2q}\}_{|q|\leq 4}\|_{l^\infty} \|g\|_{\widetilde{L}^{2}_{T}\widetilde{L}^{2}_{\xi,\nu}(\dot{B}^{s_{2}}_{2,1})}  
 \|\{2^{(s_1-3/2)q}\}_{q\geq 2}\|_{l^{1}} \|f\|_{\widetilde{L}^{\infty}_{T}\widetilde{L}^{2}_{\xi}(\dot{B}^{s_{1}-3/2}_{\infty,\infty})}  \\
&\lesssim\|g\|_{\widetilde{L}^{2}_{T}\widetilde{L}^{2}_{\xi,\nu}(\dot{B}^{s_{2}}_{2,1})}\|f\|_{\widetilde{L}^{\infty}_{T}\widetilde{L}^{2}_{\xi}(\dot{B}^{s_{1}}_{2,\infty})}.
\end{aligned}
\end{flalign*}
For term $J_{1,2}$, one has
\begin{flalign}\nonumber
\begin{aligned}
J_{1,2}\leq&\sup_{q\in\mathbb{Z}}\sum_{|j-q|\leq4}2^{(q-j)s_1} 2^{js_1}\sup_{0\leq t\leq T}\bigg(\int_{\mathbb{R}^{3}}\|\dot{\Delta}_{j}g\|^{2}_{L^{2}_{x}}d\xi\bigg)^{1/2} 2^{(s_2-3/2)(q-j)}\\
&\times \sum_{k\leq j-2}2^{(s_{2}-3/2)(j-k)}2^{k(s_{2}-3/2)}\bigg(\int_{0}^{T}\int_{\mathbb{R}^{3}}|\xi_*|^{\gamma}\|\dot{\Delta}_{k}f_*\|^{2}_{L^{\infty}_{x}}d\xi_*dt\bigg)^{1/2} \\
\lesssim &  \|\{2^{s_1q}\}_{|q|\leq 4}\|_{l^1} \|g\|_{\widetilde{L}^{\infty}_{T}\widetilde{L}^{2}_{\xi}(\dot{B}^{s_{1}}_{2,\infty})}  \|\{2^{(s_2-3/2)q}\}_{q\geq 2}\|_{l^{\infty}}  \|f\|_{\widetilde{L}^{2}_{T}\widetilde{L}^{2}_{\xi,\nu}(\dot{B}^{s_{2}-3/2}_{\infty,1})} \\
\lesssim&\|g\|_{\widetilde{L}^{\infty}_{T}\widetilde{L}^{2}_{\xi}(\dot{B}^{s_{1}}_{2,\infty})}\|f\|_{\widetilde{L}^{2}_{T}\widetilde{L}^{2}_{\xi,\nu}(\dot{B}^{s_{2}}_{2,1})}.
\end{aligned}
\end{flalign}
Thus, we get
$$J_{1}\lesssim\|f\|_{\widetilde{L}^{\infty}_{T}\widetilde{L}^{2}_{\xi}(\dot{B}^{s_{1}}_{2,\infty})}\|g\|_{\widetilde{L}^{2}_{T}\widetilde{L}^{2}_{\xi,\nu}(\dot{B}^{s_{2}}_{2,1})}+\|g\|_{\widetilde{L}^{\infty}_{T}\widetilde{L}^{2}_{\xi}(\dot{B}^{s_{1}}_{2,\infty})}\|f\|_{\widetilde{L}^{2}_{T}\widetilde{L}^{2}_{\xi,\nu}(\dot{B}^{s_{2}}_{2,1})}.$$

Next, we turn to bound $J_2$. Precisely,
\begin{equation}\nonumber
\begin{aligned}
J_2&\leq\sup_{q\in\mathbb{Z}}\sum_{|j-q|\leq4}2^{q(s_{1}+s_{2}-3/2)}\bigg(\int_{0}^{T}\int_{\mathbb{R}^{9}}|\xi|^{\gamma}|\dot{\Delta}_{q}(\dot{S}_{j-1}g\dot{\Delta}_{j}f_{\ast})|^{2}dxd\xi d\xi_{\ast}dt\bigg)^{1/2}\\
&\quad+\sup_{q\in\mathbb{Z}}\sum_{|j-q|\leq4}2^{q(s_{1}+s_{2}-3/2)}\bigg(\int_{0}^{T}\int_{\mathbb{R}^{9}}|\xi_{\ast}|^{\gamma}|\dot{\Delta}_{q}(\dot{S}_{j-1}g\dot{\Delta}_{j}f_{\ast})|^{2}dxd\xi d\xi_{\ast}dt\bigg)^{1/2}\\
&\triangleq J_{2,1}+J_{2,2},
\end{aligned}
\end{equation}
which can be estimated as for $J_{1,1}$ and $J_{1,2}$ similarly. Indeed, for $s_1<3/2$ and $s_2\leqq3/2$, we arrive at
\begin{flalign*}
\begin{aligned}
J_{2,1}&\leq\sup_{q\in\mathbb{Z}}\sum_{|j-q|\leq4}\sum_{k\leq j-2}2^{q(s_{1}+s_{2}-3/2)}\\
&\quad\times\bigg(\int_{0}^{T}dt\int_{\mathbb{R}^{3}}|\xi|^{\gamma}\|\dot{\Delta}_{k}g\|^{2}_{L^{\infty}_{x}}d\xi\int_{\mathbb{R}^{3}}\|\dot{\Delta}_{j}f_*\|^{2}_{L^{2}_{x}}d\xi_*\bigg)^{1/2}\\
&\leq \sup_{q\in\mathbb{Z}} \sum_{|j-q|\leq4}2^{(q-j)s_1} 2^{js_1}\sup_{0\leq t\leq T}\bigg(\int_{\mathbb{R}^{3}}\|\dot{\Delta}_{j}f_*\|^{2}_{L^{2}_{x}}d\xi_*\bigg)^{1/2} 2^{(s_{2}-3/2)(q-j)}\\
&\quad\times \sum_{k\leq j-2}2^{(s_{2}-3/2)(j-k)}2^{k(s_{2}-3/2)}\bigg(\int_{0}^{T}dt\int_{\mathbb{R}^{3}}|\xi|^{\gamma}\|\dot{\Delta}_{k}g\|^{2}_{L^{\infty}_{x}}d\xi\bigg)^{1/2} \\
&\lesssim\|f\|_{\widetilde{L}^{\infty}_{T}\widetilde{L}^{2}_{\xi}(\dot{B}^{s_{1}}_{2,\infty})}\|g\|_{\widetilde{L}^{2}_{T}\widetilde{L}^{2}_{\xi,\nu}(\dot{B}^{s_{2}}_{2,1})}.
\end{aligned}
\end{flalign*}
A similar computation leads to
\begin{flalign*}
J_{2,2}\lesssim\|f\|_{\widetilde{L}^{2}_{T}\widetilde{L}^{2}_{\xi,\nu}(\dot{B}^{s_{2}}_{2,1})}\|g\|_{\widetilde{L}^{\infty}_{T}\widetilde{L}^{2}_{\xi}(\dot{B}^{s_{1}}_{2,\infty})}.
\end{flalign*}
The term $J_{3}$ can be bounded by
\begin{flalign}
\begin{aligned}
J_3&\leq\sup_{q\in\mathbb{Z}}\sum_{\max\{j,j^{\prime}\}\geq q-2}\sum_{|j-j^{\prime}|\leq1}2^{q(s_{1}+s_{2}-3/2)}\\
&\quad\times\bigg(\int_{0}^{T}\int_{\mathbb{R}^{9}}|\xi|^{\gamma}|\dot{\Delta}_{q}(\dot{\Delta}_{j^{\prime}}f_{\ast}\dot{\Delta}_{j}g)|^{2}dxd\xi d\xi_{\ast}dt\bigg)^{1/2}\\
&\quad+\sup_{q\in\mathbb{Z}}\sum_{\max\{j,j^{\prime}\}\geq q-2}\sum_{|j-j^{\prime}|\leq1}2^{q(s_{1}+s_{2}-3/2)}\\
&\quad\times\bigg(\int_{0}^{T}\int_{\mathbb{R}^{9}}|\xi_{\ast}|^{\gamma}|\dot{\Delta}_{q}(\dot{\Delta}_{j^{\prime}}f_{\ast}\dot{\Delta}_{j}g)|^{2}dxd\xi d\xi_{\ast}dt\bigg)^{1/2}\\
&\triangleq J_{3,1}+J_{3,2}.
\nonumber
\end{aligned}
\end{flalign}
By employing Bernstein's lemma (Lemma \ref{1dlg-L2.2}), Young's inequality and the fact that $s_1+s_2\geq0$, we have
\begin{equation}\nonumber
\begin{aligned}
J_{3,1}&\leq\sup_{q\in\mathbb{Z}}\sum_{j\geq q-3}2^{q(s_{1}+s_{2}-3/2)}2^{3q(1-1/2)}\bigg(\int_{0}^{T}\int_{\mathbb{R}^{3}}|\xi|^{\gamma}\|\dot{\Delta}_{q}[\dot{\Delta}_{j}f_{\ast}\dot{\Delta}_{j}g]\|^{2}_{L^{1}_{x}}d\xi d\xi_{\ast}dt\bigg)^{1/2}\\
&\lesssim \bigg\|  \bigg\{ \Big(2^{(s_1+s_2)q}\mathbf{I}_{q\leq 3}\Big)\ast \Big(2^{qs_2}\|\sqrt{\nu(\xi)}\dot{\Delta}_{q}g\|_{L^2_TL^2_{\xi}L^2_x} \Big) \bigg\}_{q\in\mathbb{Z}}         \bigg\|_{l^{\infty}} \|f\|_{\widetilde{L}^{\infty}_{T}L^{2}_{\xi}(\dot{B}^{s_{1}}_{2,\infty})}\\
&\lesssim \| \{ 2^{(s_1+s_2)q}\}_{q\leq 3}\|_{l^{\infty}}\|g\|_{\widetilde{L}^{2}_{T}\widetilde{L}^{2}_{\xi,\nu}(\dot{B}^{s_{2}}_{2,1})}\|f\|_{\widetilde{L}^{\infty}_{T}\widetilde{L}^{2}_{\xi}(\dot{B}^{s_{1}}_{2,\infty})}\\
&\lesssim \|g\|_{\widetilde{L}^{2}_{T}\widetilde{L}^{2}_{\xi,\nu}(\dot{B}^{s_{2}}_{2,1})}\|f\|_{\widetilde{L}^{\infty}_{T}\widetilde{L}^{2}_{\xi}(\dot{B}^{s_{1}}_{2,\infty})}.
\end{aligned}
\end{equation}
Similarly,
$$J_{3,2}\lesssim\|f\|_{\widetilde{L}^{2}_{T}\widetilde{L}^{2}_{\xi,\nu}(\dot{B}^{s_{1}}_{2,1})}\|g\|_{\widetilde{L}^{\infty}_{T}\widetilde{L}^{2}_{\xi}(\dot{B}^{s_{2}}_{2,\infty})}.$$
Therefore, collecting all above estimates of $J_{i,j}~(i=1,2,3,~j=1,2)$, we  end up with (\ref{1dlg-E2.23}) eventually.
\end{proof}

Finally, we present the  following trilinear estimates, which  
will be used to establish the macroscopic
dissipation rates in Sections \ref{1dlg-S5} and \ref{1dlg-S6}. For brevity, we would like to omit those proofs as they are closely connected with those of Lemmas \ref{1dlg-L2.7} and \ref{1dlg-L2.8}, which are left to interested readers. 
\begin{lemma}\label{1dlg-L2.9}
Let $\zeta=\zeta(\xi)\in\mathcal{S}(\mathbb{R}^{3}_{\xi})$ and $s_1, s_2\in\mathbb{R}$ such that $-3/2<s_1,s_2\leq 3/2$ and $s_{1}+s_{2}>0$. Then it holds that
\begin{equation}\label{1dlg-E2.24}
\begin{aligned}
&\sum_{q\in\mathbb{Z}}2^{q(s_{1}+s_{2}-3/2)}\bigg(\int_{0}^{T}\|(\dot{\Delta}_{q}\Gamma(f,g),\zeta)_{\xi}\|^{2}_{L^{2}_{x}}\,dt\bigg)^{1/2}\leq C\|g\|_{\widetilde{L}^{2}_{T}\widetilde{L}^{2}_{\xi}(\dot{B}^{s_{2}}_{2,1})}\|f\|_{\widetilde{L}^{\infty}_{T}\widetilde{L}^{2}_{\xi}(\dot{B}^{s_{1}}_{2,1})}
\end{aligned}
\end{equation}
for any $T>0$, where $C>0$ is a constant depending only on $\zeta$ and ($s_{1},s_{2})$.
\end{lemma}
\begin{lemma}\label{1dlg-L2.10}
Let $\zeta=\zeta(\xi)\in\mathcal{S}(\mathbb{R}^{3}_{\xi})$ and $s_1, s_2\in\mathbb{R}$ such that $-3/2\leq s_1<3/2, -3/2<s_2\leq 3/2$ and $s_{1}+s_{2}\geq0$. Then it holds that
\begin{equation}\label{1dlg-E2.25}
\begin{aligned}
&\sup_{q\in\mathbb{Z}}2^{q(s_{1}+s_{2}-3/2)}\bigg(\int_{0}^{T}\|(\dot{\Delta}_{q}\Gamma(f,g),\zeta)_{\xi}\|^{2}_{L^{2}_{x}}\,dt\bigg)^{1/2} \leq C\|g\|_{\widetilde{L}^{2}_{T}\widetilde{L}^{2}_{\xi}(\dot{B}^{s_{2}}_{2,1})}\|f\|_{\widetilde{L}^{\infty}_{T}\widetilde{L}^{2}_{\xi}(\dot{B}^{s_{1}}_{2,\infty})}
\end{aligned}
\end{equation}
for any $T>0$, where $C$ is a constant depending only on $\zeta$ and ($s_{1},s_{2})$.
\end{lemma}

\section{Global existence}\label{1dlg-S5}

\subsection{A priori estimates}
This section is devoted to deriving the key a priori estimate for the Cauchy problem \eqref{1dlg-E1.4}, 
which will be used to prove Theorem \ref{1dlg-T1.2}. 
\begin{prop}\label{1dlg-P5.1}
Assume that  $f$ is the solution to the Cauchy problem $\rm(\ref{1dlg-E1.4})$ on any interval $[0,T)$.  Then for $0<t<T$, it holds that
\begin{equation}\label{1dlg-E5.8}
\begin{aligned}
&\mathcal{E}_{t}(f)+\mathcal{D}_{t}(f)\\
&\quad \leq C_{1}\Big(\|f_{0}\|^{\ell}_{\widetilde{L}^{2}_{\xi}(\dot{B}^{1/2}_{2,1})}+\|f_{0}\|^{h}_{\widetilde{L}^{2}_{\xi}(\dot{B}^{3/2}_{2,1})}\Big)+C_{1}\Big(\sqrt{\mathcal{E}_{t}(f)}+\mathcal{E}_{t}(f)\Big)\mathcal{D}_{t}(f),
\end{aligned}
\end{equation}
where $\mathcal{E}_{t}(f)$ and $\mathcal{D}_{t}(f)$ are defined by $\rm(\ref{1dlg-E1.16})$ and $\rm(\ref{1dlg-E1.17})$, respectively, and $C_{1}>0$ is a generic constant independent of $T$.
\end{prop}
The proof of Proposition \ref{1dlg-P5.1} is divided into low-frequency analysis and high-frequency analysis below.

\subsubsection{Low-frequency analysis}
In this subsection, we establish uniform estimates of solutions to the Cauchy problem (\ref{1dlg-E1.4}) in the low-frequency regime. To capture the dissipation of $\mathbf{P}f$, we perform the classical analysis (see \cite{Duan-2011}) that the coefficient functions $(a,b,c)$ satisfy the following fluid-type equations
\begin{equation}\label{1dlg-E5.90}
 \left\{
\begin{aligned}
&\partial_{t}a+\nabla_{x}\cdot b=0,\\
&\partial_{t}b+\nabla_{x} (a+2c)+\nabla_{x}\cdot\Theta( \{\mathbf{I-P}\}f)=0,\\
&\partial_{t} c+\frac{1}{3}\nabla_{x}\cdot  b+\frac{5}{3}\nabla_{x}\cdot\Lambda( \{\mathbf{I-P}\}f)=0,\\
&\partial_{t}\big(\Theta_{ij}( \{\mathbf{I-P}\}f)+2 c\,\delta_{ij}\big)+\partial_{i} b_{j}+\partial_{j} b_{i}+\Theta_{ij}( \mathbf{r}( \{\mathbf{I-P}\}f))=\Theta_{ij}( \mathbf{h}),\\
&\partial_{t}\Lambda_{i}( \{\mathbf{I-P}\}f)+\partial_{i} c+\Lambda_{i}( \mathbf{r}( \{\mathbf{I-P}\}f))=\Lambda_{i}( \mathbf{h}),\\
\end{aligned}
 \right.
\end{equation}
where the high-order moment functions $\Theta=(\Theta_{ij}(\cdot))_{3\times3}$ and $\Lambda=(\Lambda_{i}(\cdot))_{1\leq i\leq3}$ are defined by
\begin{align}
\Theta_{ij}(f)=((\xi_{i}\xi_{j}-1)\mu^{1/2},f)_{\xi}\quad \text{and}\quad \Lambda_{i}(f)=\frac{1}{10}(|\xi|^{2}-5)\xi_{i}\mu^{1/2},f)_{\xi},\label{theta}
\end{align}
respectively, with the inner product taken with respect to the velocity variable $\xi$ only. In addition, we define
\begin{align}
\mathbf{r}( \{\mathbf{I-P}\}f)=\xi\cdot\nabla_{x}\{\mathbf{I-P}\}f+L\{\mathbf{I-P}\}f\quad \text{and}\quad \mathbf{h}=\Gamma(f,f).\label{r}
\end{align}
\begin{lemma}\label{1dlg-R5.3}
Assume that  $f$ is the solution to the Cauchy problem $\rm(\ref{1dlg-E1.4})$ on any interval $[0,T)$.  Then for $0<t<T$, it holds that
\begin{equation}\label{1dlg-E5.17}
\begin{aligned}
&\|f\|^{\ell}_{\widetilde{L}^{\infty}_{t}\widetilde{L}^{2}_{\xi}(\dot{B}^{1/2}_{2,1})}+\|\mathbf{P}f\|^{\ell}_{\widetilde{L}^{2}_{t}\widetilde{L}^{2}_{\xi}(\dot{B}^{3/2}_{2,1})}+\|\{\mathbf{I-P}\}f\|^{\ell}_{\widetilde{L}^{2}_{t}\widetilde{L}^{2}_{\xi,\nu}(\dot{B}^{1/2}_{2,1})}\\
&\quad\leq C\|f_{0}\|^{\ell}_{\widetilde{L}^{2}_{\xi}(\dot{B}^{1/2}_{2,1})}+C\Big(\sqrt{\mathcal{E}_{t}(f)}+\mathcal{E}_{t}(f)\Big)\mathcal{D}_{t}(f),
\end{aligned}
\end{equation}
where $\mathcal{E}_{t}(f)$ and $\mathcal{D}_{t}(f)$ are defined by $\rm(\ref{1dlg-E1.16})$ and $\rm(\ref{1dlg-E1.17})$, respectively, and $C$ is a constant independent of $T$.
\end{lemma}
\begin{proof}
We aim to construct a suitable Lyapunov functional and capture the dissipation effects of $\mathbf{P}f$ and $\{\mathbf{I-P}\}f$ for low frequencies in the spirit of hypocoercivity. For that end, by applying the operator $\dot{\Delta}_{q}$ to (\ref{1dlg-E1.4}), we get \vspace{-2mm}
\begin{equation}\label{1dlg-E1.40}
\begin{aligned}
&\partial_{t}\dot{\Delta}_{q}f+\xi\cdot\nabla_{x}\dot{\Delta}_{q}f+L\dot{\Delta}_{q}f=\dot{\Delta}_{q}\Gamma(f,f).
\end{aligned}
\end{equation}
Then, taking the $L^{2}_{\xi,x}$ inner product of \eqref{1dlg-E1.40} with $\dot{\Delta}_{q}f$ and employing Lemma \ref{1dlg-L2.6}, we have
\begin{equation}\label{1dlg-E5.12}
\begin{aligned}
&\frac{1}{2}\frac{d}{dt}\|\dot{\Delta}_{q}f\|^{2}_{L^{2}_{\xi}L^{2}_{x}}+\lambda_{0}\|\sqrt{\nu(\xi)}\dot{\Delta}_{q}\{\mathbf{I-P}\}f\|^{2}_{L^{2}_{\xi}L^{2}_{x}}\\
&\leq\Big|\Big(\dot{\Delta}_{q}\Gamma(f,f),\dot{\Delta}_{q}\{\mathbf{I-P}\}f\Big)_{\xi,x}\Big|,
\end{aligned}
\end{equation}
where $\nu(\xi)$ is given by $\eqref{nu}$, and we used the key fact that $\big( \dot{\Delta}_{q}\Gamma(f,f),\dot{\Delta}_{q}\mathbf{P}f \big)_{\xi,x}=0$.

To characterize the dissipation of $\dot{\Delta}_j\mathbf{P}f$, the frequency localization of \eqref{1dlg-E5.90} takes the form
\begin{equation}\label{1dlg-E5.9}
 \left\{
\begin{aligned}
&\partial_{t}\dot{\Delta}_{q}a+\nabla_{x}\cdot\dot{\Delta}_{q}b=0,\\
&\partial_{t}\dot{\Delta}_{q}b+\nabla_{x}\dot{\Delta}_{q}(a+2c)+\nabla_{x}\cdot\Theta(\dot{\Delta}_{q}\{\mathbf{I-P}\}f)=0,\\
&\partial_{t}\dot{\Delta}_{q}c+\frac{1}{3}\nabla_{x}\cdot\dot{\Delta}_{q} b+\frac{5}{3}\nabla_{x}\cdot\Lambda(\dot{\Delta}_{q}\{\mathbf{I-P}\}f)=0,\\
&\partial_{t}\big(\Theta_{ij}(\dot{\Delta}_{q}\{\mathbf{I-P}\}f)+2\dot{\Delta}_{q}c\,\delta_{ij}\big)+\partial_{i}\dot{\Delta}_{q}b_{j}+\partial_{j}\dot{\Delta}_{q}b_{i}\\
&\quad\quad+\Theta_{ij}(\mathbf{r}( \dot{\Delta}_{q}\{\mathbf{I-P}\}f))=\Theta_{ij}(\dot{\Delta}_{q}\mathbf{h}),\\
&\partial_{t}\Lambda_{i}(\dot{\Delta}_{q}\{\mathbf{I-P}\}f)+\partial_{i}\dot{\Delta}_{q}c+\Lambda_{i}(\mathbf{r}( \dot{\Delta}_{q}\{\mathbf{I-P}\}f))=\Lambda_{i}(\dot{\Delta}_{q}\mathbf{h}).\\
\end{aligned}
 \right.
\end{equation}
Taking the $L^{2}_{x}$ inner product of $(\ref{1dlg-E5.9})_{2}$ with $\dot{\Delta}_{q}\nabla_{x}a$ and utilizing $(\ref{1dlg-E5.9})_{1}$, we derive
\begin{equation}\label{1dlg-E5.13}
\begin{aligned}
&\frac{d}{dt}\sum_{i=1}^{3}\Big(\dot{\Delta}_{q}\partial_{i}a,\dot{\Delta}_{q}b_{i}\Big)_{x}+\|\dot{\Delta}_{q}\nabla_{x}a\|^{2}_{L^{2}_{x}}-\|\dot{\Delta}_{q}\nabla_{x}\cdot b\|^{2}_{L^{2}_{x}}\\
&\quad+2\Big(\dot{\Delta}_{q}\nabla_{x}c,\dot{\Delta}_{q}\nabla_{x}a\Big)_{x}+\Big(\nabla_{x}\cdot\Theta(\dot{\Delta}_{q}\{\mathbf{I-P}\}f),\nabla_{x}\dot{\Delta}_{q}a\Big)_{x}=0.
\end{aligned}
\end{equation}
Next, multiplying $(\ref{1dlg-E5.9})_{5}$ by $\dot{\Delta}_{q}\partial_{i}c$ and summing over $i$ for $1\leq i\leq 3$ and using $(\ref{1dlg-E5.9})_{3}$, we obtain
\begin{equation}\label{1dlg-E5.14}
\begin{aligned}
&\frac{d}{dt}\sum_{i=1}^{3}\Big(\dot{\Delta}_{q}\partial_{i}c,\Lambda_{i}(\dot{\Delta}_{q}\{\mathbf{I-P}\}f)\Big)_{x}+\|\dot{\Delta}_{q}\nabla_{x}c\|^{2}_{L^{2}_{x}}\\
&\quad\quad-\frac{1}{3}\Big(\nabla_{x}\cdot\Lambda(\dot{\Delta}_{q}\{\mathbf{I-P}\}f),\nabla_{x}\cdot\dot{\Delta}_{q}b\Big)_{x}\\
&\quad\quad-\frac{5}{3}\|\nabla\cdot \Lambda(\dot{\Delta}_q\{\mathbf{I-P}\}f)\|_{L^2_x}^2+\sum_{i=1}^{3}\Big(\Lambda_{i}(\mathbf{r}( \dot{\Delta}_{q}\{\mathbf{I-P}\}f)),\dot{\Delta}_{q}\partial_{i}c\Big)_{x}\\
&\quad\leq \sum_{i=1}^{3}\big|\big(\Lambda_{i}(\dot{\Delta}_{q}\mathbf{h}),\dot{\Delta}_{q}\partial_{i}c\big)_{x}\big|.
\end{aligned}
\end{equation}
To deduce the dissipation of $\dot{\Delta}_q b$, one may deduce from  $(\ref{1dlg-E5.9})$ that
\begin{equation}
\begin{aligned}\label{1dlg-E5.944}
&\partial_t \Theta_{ij}(\dot{\Delta}_q \{\mathbf{I-P}\}f)+\partial_i \dot{\Delta}_q b_j + \partial_j \dot{\Delta}_q b_i - \frac{2}{3} \nabla_x \cdot \dot{\Delta}_q b \, \delta_{ij}\\
&\quad- \frac{10}{3} \nabla_x \cdot \Lambda(\dot{\Delta}_q \{\mathbf{I-P}\}f) \, \delta_{ij}  + \Theta_{ij}(\mathbf{r}(\dot{\Delta}_q \{\mathbf{I-P}\}f)) = \Theta_{ij}(\dot{\Delta}_q \mathbf{h}).
\end{aligned}
\end{equation}
We multiply \eqref{1dlg-E5.944} by $\partial_{i}\dot{\Delta}_{q}b_{j}+\partial_{j}\dot{\Delta}_{q}b_{i}$, sum over $1\leq i,j\leq3$ and then use $(\ref{1dlg-E5.9})_{2}$ to eliminate $\dot{\Delta}_{q}\partial_{t}b$. Consequently, since  $\sum\limits_{i,j=1}^{3} \|\partial_{i}\dot{\Delta}_{q}b_{j}+\partial_{j}\dot{\Delta}_{q}b_{i}\|_{L^2_x}^2=2\|\nabla_x \dot{\Delta}_{q}b\|_{L^2_x}^2+2\|\nabla_x\cdot\dot{\Delta}_{q}b\|_{L^2_x}^2$, it holds that
\begin{equation}\label{1dlg-E5.15}
\begin{aligned}
&\frac{d}{dt} \sum_{i,j=1}^3 \Big( \partial_i \dot{\Delta}_q b_j + \partial_j \dot{\Delta}_q b_i, \Theta_{ij} \Big)_x \\
&\quad+ 2 \|\nabla_x \dot{\Delta}_q b\|_{L^2_x}^2 - \frac{4}{3} \|\nabla \cdot \dot{\Delta}_q b\|_{L^2_x}^2-\frac{20}{3} \Big( \nabla_x \cdot \dot{\Delta}_q b, \nabla_x \cdot \Lambda(\dot{\Delta}_q \{\mathbf{I-P}\}f) \Big)_x\\
&\quad - 2 \Big( \nabla \dot{\Delta}_q (a + 2c), \nabla \cdot \Theta ( \dot{\Delta}_{q}\{\mathbf{I-P}\}f) \Big)_x - 2  \|\nabla_x \Theta( \dot{\Delta}_{q}\{\mathbf{I-P}\}f)\|_{L^2_x}^2\\
&\quad+ \sum_{i,j=1}^3 \Big( \Theta_{ij}(\mathbf{r}( \dot{\Delta}_{q}\{\mathbf{I-P}\}f)), \partial_i \dot{\Delta}_q b_j + \partial_j \dot{\Delta}_q b_i \Big)_x\\
&\leq \sum_{i,j=1}^3 \Big| \Big( \Theta_{ij}(\dot{\Delta}_q \mathbf{h}), \partial_i \dot{\Delta}_q b_j + \partial_j \dot{\Delta}_q b_i \Big)_x \Big|.
\end{aligned}
\end{equation}
To proceed, we let $\eqref{1dlg-E5.12}+\eta_{1}\times \Big((\ref{1dlg-E5.13})+ 60 (\ref{1dlg-E5.14})+ 3(\ref{1dlg-E5.15})\Big)$.  This leads to
\begin{equation}\label{1dlg-E5.10}
\begin{aligned}
&\frac{d}{dt}\mathcal{L}_{1,q}(t)+\mathcal{D}_{1,q}(t)\\
&\quad\lesssim\Big|\Big(\dot{\Delta}_{q}\Gamma(f,f),\dot{\Delta}_{q}\{\mathbf{I-P}\}f\Big)_{\xi,x}\Big|\\
&\quad\quad+\eta_1\sum_{i=1}^{3}\Big|\Big(\Lambda_{i}(\dot{\Delta}_{q}\mathbf{h}),\dot{\Delta}_{q}\partial_{i}c\Big)_{x}\Big|\\
&\quad\quad+\eta_1\sum_{i,j=1}^{3}\Big|\Big(\Theta_{ij}(\dot{\Delta}_{q}\mathbf{h}),\partial_{i}\dot{\Delta}_{q}b_{j}+\partial_{j}\dot{\Delta}_{q}b_{i}\Big)_{x}\Big|,
\end{aligned}
\end{equation}
where the Lyapunov functional $\mathcal{L}_{1,q}(t)$ and the dissipation $\mathcal{D}_{1,q}(t)$ are, respectively, defined by
\begin{equation}\label{1dlg-E5.11}
\begin{aligned}
\mathcal{L}_{1,q}(t)&\triangleq \frac{1}{2}\|\dot{\Delta}_{q}f\|^{2}_{L^{2}_{\xi}L^{2}_{x}}\\
&\quad+\eta_{1}\bigg(\sum_{i=1}^{3}\Big(\dot{\Delta}_{q}\partial_{i}a,\dot{\Delta}_{q}b_{i}\Big)_{x}+60\sum_{i=1}^{3}\Big(\dot{\Delta}_{q}\partial_{i}c,\Lambda_{i}(\dot{\Delta}_{q}\{\mathbf{I-P}\}f)\Big)_{x}\\
&\quad+3\sum_{i,j=1}^{3}\Big(\partial_{i}\dot{\Delta}_{q}b_{j}+\partial_{j}\dot{\Delta}_{q}b_{i},\Theta_{ij}(\dot{\Delta}_{q}\{\mathbf{I-P}\}f)\Big)_{x}\bigg),\\
\end{aligned}
\end{equation}
and
\begin{equation}\label{1dlg-E5.110}
\begin{aligned}
\mathcal{D}_{1,q}(t)&\triangleq \lambda_{0}\|\sqrt{\nu(\xi)}\dot{\Delta}_{q}\{\mathbf{I-P}\}f\|^{2}_{L^{2}_{\xi}L^{2}_{x}}+\eta_{1}\|\dot{\Delta}_{q}\nabla_{x}(a,b,\sqrt{60} c)\|^{2}_{L^{2}_{x}}\\
&\quad+\eta_1\bigg( 2\Big(\dot{\Delta}_{q}\nabla_{x}c,\dot{\Delta}_{q}\nabla_{x}a\Big)_{x} -\Big(\nabla_{x}\cdot\Theta(\dot{\Delta}_{q}\{\mathbf{I-P}\}f),\nabla_{x}\dot{\Delta}_{q}(5a+12 c)\Big)_{x}\\
&\quad\quad-100 \|\nabla_x\cdot \Lambda_{i}(\dot{\Delta}_q\{\mathbf{I-P}\}f)\|_{L^2_x}^2-6\|\nabla_x\Theta( \dot{\Delta}_{q}\{\mathbf{I-P}\}f)\|_{L^2_x}^2\\
&\quad\quad+60\sum_{i=1}^{3}\Big(\Lambda_{i}(\mathbf{r}( \dot{\Delta}_{q}\{\mathbf{I-P}\}f)),\dot{\Delta}_{q}\partial_{i}c\Big)_{x}\\
&\quad\quad+3\sum_{i,j=1}^{3}\Big(\Theta_{ij}(\mathbf{r}( \dot{\Delta}_{q}\{\mathbf{I-P}\}f)),\partial_{i}\dot{\Delta}_{q}b_{j}+\partial_{j}\dot{\Delta}_{q}b_{i}\Big)_{x}\bigg)
\end{aligned}
\end{equation}
with $\eta_{1}\in(0,1)$ to be determined. It is clear that for any $q\leq0$, $\mathcal{L}_{1,q}(t)$ fulfills
\begin{equation}\label{1dlg-E5.18}
\frac{1}{2}(1-C\eta_{1})\|\dot{\Delta}_{q}f\|^{2}_{L^{2}_{\xi}L^{2}_{x}}\leq \mathcal{L}_{1,q}(t)\leq\frac{1}{2}(1+C\eta_{1})\|\dot{\Delta}_{q}f\|^{2}_{L^{2}_{\xi}L^{2}_{x}}.
\end{equation}
 In order to justify the coercivity property of $\mathcal{D}_{1,q}(t)$, we see that, for $q\leq0$,
\begin{align*}
\|\dot{\Delta}_{q}\nabla_{x}\mathbf{G}(\{\mathbf{I-P}\}f)\|^{2}_{L^{2}_{x}}&\leq C \|\dot{\Delta}_{q}\{\mathbf{I-P}\}f\|^{2}_{L^{2}_{\xi}L^{2}_{x}},\quad \mathbf{G}=\Lambda, \Theta,\\
2\Big|\Big(\dot{\Delta}_{q}\nabla_{x}c,\dot{\Delta}_{q}\nabla_{x}a\Big)_{x}\Big|&\leq  \frac{1}{4}\|\dot{\Delta}_q \nabla_x a\|_{L^2_x}^2+4\|\dot{\Delta}_q\nabla_x c\|_{L^2_x}^2,\\
\Big|\Big(\nabla_{x}\cdot\Theta(\dot{\Delta}_{q}\{\mathbf{I-P}\}f),\nabla_{x}\dot{\Delta}_{q}(5a+12c)\Big)_{x}\Big|&\leq C\|\dot{\Delta}_{q}\{\mathbf{I-P}\}f\|_{L^2_{\xi}L^2_x}^2+\frac{1}{16}\|\dot{\Delta}_q( \nabla_x a, \nabla_x c)\|_{L^2_x}^2
\end{align*}
for any $q\leq0$. Furthermore, since $\Theta_{ij}$ and $\Lambda_{i}$ given by \eqref{theta} can absorb any velocity weights,
we have
\begin{equation}\nonumber
\begin{aligned}
60 \sum_{i=1}^{3}\Big|\Big(\Lambda_{i}(\mathbf{r}( \dot{\Delta}_{q}\{\mathbf{I-P}\}f)),\dot{\Delta}_{q}\partial_{i}c\Big)_{x}\Big|&\leq C\|\dot{\Delta}_{q}\{\mathbf{I-P}\}f\|_{L^2_{\xi}L^2_{x}}^2+\|\dot{\Delta}_{q}\nabla_x c\|_{L^2_x}^2,
\end{aligned}
\end{equation}
and
\begin{equation}\nonumber
\begin{aligned}
&3\sum_{i,j=1}^{3}\Big|\Big(\Theta_{ij}(\mathbf{r}( \dot{\Delta}_{q}\{\mathbf{I-P}\}f)),\partial_{i}\dot{\Delta}_{q}b_{j}+\partial_{j}\dot{\Delta}_{q}b_{i}\Big)_{x}\Big|\\
&\quad\quad \leq C \|\dot{\Delta}_{q}\{\mathbf{I-P}\}f\|_{L^2_{\xi}L^2_{x}}^2+\frac{1}{4}\|\dot{\Delta}_{q}\nabla_x b\|_{L^2_x}^2.
\end{aligned}
\end{equation}
Gathering these estimates at hand, we arrive at
\begin{equation}\label{1dlg-E5.19}
\begin{aligned}
\mathcal{D}_{1,q}(t)&\geq (\lambda_0-C\eta_1)\|\sqrt{\nu(\xi)}\dot{\Delta}_{q}\{\mathbf{I-P}\}f\|^{2}_{L^{2}_{\xi}L^{2}_{x}}+\frac{1}{2}\eta_{1}\|\dot{\Delta}_{q}\nabla_x(a,b,c)\|^{2}_{L^{2}_{x}}.
\end{aligned}
\end{equation}
Using the fact that $\|\dot{\Delta}_{q}\mathbf{P}f\|_{L^2_{\xi}L^2_x}\sim \|\dot{\Delta}_{q}(a,b,c)\|_{L^2_x}$ and choosing 
$$
\eta_{1}=\min\bigg\{{1,\frac{1}{2C},\frac{\lambda_{0}}{2C}}\bigg\},
$$ 
we obtain from Lemma \ref{1dlg-L2.2}, (\ref{1dlg-E5.18}) and (\ref{1dlg-E5.19}) that, for any $q\leq0$,
\begin{equation}\label{1dlg-E5.21}
\mathcal{L}_{1,q}(t)\sim\|\dot{\Delta}_{q}f\|^{2}_{L^{2}_{\xi}L^{2}_{x}},
\end{equation}
and
\begin{equation}\label{1dlg-E5.22}
\mathcal{D}_{1,q}(t)\gtrsim2^{2q}\|\dot{\Delta}_{q}f\|^{2}_{L^{2}_{\xi}L^{2}_{x}}+\|\sqrt{\nu(\xi)}\dot{\Delta}_{q}\{\mathbf{I-P}\}f\|^{2}_{L^{2}_{\xi}L^{2}_{x}}.
\end{equation}
In accordance with (\ref{1dlg-E5.21})-(\ref{1dlg-E5.22}), the following Lyapunov inequality holds:
\begin{equation}\label{1dlg-E5.23}
\begin{aligned}
&\frac{d}{dt}\mathcal{L}_{1,q}(t)+2^{2q}\mathcal{L}_{1,q}(t)+\|\sqrt{\nu(\xi)}\dot{\Delta}_{q}\{\mathbf{I-P}\}f\|^{2}_{L^{2}_{\xi}L^{2}_{x}}\\
&\quad\lesssim|(\dot{\Delta}_{q}\Gamma(f,f),\dot{\Delta}_{q}\{\mathbf{I-P}\}f)_{\xi,x}|\\
&\quad\quad+\sum_{i=1}^{3}\|\Lambda_{i}(\dot{\Delta}_{q}\mathbf{h})\|^{2}_{L^{2}_{x}}+\sum_{i,j=1}^{3}\|\Theta_{ij}(\dot{\Delta}_{q}\mathbf{h})\|^{2}_{L^{2}_{x}}.
\end{aligned}
\end{equation}

Then, we establish the desired low-frequency estimates, which rely on the Lyapunov inequality \eqref{1dlg-E5.23}. By integrating \eqref{1dlg-E5.23} over the interval $[0,t]$ and taking the square root of both sides of the resulting inequality, we obtain
\begin{equation}\label{1dlg-E5.24}
\begin{aligned}
&\|\dot{\Delta}_{q}f\|_{L^{2}_{\xi}L^{2}_{x}}+2^{q}\bigg(\int_{0}^{t}\|\dot{\Delta}_{q}f\|^{2}_{L^{2}_{\xi}L^{2}_{x}}\,d\tau\bigg)^{1/2}\\
&\quad\quad+\bigg(\int_{0}^{t}\|\sqrt{\nu(\xi)}\dot{\Delta}_{q}\{\mathbf{I-P}\}f\|^{2}_{L^{2}_{\xi}L^{2}_{x}}\,d\tau\bigg)^{1/2}\\
&\quad\lesssim\|\dot{\Delta}_{q}f_{0}\|_{L^{2}_{\xi}L^{2}_{x}}+\bigg(\int_{0}^{t}|(\dot{\Delta}_{q}\Gamma(f,f),\dot{\Delta}_{q}\{\mathbf{I-P}\}f)_{\xi,x}|\,d\tau\bigg)^{1/2}\\
&\quad\quad+\sum_{i=1}^{3}\|\Lambda_{i}(\dot{\Delta}_{q}\mathbf{h})\|_{L^{2}_{t}L^{2}_{x}}+\sum_{i,j=1}^{3}\|\Theta_{ij}(\dot{\Delta}_{q}\mathbf{h})\|_{L^{2}_{t}L^{2}_{x}}.
\end{aligned}
\end{equation}
Multiplying (\ref{1dlg-E5.24}) by $2^{q/2}$, taking an upper bound on $[0,t]$ and summing over $q\leq0$, we arrive at
\begin{equation}\label{1dlg-E5.25}
\begin{aligned}
&\|f\|^{\ell}_{\widetilde{L}^{\infty}_{t}\widetilde{L}^{2}_{\xi}(\dot{B}^{1/2}_{2,1})}+\|f\|^{\ell}_{\widetilde{L}^{2}_{t}\widetilde{L}^{2}_{\xi}(\dot{B}^{3/2}_{2,1})}+\|\{\mathbf{I-P}\}f\|^{\ell}_{\widetilde{L}^{2}_{t}\widetilde{L}^{2}_{\xi,\nu}(\dot{B}^{1/2}_{2,1})}\\
&\quad\lesssim\|f_{0}\|^{\ell}_{\widetilde{L}^{2}_{\xi}(\dot{B}^{1/2}_{2,1})}+\sum_{q\leq0}2^{\frac{q}{2}}\bigg(\int_{0}^{t}\big|\big(\dot{\Delta}_{q}\Gamma(f,f),\dot{\Delta}_{q}\{\mathbf{I-P}\}f\big)_{\xi,x}\big|\,d\tau\bigg)^{1/2}\\
&\quad\quad+\sum_{i=1}^{3}\sum_{q\leq0}2^{\frac{q}{2}}\|\Lambda_{i}(\dot{\Delta}_{q}\mathbf{h})\|_{L^{2}_{t}L^{2}_{x}}+\sum_{i,j=1}^{3}\sum_{q\leq0}2^{\frac{q}{2}}\|\Theta_{ij}(\dot{\Delta}_{q}\mathbf{h})\|_{L^{2}_{t}L^{2}_{x}}.
\end{aligned}
\end{equation}
The nonlinear terms on the right-hand side of \eqref{1dlg-E5.25} can be estimated as follows. It follows from the trilinear estimate \eqref{1dlg-E2.13} in Lemma \ref{1dlg-L2.7}  (choosing $h=\{\mathbf{I-P}\}f, g=f$, $s_{1}=1/2$ and $s_{2}=3/2$)  that
\begin{equation}\label{1dlg-E5.260}
\begin{aligned}
&\sum_{q\leq0}2^{\frac{q}{2}}\bigg(\int_{0}^{t}\big|\big(\dot{\Delta}_{q}\Gamma(f,f),\dot{\Delta}_{q}\{\mathbf{I-P}\}f\big)_{\xi,x}\big|d\tau\bigg)^{1/2}\\
&\quad\lesssim \|f\|^{1/2}_{\widetilde{L}^{\infty}_{t}\widetilde{L}^{2}_{\xi}(\dot{B}^{1/2}_{2,1})}\|f\|^{1/2}_{\widetilde{L}^{2}_{t}\widetilde{L}^{2}_{\xi,\nu}(\dot{B}^{3/2}_{2,1})}\|\{\mathbf{I-P}\}f\|^{1/2}_{\widetilde{L}^{2}_{t}\widetilde{L}^{2}_{\xi,\nu}(\dot{B}^{1/2}_{2,1})}.
\end{aligned}
\end{equation}
Based on the macro-micro decomposition and the frequency cut-off properties stated in \eqref{1dlg-E2.6}, one can get 
\begin{align}
\|f\|_{\widetilde{L}^{\infty}_{t}\widetilde{L}^{2}_{\xi}(\dot{B}^{1/2}_{2,1})}&\lesssim \|f\|_{\widetilde{L}^{\infty}_{t}\widetilde{L}^{2}_{\xi}(\dot{B}^{1/2}_{2,1})}^{\ell}+\|f\|_{\widetilde{L}^{\infty}_{t}\widetilde{L}^{2}_{\xi}(\dot{B}^{3/2}_{2,1})}^{h},\label{1mm}\\
\|f\|_{\widetilde{L}^{2}_{t}\widetilde{L}^{2}_{\xi,\nu}(\dot{B}^{3/2}_{2,1})}&\lesssim \|\mathbf{P}f\|_{\widetilde{L}^{2}_{t}\widetilde{L}^{2}_{\xi}(\dot{B}^{3/2}_{2,1})}^{\ell} +\|\{\mathbf{I-P}\}f\|_{\widetilde{L}^{2}_{t}\widetilde{L}^{2}_{\xi,\nu}(\dot{B}^{1/2}_{2,1})}^{\ell}+\|f\|_{\widetilde{L}^{2}_{t}\widetilde{L}^{2}_{\xi,\nu}(\dot{B}^{3/2}_{2,1})}^{h}\label{2mm}
\end{align}
and
\begin{equation}\label{3mm}
\begin{aligned}
&\|\{\mathbf{I-P}\}f\|_{\widetilde{L}^{2}_{t}\widetilde{L}^{2}_{\xi,\nu}(\dot{B}^{1/2}_{2,1})}\\
&\quad\lesssim \|\{\mathbf{I-P}\}f\|_{\widetilde{L}^{2}_{t}\widetilde{L}^{2}_{\xi,\nu}(\dot{B}^{1/2}_{2,1})}^{\ell}+\|\{\mathbf{I-P}\}f\|_{\widetilde{L}^{2}_{t}\widetilde{L}^{2}_{\xi,\nu}(\dot{B}^{3/2}_{2,1})}^{h}.
\end{aligned}
\end{equation}
Putting \eqref{1mm}-\eqref{3mm} into \eqref{1dlg-E5.260} yields
\begin{equation}\label{1dlg-E5.26}
\begin{aligned}
&\sum_{q\leq0}2^{\frac{q}{2}}\bigg(\int_{0}^{t}\big|\big(\dot{\Delta}_{q}\Gamma(f,f),\dot{\Delta}_{q}\{\mathbf{I-P}\}f\big)_{\xi,x}\big|dt\bigg)^{1/2}\lesssim \sqrt{\mathcal{E}_{t}(f)}\mathcal{D}_{t}(f).
\end{aligned}
\end{equation}
Remembering the notations \eqref{theta}-\eqref{r} and employing Lemma \ref{1dlg-L2.9}, we infer that
\begin{equation}\label{1dlg-E5.27}
\begin{aligned}
&\sum_{i=1}^{3}\sum_{q\leq0}2^{\frac{q}{2}}\|\Lambda_{i}(\dot{\Delta}_{q}\mathbf{h})\|_{L^{2}_{t}L^{2}_{x}}+\sum_{i,j=1}^{3}\sum_{q\leq0}2^{\frac{q}{2}}\|\Theta_{ij}(\dot{\Delta}_{q}\mathbf{h})\|_{L^{2}_{t}L^{2}_{x}}\\
&\quad\lesssim\|f\|_{\widetilde{L}^{\infty}_{t}\widetilde{L}^{2}_{\xi}(\dot{B}^{1/2}_{2,1})}\|f\|_{\widetilde{L}^{2}_{t}\widetilde{L}^{2}_{\xi}(\dot{B}^{3/2}_{2,1})}\\
&\quad\lesssim\mathcal{E}_{t}(f)\mathcal{D}_{t}(f).
\end{aligned}
\end{equation} 
By inserting all the above estimates into (\ref{1dlg-E5.25}), we arrive at (\ref{1dlg-E5.17}).
\end{proof}

\subsubsection{High-frequency analysis}
In this subsection, we establish the a priori estimates of solutions to the Cauchy problem (\ref{1dlg-E1.4}) in the high-frequency regime. 

\begin{lemma}\label{1dlg-R5.5}
Assume that  $f$ is the solution to the Cauchy problem $\rm(\ref{1dlg-E1.4})$ on any interval $[0,T)$.  Then for $0<t<T$, it holds that
\begin{equation}\label{1dlg-E5.30}
\begin{aligned}
&\|f\|^{h}_{\widetilde{L}^{\infty}_{t}\widetilde{L}^{2}_{\xi}(\dot{B}^{3/2}_{2,1})}+\|\mathbf{P}f\|^{h}_{\widetilde{L}^{2}_{t}\widetilde{L}^{2}_{\xi}(\dot{B}^{3/2}_{2,1})}+\|\{\mathbf{I-P}\}f\|^{h}_{\widetilde{L}^{2}_{t}\widetilde{L}^{2}_{\xi}(\dot{B}^{3/2}_{2,1})}\\
&\quad\leq C\|f_{0}\|_{\widetilde{L}^{2}_{\xi}(\dot{B}^{3/2}_{2,1})}^{h}+C\Big(\sqrt{\mathcal{E}_{t}(f)}+\mathcal{E}_{t}(f)\Big)\mathcal{D}_{t}(f),
\end{aligned}
\end{equation}
where $\mathcal{E}_{t}(f)$ and $\mathcal{D}_{t}(f)$ are defined by $\rm(\ref{1dlg-E1.16})$ and $\rm(\ref{1dlg-E1.17})$, respectively, and $C$ is a generic constant independent of $T$.
\end{lemma}
\begin{proof}
Let the constant $\eta_{2}\in(0,1)$ be determined later. For any $q\geq -1$, we define the Lyapunov functional 
\begin{equation}\label{1dlg-E5.29}
\begin{aligned}
\mathcal{L}_{2,q}(t)&\triangleq \frac{1}{2}\|\dot{\Delta}_{q}f\|^{2}_{L^{2}_{\xi}L^{2}_{x}}\\
&\quad+\eta_{2}2^{-2q}\bigg(\sum_{i=1}^{3}\Big(\dot{\Delta}_{q}\partial_{i}a,\dot{\Delta}_{q}b_{i}\Big)_{x}+60\sum_{i=1}^{3}\Big(\dot{\Delta}_{q}\partial_{i}c,\Lambda_{i}(\dot{\Delta}_{q}\{\mathbf{I-P}\}f)\Big)_{x}\\
&\quad+3\sum_{i,j=1}^{3}\Big(\partial_{i}\dot{\Delta}_{q}b_{j}+\partial_{j}\dot{\Delta}_{q}b_{i},\Theta_{ij}(\dot{\Delta}_{q}\{\mathbf{I-P}\}f)+2\dot{\Delta}_{q}c\,\delta_{ij}\Big)_{x}\bigg),\\
\end{aligned}
\end{equation}
and the dissipation functional
\begin{equation}\label{1dlg-E5.2900}
\begin{aligned}
\mathcal{D}_{2,q}(t)&\triangleq \lambda_{0}\|\sqrt{\nu(\xi)}\dot{\Delta}_{q}\{\mathbf{I-P}\}f\|^{2}_{L^{2}_{\xi}L^{2}_{x}}+\eta_{2}2^{-2q}\|\dot{\Delta}_{q}\nabla_{x}(a,b, \sqrt{60} c)\|^{2}_{L^{2}_{x}}\\
&\quad+\eta_{2}2^{-2q}\bigg( 2\Big(\dot{\Delta}_{q}\nabla_{x}c,\dot{\Delta}_{q}\nabla_{x}a\Big)_{x} -\Big(\nabla_{x}\cdot\Theta(\dot{\Delta}_{q}\{\mathbf{I-P}\}f),\nabla_{x}\dot{\Delta}_{q}(5a+12 c)\Big)_{x}\\
&\quad\quad-100 \|\nabla\cdot \Lambda(\dot{\Delta}_q\{\mathbf{I-P}\}f)\|_{L^2_x}^2-6\|\nabla\cdot\Theta( \dot{\Delta}_{q}\{\mathbf{I-P}\}f)\|_{L^2_x}^2\\
&\quad\quad+60\sum_{i=1}^{3}\Big(\Lambda_{i}(\mathbf{r}( \dot{\Delta}_{q}\{\mathbf{I-P}\}f)),\dot{\Delta}_{q}\partial_{i}c\Big)_{x}\\
&\quad\quad+3\sum_{i,j=1}^{3}\Big(\Theta_{ij}(\mathbf{r}( \dot{\Delta}_{q}\{\mathbf{I-P}\}f)),\partial_{i}\dot{\Delta}_{q}b_{j}+\partial_{j}\dot{\Delta}_{q}b_{i}\Big)_{x}\bigg).
\end{aligned}
\end{equation}
With the aid of \eqref{1dlg-E5.12}, \eqref{1dlg-E5.13}, \eqref{1dlg-E5.14} and \eqref{1dlg-E5.15}, we obtain 
\begin{equation}\label{1dlg-E5.28}
\begin{aligned}
&\frac{d}{dt}\mathcal{L}_{2,q}(t)+\mathcal{D}_{2,q}(t)\\
&\lesssim\Big|\Big(\dot{\Delta}_{q}\Gamma(f,f),\dot{\Delta}_{q}\{\mathbf{I-P}\}f\Big)_{\xi,x}\Big|\\
&\quad+\eta_{2}2^{-2q}\sum_{i=1}^{3}\Big|\Big(\Lambda_{i}(\dot{\Delta}_{q}\mathbf{h}),\dot{\Delta}_{q}\partial_{i}c\Big)_{x}\Big|\\
&\quad+\eta_{2}2^{-2q}\sum_{i,j=1}^{3}\Big|\Big(\Theta_{ij}(\dot{\Delta}_{q}\mathbf{h}),\partial_{i}\dot{\Delta}_{q}b_{j}+\partial_{j}\dot{\Delta}_{q}b_{i}\Big)_{x}\Big|.
\end{aligned}
\end{equation}
Arguing similarly as in the proof of Lemma \ref{1dlg-R5.3}, we can verify that, for any $q\geq-1$,
\begin{equation}\label{1dlg-E5.31}
\frac{1}{2}(1-C\eta_{2})\|\dot{\Delta}_{q}f\|^{2}_{L^{2}_{\xi}L^{2}_{x}}\leq \mathcal{L}_{2,q}(t)\leq\frac{1}{2}(1+C\eta_{2})\|\dot{\Delta}_{q}f\|^{2}_{L^{2}_{\xi}L^{2}_{x}}
\end{equation}
and
\begin{equation}\label{1dlg-E5.32}
\begin{aligned}
\mathcal{D}_{2,q}(t)&\geq (\lambda_{0}-C\eta_{2})\|\sqrt{\nu(\xi)}\dot{\Delta}_{q}\{\mathbf{I-P}\}f\|^{2}_{L^{2}_{\xi}L^{2}_{x}}\\
&\quad+\frac{1}{4}2^{-2q}\eta_{2} \|\dot{\Delta}_{q}\nabla_x(a,b,c)\|^{2}_{L^{2}_{x}}.
\end{aligned}
\end{equation}
Thus, one can choose a sufficiently small constant $\eta_2>0$ in  (\ref{1dlg-E5.31})-(\ref{1dlg-E5.32}) and take advantage of Lemma \ref{1dlg-L2.2} so that
\begin{equation}\label{1dlg-E5.33}
\mathcal{L}_{2,q}(t)\sim\|\dot{\Delta}_{q}f\|^{2}_{L^{2}_{\xi}L^{2}_{x}}
\end{equation}
and
\begin{equation}\label{1dlg-E5.34}
\mathcal{D}_{2,q}(t)\gtrsim\|\dot{\Delta}_{q}f\|^{2}_{L^{2}_{\xi}L^{2}_{x}}+\|\sqrt{\nu(\xi)}\dot{\Delta}_{q}\{\mathbf{I-P}\}f\|^{2}_{L^{2}_{\xi}L^{2}_{x}}.
\end{equation}
Combining (\ref{1dlg-E5.28}) with (\ref{1dlg-E5.33})-(\ref{1dlg-E5.34}), we have 
\begin{equation}\label{1dlg-E5.35}
\begin{aligned}
&\frac{d}{dt}\mathcal{L}_{2,q}(t)+\mathcal{L}_{2,q}(t)+\|\sqrt{\nu(\xi)}\dot{\Delta}_{q}\{\mathbf{I-P}\}f\|^{2}_{L^{2}_{\xi}L^{2}_{x}}\\
&\quad\lesssim|(\dot{\Delta}_{q}\Gamma(f,f),\dot{\Delta}_{q}\{\mathbf{I-P}\}f)_{\xi,x}|\\
&\quad\quad+2^{-2q}\sum_{i=1}^{3}\|\Lambda_{i}(\dot{\Delta}_{q}\mathbf{h})\|^{2}_{L^{2}_{x}}+2^{-2q}\sum_{i,j=1}^{3}\|\Theta_{ij}(\dot{\Delta}_{q}\mathbf{h})\|^{2}_{L^{2}_{x}}
\end{aligned}
\end{equation}
for $q\geq-1$, which leads to
\begin{equation}\label{1dlg-E5.36}
\begin{aligned}
&\|f\|^{h}_{\widetilde{L}^{\infty}_{t}\widetilde{L}^{2}_{\xi}(\dot{B}^{3/2}_{2,1})}+\|f\|^{h}_{\widetilde{L}^{2}_{t}\widetilde{L}^{2}_{\xi}(\dot{B}^{3/2}_{2,1})}+\|\{\mathbf{I-P}\}f\|^{h}_{\widetilde{L}^{2}_{t}\widetilde{L}^{2}_{\xi,\nu}(\dot{B}^{3/2}_{2,1})}\\
&\quad\lesssim\|f_{0}\|^{h}_{\widetilde{L}^{2}_{\xi}(\dot{B}^{3/2}_{2,1})}+\sum_{q\geq-1}2^{\frac{3q}{2}}\bigg(\int_{0}^{t}\big|\big(\dot{\Delta}_{q}\Gamma(f,f),\dot{\Delta}_{q}\{\mathbf{I-P}\}f\big)_{\xi,x}\big|d\tau\bigg)^{1/2}\\
&\quad\quad+\sum_{i=1}^{3}\sum_{q\geq-1}2^{\frac{q}{2}}\|\Lambda_{i}(\dot{\Delta}_{q}\mathbf{h})\|_{L^{2}_{t}L^{2}_{x}}+\sum_{i,j=1}^{3}\sum_{q\geq-1}2^{\frac{q}{2}}\|\Theta_{ij}(\dot{\Delta}_{q}\mathbf{h})\|_{L^{2}_{t}L^{2}_{x}}.
\end{aligned}
\end{equation}
The trilinear estimate \eqref{1dlg-E2.13} (choosing $g=f$, $h=\{\mathbf{I-P}\}f$ and $s_1=s_2=3/2$) ensures that
\begin{equation}\label{1dlg-E5.37}
\begin{aligned}
&\sum_{q\leq0}2^{\frac{3q}{2}}\bigg(\int_{0}^{t}\big|\big(\dot{\Delta}_{q}\Gamma(f,f),\dot{\Delta}_{q}\{\mathbf{I-P}\}f\big)_{\xi,x}\big|d\tau\bigg)^{1/2}\\
&\quad\lesssim \|f\|^{1/2}_{\widetilde{L}^{\infty}_{t}\widetilde{L}^{2}_{\xi}(\dot{B}^{3/2}_{2,1})}\|f\|^{1/2}_{\widetilde{L}^{2}_{t}\widetilde{L}^{2}_{\xi}(\dot{B}^{3/2}_{2,1})}\|\{\mathbf{I-P}\}f\|^{1/2}_{\widetilde{L}^{2}_{t}\widetilde{L}^{2}_{\xi,\nu}(\dot{B}^{3/2}_{2,1})}\\
&\quad\lesssim \sqrt{\mathcal{E}_{t}(f)}\mathcal{D}_{t}(f),
\end{aligned}
\end{equation}
where we have used \eqref{2mm},
\begin{equation}\label{4mm}
\begin{aligned}
\|f\|_{\widetilde{L}^{\infty}_{t}\widetilde{L}^{2}_{\xi}(\dot{B}^{3/2}_{2,1})}&\lesssim \|f\|_{\widetilde{L}^{\infty}_{t}\widetilde{L}^{2}_{\xi}(\dot{B}^{1/2}_{2,1})}^{\ell}+\|f\|_{\widetilde{L}^{\infty}_{t}\widetilde{L}^{2}_{\xi}(\dot{B}^{3/2}_{2,1})}^{h}
\end{aligned}
\end{equation}
and
\begin{equation}\label{5mm}
\begin{aligned}
&\|\{\mathbf{I-P}\}f\|_{\widetilde{L}^{2}_{t}\widetilde{L}^{2}_{\xi,\nu}(\dot{B}^{3/2}_{2,1})}\\
&\quad\lesssim \|\{\mathbf{I-P}\}f\|_{\widetilde{L}^{2}_{t}\widetilde{L}^{2}_{\xi,\nu}(\dot{B}^{1/2}_{2,1})}^{\ell}+\|\{\mathbf{I-P}\}f\|_{\widetilde{L}^{2}_{t}\widetilde{L}^{2}_{\xi,\nu}(\dot{B}^{3/2}_{2,1})}^{h}.
\end{aligned}
\end{equation}
Furthermore, by applying Lemma \ref{1dlg-L2.9}, along with the equations \eqref{2mm} and \eqref{4mm}, we  obtain
\begin{equation}\label{1dlg-E5.38}
\begin{aligned}
&\sum_{i=1}^{3}\sum_{q\geq-1}2^{\frac{q}{2}}\|\Lambda_{i}(\dot{\Delta}_{q}\mathbf{h})\|_{L^{2}_{t}L^{2}_{x}}+\sum_{i,j=1}^{3}\sum_{q\geq-1}2^{\frac{q}{2}}\|\Theta_{ij}(\dot{\Delta}_{q}\mathbf{h})\|_{L^{2}_{t}L^{2}_{x}}\\
&\quad\lesssim\|f\|_{\widetilde{L}^{\infty}_{t}\widetilde{L}^{2}_{\xi}(\dot{B}^{3/2}_{2,1})}\|f\|_{\widetilde{L}^{2}_{t}\widetilde{L}^{2}_{\xi}(\dot{B}^{3/2}_{2,1})}\\
&\quad\lesssim\mathcal{E}_{t}(f)\mathcal{D}_{t}(f).
\end{aligned}
\end{equation}
Thus, the combination of  (\ref{1dlg-E5.36}), \eqref{1dlg-E5.37} and \eqref{1dlg-E5.38} leads to (\ref{1dlg-E5.30}). The proof of Lemma \ref{1dlg-R5.5} is complete.
\end{proof}

\subsection{Proof of Theorem \ref{1dlg-T1.2}}
Assume that the initial data $f_{0}$ satisfies $f_0\in \widetilde{L}^{2}_{\xi}(\dot{B}^{1/2}_{2,1})\cap \widetilde{L}^{2}_{\xi}(\dot{B}^{3/2}_{2,1})$ and \eqref{1dlg-E1.18}, where the constant $\varepsilon_0$ will be determined later. There exists a suitably small $\var_1>0$ and a maximal time $T^*=T^*(\var_1)$ such that for any $\var_0\leq \var_1$, the Cauchy problem (\ref{1dlg-E1.4}) admits a unique solution $f(t,x,\xi)$ on $[0,T)$ satisfying 
$$
f\in \widetilde{L}^{\infty}(0,T;\widetilde{L}^{2}_{\xi}(\dot{B}^{1/2}_{2,1})\cap\widetilde{L}^{2}_{\xi}(\dot{B}^{3/2}_{2,1}))\cap \widetilde{L}^2(0,T;\widetilde{L}^{2}_{\xi,\nu}(\dot{B}^{1/2}_{2,1})\cap \widetilde{L}^{2}_{\xi,\nu}(\dot{B}^{3/2}_{2,1}))
$$
for any $0<T<T^*$ and that
$$
E_t(f)\triangleq\|f(t)\|_{\widetilde{L}^{2}_{\xi}(\dot{B}^{1/2}_{2,1})\cap \widetilde{L}^{2}_{\xi}(\dot{B}^{3/2}_{2,1})}~~\text{is continuous over $[0,T^*)$}.
$$ 
The construction of an approximate sequence relies on using the Hahn-Banach extension theorem in the inhomogeneous space $\widetilde{L}^{\infty}_t\widetilde{L}^2_\xi(B^{3/2}_{2,1})$ (see \cite{Alexandre-2011, Morimoto-2016}), where the smallness in $L^2_\xi L^{\infty}_x$ can be ensured by the homogeneous $\widetilde{L}^2_\xi(\dot{B}^{3/2}_{2,1})$-control. The local existence is then completed by performing a compactness argument. Since the proof is quite standard, the details are omitted here. 

Our goal is to show $T^*=+\infty$ under the condition \eqref{1dlg-E1.18}. Define
\begin{equation*}
\widetilde{T}\triangleq \sup\bigg\{t\in[0,T^*)\,:\,E_t(f)+\mathcal{D}_{t}(f)\leq 4 C_1 \Big(\|f_{0}\|^{\ell}_{\widetilde{L}^{2}_{\xi}(\dot{B}^{1/2}_{2,1})}+\|f_{0}\|^{h}_{\widetilde{L}^{2}_{\xi}(\dot{B}^{3/2}_{2,1})}\Big)\bigg\},
\end{equation*}
where $\mathcal{D}_{t}(f)$, as defined in \eqref{1dlg-E1.17}, is continuous over $[0,T^*)$, and $C_1$ is given by proposition \ref{1dlg-P5.1}. Clearly, we have
$0<\widetilde{T}\leq T^*$. 

We claim that $\widetilde{T}=T^*$. Indeed, if $\widetilde{T}<T^*$, then by Lemma \ref{1dlg-L2.06} and the uniform {\emph{a priori}} estimates obtained in Proposition \ref{1dlg-P5.1}, it follows that for all $0<t<\widetilde{T}$,
\begin{equation}\nonumber
\begin{aligned}
&\mathcal{E}_{t}(f)+\mathcal{D}_{t}(f)\\
&\quad\leq C_{1}\Big(\|f_{0}\|^{\ell}_{\widetilde{L}^{2}_{\xi}(\dot{B}^{1/2}_{2,1})}+\|f_{0}\|^{h}_{\widetilde{L}^{2}_{\xi}(\dot{B}^{3/2}_{2,1})}\Big)\\
&\quad\quad+4 C_{1}\Big(\mathcal{E}_{t}(f)+\sqrt{\mathcal{E}_{t}(f)}\Big) \Big(\|f_{0}\|^{\ell}_{\widetilde{L}^{2}_{\xi}(\dot{B}^{1/2}_{2,1})}+\|f_{0}\|^{h}_{\widetilde{L}^{2}_{\xi}(\dot{B}^{3/2}_{2,1})}\Big)\\
&\quad\leq \frac{3}{2} C_{1}\Big(\|f_{0}\|^{\ell}_{\widetilde{L}^{2}_{\xi}(\dot{B}^{1/2}_{2,1})}+\|f_{0}\|^{h}_{\widetilde{L}^{2}_{\xi}(\dot{B}^{3/2}_{2,1})}\Big)+6 C_{1}\Big(\|f_{0}\|^{\ell}_{\widetilde{L}^{2}_{\xi}(\dot{B}^{1/2}_{2,1})}+\|f_{0}\|^{h}_{\widetilde{L}^{2}_{\xi}(\dot{B}^{3/2}_{2,1})}\Big)\mathcal{E}_{t}(f),
\end{aligned}
\end{equation}
which implies
\begin{equation}\label{1dlg-E5.41}
\begin{aligned}
E_t(f)+\mathcal{D}_{t}(f)\leq \mathcal{E}_{t}(f)+\mathcal{D}_{t}(f)\leq 3C_{1}\Big(\|f_{0}\|^{\ell}_{\widetilde{L}^{2}_{\xi}(\dot{B}^{1/2}_{2,1})}+\|f_{0}\|^{h}_{\widetilde{L}^{2}_{\xi}(\dot{B}^{3/2}_{2,1})}\Big)
\end{aligned}
\end{equation}
provided that in \eqref{1dlg-E1.18} we choose
$$
\varepsilon_0\triangleq\min\left\{\var_1,\frac{1}{12C_{1}}\right\}.
$$
By the classical continuity argument, this contradicts the definition of $\widetilde{T}$.  The claim is thus proved.

Next, we assume $\widetilde{T}=T^*<+\infty$. Since $f$ has the uniform estimate in (\ref{1dlg-E5.41}), we can take some $t$ sufficiently close to $T^*$ as the new initial data and apply the local-in-time existence result again, allowing the existence time of the solution to exceed $T^*$. This contradicts the maximality of $T^*$. Hence, $\widetilde{T}=T=+\infty$, and $f$ is indeed a global solution to the Cauchy problem \eqref{1dlg-E1.4} satisfying the uniform estimate \eqref{1dlg-E1.19}. Furthermore, the positivity of $F=\mu+\mu^{1/2}f$ can be found in, for instance, \cite{Guo-2003}.

To complete the proof of Theorem \ref{1dlg-T1.2}, we give the proof of uniqueness. Let $f_1, f_2$ be two solutions to the Cauchy problem \eqref{1dlg-E1.4} with the same initial data $f_0$ on $[0,T]$ satisfying \eqref{1dlg-E1.19}. Taking the difference $\widetilde{f}=f_1-f_2$ of the Boltzmann equation (\ref{1dlg-E1.4}) for $f_1$ and $f_2$, we obtain
\begin{equation}\label{1dlg-E3.8}
\partial_{t}\widetilde{f}+\xi\cdot\nabla_{x}\widetilde{f}+L\widetilde{f}=\Gamma(\widetilde{f},f_1)+\Gamma(f_2,\widetilde{f}).
\end{equation}
Applying $\dot{\Delta}_{q}$ to (\ref{1dlg-E3.8}) and taking the inner product of the resulting equation with $\Dot{\Delta}_{q}\widetilde{f}$ over $\mathbb{R}^{3}_{x}\times\mathbb{R}^{3}_{\xi}$ yields
\begin{equation}\label{1dlg-E3.9}
\begin{aligned}
&\frac{1}{2}\frac{d}{dt}\|\Dot{\Delta}_{q}\widetilde{f}\|^{2}_{L^{2}_{\xi}L^{2}_{x}}+\lambda_0\|\Dot{\Delta}_{q}\{\mathbf{I-P}\}\widetilde{f}\|^{2}_{L^{2}_{\xi,\nu}L^{2}_{x}}\\
&\quad\leq\Big|\Big(\Dot{\Delta}_{q}\Gamma(\widetilde{f},f_1),\dot{\Delta}_{q}\widetilde{f}\Big)_{\xi,x}\Big|+\Big|\Big(\Dot{\Delta}_{q}\Gamma(f_2,\widetilde{f}),\dot{\Delta}_{q}\widetilde{f}\Big)_{\xi,x}\Big|.
\end{aligned}
\end{equation}
Integrating (\ref{1dlg-E3.9}) with respect to the time variable over $[0,t]$ with $0\leq t\leq T_1\leq T$, and taking the square root of both sides of the resulting inequality, we have
\begin{align*}
\begin{aligned}
&
\sup_{t\in[0,T_1]}\|\Dot{\Delta}_{q}\widetilde{f}\|_{L^{2}_{\xi}L^{2}_{x}}+
\sqrt{\lambda_0}\|\Dot{\Delta}_{q}\{\mathbf{I-P}\}\widetilde{f}\|_{L^{2}_{\xi,\nu}L^{2}_{x}}\\
&\quad\leq \bigg(\int_{0}^{T_1}\bigg|\Big(\Dot{\Delta}_{q}\Gamma(\widetilde{f},f_1),\Dot{\Delta}_{q}\widetilde{f}\Big)_{\xi,x}\Big|dt\bigg)^{1/2}+ \bigg(\int_{0}^{T_1}\Big|\Big(\Dot{\Delta}_{q}\Gamma(f_2,\widetilde{f}),\Dot{\Delta}_{q}\widetilde{f}\Big)_{\xi,x}\Big|dt\bigg)^{1/2},
\end{aligned}
\end{align*}
from which we infer that
\begin{equation}\label{1dlg-E6.46}
\begin{aligned}
&\|\widetilde{f}\|_{\widetilde{L}^{\infty}_{T_1}\widetilde{L}^{2}_{\xi}(\dot{B}^{1/2}_{2,1})}+\|\{\mathbf{I-P}\}\widetilde{f}\|_{\widetilde{L}^{2}_{T_1}\widetilde{L}^{2}_{\xi,\nu}(\dot{B}^{1/2}_{2,1})}\\
&\quad\leq C\sum_{q\in\mathbb{Z}}2^{\frac{q}{2}}\Big(\int_{0}^{T_1}\bigg|\Big(\Dot{\Delta}_{q}\Gamma(\widetilde{f},f_1),\Dot{\Delta}_{q}\widetilde{f}\Big)_{\xi,x}\Big|dt\bigg)^{1/2}\\
&\quad\quad+C\sum_{q\in\mathbb{Z}}2^{\frac{q}{2}}\bigg(\int_{0}^{T_1}\Big|\Big(\Dot{\Delta}_{q}\Gamma(f_2,\widetilde{f}),\Dot{\Delta}_{q}\widetilde{f}\Big)_{\xi,x}\Big|dt\bigg)^{1/2}.
\end{aligned}
\end{equation}
The trilinear estimate in Lemma \ref{1dlg-L2.7} with $s_1=1/2$ and $s_2=3/2$ ensures that
\begin{flalign*}
\begin{aligned}
&\sum_{q\in\mathbb{Z}}2^{\frac{q}{2}}\Big(\int_{0}^{T_1}\bigg|\Big(\Dot{\Delta}_{q}\Gamma(\widetilde{f},f_1),\Dot{\Delta}_{q}\widetilde{f}\Big)_{\xi,x}\Big|dt\bigg)^{1/2}\\
&\quad \leq C\|\widetilde{f}\|^{1/2}_{\widetilde{L}^{2}_{T_1}\widetilde{L}^{2}_{\xi,\nu}(\dot{B}^{1/2}_{2,1})}\\
&\quad\quad\times\Big(\|\widetilde{f}\|^{1/2}_{\widetilde{L}^{\infty}_{T_1}\widetilde{L}^{2}_{\xi}(\dot{B}^{1/2}_{2,1})}\|f_1\|^{1/2}_{\widetilde{L}^{2}_{T_1}\widetilde{L}^{2}_{\xi,\nu}(\dot{B}^{3/2}_{2,1})}+\|\widetilde{f}\|^{1/2}_{\widetilde{L}^{2}_{T_1}\widetilde{L}^{2}_{\xi,\nu}(\dot{B}^{3/2}_{2,1})}\|f_1\|^{1/2}_{\widetilde{L}^{\infty}_{T_1}\widetilde{L}^{2}_{\xi}(\dot{B}^{1/2}_{2,1})}\Big)\\
&\quad\leq C \Big(\|f_1\|^{1/2}_{\widetilde{L}^{\infty}_{T_1}\widetilde{L}^{2}_{\xi}(\dot{B}^{1/2}_{2,1})}+\|f_1\|^{1/2}_{ \widetilde{L}^{2}_{T_1}\widetilde{L}^{2}_{\xi,\nu}(\dot{B}^{3/2}_{2,1})} \Big)\\
&\quad\quad\times\Big( \|\widetilde{f}\|_{\widetilde{L}^{\infty}_{T_1}\widetilde{L}^{2}_{\xi}(\dot{B}^{1/2}_{2,1})}+\|\widetilde{f}\|_{\widetilde{L}^{2}_{T_1}\widetilde{L}^{2}_{\xi,\nu}(\dot{B}^{1/2}_{2,1})}+\|\widetilde{f}\|_{ \widetilde{L}^{2}_{T_1}\widetilde{L}^{2}_{\xi,\nu}(\dot{B}^{3/2}_{2,1})}\Big).
\end{aligned}
\end{flalign*}
A similar computation yields
\begin{flalign*}
\begin{aligned}
&\sum_{q\in\mathbb{Z}}2^{\frac{q}{2}}\bigg(\int_{0}^{T_1}\Big|\Big(\Dot{\Delta}_{q}\Gamma(f_2,\widetilde{f}),\Dot{\Delta}_{q}\widetilde{f}\Big)_{\xi,x}\Big|dt\bigg)^{1/2}\\
&\quad\leq C \Big(\|f_2\|^{1/2}_{\widetilde{L}^{\infty}_{T_1}\widetilde{L}^{2}_{\xi}(\dot{B}^{1/2}_{2,1})}+\|f_2\|^{1/2}_{ \widetilde{L}^{2}_{T_1}\widetilde{L}^{2}_{\xi,\nu}(\dot{B}^{3/2}_{2,1})} \Big)\\
&\quad\quad\times\Big( \|\widetilde{f}\|_{\widetilde{L}^{\infty}_{T_1}\widetilde{L}^{2}_{\xi}(\dot{B}^{1/2}_{2,1})}+\|\widetilde{f}\|_{\widetilde{L}^{2}_{T_1}\widetilde{L}^{2}_{\xi,\nu}(\dot{B}^{1/2}_{2,1})}+\|\widetilde{f}\|_{ \widetilde{L}^{2}_{T_1}\widetilde{L}^{2}_{\xi,\nu}(\dot{B}^{3/2}_{2,1})}\Big).
\end{aligned}
\end{flalign*}
By substituting the above two inequalities into  (\ref{1dlg-E6.46}), we obtain
\begin{equation}\label{un1}
\begin{aligned}
&\|\widetilde{f}\|_{\widetilde{L}^{\infty}_{T_1}\widetilde{L}^{2}_{\xi}(\dot{B}^{1/2}_{2,1})}+\|\{\mathbf{I-P}\}\widetilde{f}\|_{\widetilde{L}^{2}_{T_1}\widetilde{L}^{2}_{\xi,\nu}(\dot{B}^{1/2}_{2,1})}\\
&\quad\leq C \Big(\|(f_1,f_2)\|^{1/2}_{\widetilde{L}^{\infty}_{T_1}\widetilde{L}^{2}_{\xi}(\dot{B}^{1/2}_{2,1})}+\|(f_1,f_2)\|^{1/2}_{ \widetilde{L}^{2}_{T_1}\widetilde{L}^{2}_{\xi,\nu}(\dot{B}^{3/2}_{2,1})} \Big)\\
&\quad\quad\times\Big( \|\widetilde{f}\|_{\widetilde{L}^{\infty}_{T_1}\widetilde{L}^{2}_{\xi}(\dot{B}^{1/2}_{2,1})}+\|\widetilde{f}\|_{\widetilde{L}^{2}_{T_1}\widetilde{L}^{2}_{\xi,\nu}(\dot{B}^{1/2}_{2,1})}+\|\widetilde{f}\|_{\widetilde{L}^{2}_{T_1}\widetilde{L}^{2}_{\xi,\nu}(\dot{B}^{3/2}_{2,1}) }\Big).
\end{aligned}
\end{equation}
Similarly, it follows from Lemma \ref{1dlg-L2.7} with $s_1=s_2=3/2$ that
\begin{equation}\label{uv2}
\begin{aligned}
&\|\widetilde{f}\|_{\widetilde{L}^{\infty}_{T_1}\widetilde{L}^{2}_{\xi}(\dot{B}^{3/2}_{2,1})}+\|\{\mathbf{I-P}\}\widetilde{f}\|_{\widetilde{L}^{2}_{T_1}\widetilde{L}^{2}_{\xi,\nu}(\dot{B}^{3/2}_{2,1})}\\
&\quad\leq C \Big(\|(f_1,f_2)\|^{1/2}_{\widetilde{L}^{\infty}_{T_1}\widetilde{L}^{2}_{\xi}(\dot{B}^{3/2}_{2,1})}+\|(f_1,f_2)\|^{1/2}_{ \widetilde{L}^{2}_{T_1}\widetilde{L}^{2}_{\xi,\nu}(\dot{B}^{3/2}_{2,1})} \Big)\\
&\quad\quad\times\Big( \|\widetilde{f}\|_{\widetilde{L}^{\infty}_{T_1}\widetilde{L}^{2}_{\xi}(\dot{B}^{3/2}_{2,1})}+\|\widetilde{f}\|_{\widetilde{L}^{2}_{T_1}\widetilde{L}^{2}_{\xi,\nu}(\dot{B}^{3/2}_{2,1})}\Big).
\end{aligned}
\end{equation}
Concerning the term involving $f_i$ with $i=1,2$, in view of \eqref{1dlg-E1.18} and (\ref{1dlg-E1.19}), it holds that
\begin{equation}\label{uv4}
\begin{aligned}
&\|f_i\|_{\widetilde{L}^{\infty}_{T_1}\widetilde{L}^{2}_{\xi}(\dot{B}^{1/2}_{2,1})\cap \widetilde{L}^{\infty}_{T_1}\widetilde{L}^{2}_{\xi}(\dot{B}^{3/2}_{2,1})}\leq C \var_0,\\
\end{aligned}
\end{equation}
and
\begin{equation}\label{uv40}
\begin{aligned}
\|f_i\|_{\widetilde{L}^{2}_{T_1}\widetilde{L}^{2}_{\xi,\nu}(\dot{B}^{3/2}_{2,1})}&\leq CT_1^{1/2}\|\mathbf{P}f_i\|_{\widetilde{L}^{\infty}_{T_1}\widetilde{L}^{2}_{\xi}(\dot{B}^{3/2}_{2,1})}+C\|\{\mathbf{I-P}\}f_i\|_{\widetilde{L}^{2}_{T_1}\widetilde{L}^{2}_{\xi,\nu}(\dot{B}^{3/2}_{2,1})}\\
&\leq C(1+T_1^{1/2}) \var_0.
\end{aligned}
\end{equation}
In addition, one observes that
\begin{align}
\|\mathbf{P}\widetilde{f}\|_{\widetilde{L}^{2}_{T_1}\widetilde{L}^{2}_{\xi,\nu}(\dot{B}^{1/2}_{2,1})\cap \widetilde{L}^{2}_{T_1}\widetilde{L}^{2}_{\xi,\nu}(\dot{B}^{3/2}_{2,1})}\leq C T_1^{1/2}  \|\widetilde{f}\|_{\widetilde{L}^{\infty}_{T_1}\widetilde{L}^{2}_{\xi}(\dot{B}^{1/2}_{2,1})\cap \widetilde{L}^{\infty}_{T_1}\widetilde{L}^{2}_{\xi}(\dot{B}^{3/2}_{2,1})}.\label{uv41}
\end{align}
Adding \eqref{un1}, \eqref{uv2} and \eqref{uv41} together and using \eqref{uv4} and \eqref{uv40}, we arrive at
\begin{equation}\label{uv5}
\begin{aligned}
&\|\widetilde{f}\|_{\widetilde{L}^{\infty}_{T_1}\widetilde{L}^{2}_{\xi}(\dot{B}^{1/2}_{2,1})\cap \widetilde{L}^{\infty}_{T_1}\widetilde{L}^{2}_{\xi}(\dot{B}^{3/2}_{2,1})}+\|\widetilde{f}\|_{\widetilde{L}^{2}_{T_1}\widetilde{L}^{2}_{\xi,\nu}(\dot{B}^{1/2}_{2,1})\cap \widetilde{L}^{2}_{T_1}\widetilde{L}^{2}_{\xi,\nu}(\dot{B}^{3/2}_{2,1})}\\
&\quad \leq C(T_1^{1/2}+(1+T_1^{1/4})\var_0^{1/2})\\
&\quad\quad\times\Big(\|\widetilde{f}\|_{\widetilde{L}^{\infty}_{T_1}\widetilde{L}^{2}_{\xi}(\dot{B}^{1/2}_{2,1})\cap \widetilde{L}^{\infty}_{T_1}\widetilde{L}^{2}_{\xi}(\dot{B}^{3/2}_{2,1})}+\|\widetilde{f}\|_{\widetilde{L}^{2}_{T_1}\widetilde{L}^{2}_{\xi,\nu}(\dot{B}^{1/2}_{2,1})\cap \widetilde{L}^{2}_{T_1}\widetilde{L}^{2}_{\xi,\nu}(\dot{B}^{3/2}_{2,1})} \Big).
\end{aligned}
\end{equation}
Letting $T_1$ be suitably small and recalling the smallness of $\var_0$, we conclude from \eqref{uv5} that $f_1=f_2$ on $[0,T_1]\times\mathbb{R}^3\times\mathbb{R}^3$. The above argument can be repeated on $[T_1,2T_1]$, $[2T_1,3T_1]$,..., until the whole interval $[0,T]$ is exhausted. Therefore, the proof of Theorem \ref{1dlg-T1.2} is complete.   \hfill $\Box$

\section{Proof of Theorem \ref{1dlg-T1.3}}\label{1dlg-S6}
This section is dedicated to proving Theorem \ref{1dlg-T1.3} concerning the optimal decay rates of the solution to (\ref{1dlg-E1.4}).
To overcome the challenging difficulty arising from the bilinear nonlocal collision operator, we develop the weighted Lyapunov energy approach, which is different from what  was adopted in the investigation of viscous compressible fluids \cite{Xin-2021}.

\subsection{Time-weighted Lyapunov approach} For that end, we introduce a new time-weighted energy functional
\begin{equation}\nonumber
\begin{aligned}
\mathcal{X}_{M}(t)&\triangleq\|(1+\tau)^{M}f\|_{\widetilde{L}^{\infty}_{t}\widetilde{L}^{2}_{\xi}(\dot{B}^{3/2}_{2,1})}+\|(1+\tau)^{M}\mathbf{P}f\|_{\widetilde{L}^{2}_{t}\widetilde{L}^{2}_{\xi}(\dot{B}^{5/2}_{2,1})}^{\ell}\\
&\quad+\|(1+\tau)^{M}\{\mathbf{I-P}\}f\|_{\widetilde{L}^{2}_{t}\widetilde{L}^{2}_{\xi,\nu}(\dot{B}^{3/2}_{2,1})}^{\ell}+\|(1+\tau)^{M}f\|^{h}_{\widetilde{L}^{2}_{t}\widetilde{L}^{2}_{\xi,\nu}(\dot{B}^{3/2}_{2,1})},\\
\end{aligned}
\end{equation}
where $M>0$ is chosen sufficiently large. Consequently, we have the following time-weighted Lyapunov estimate.

\begin{prop}\label{1dlg-R6.1}
Let $f$ be the global solution to the Cauchy problem $\rm{(\ref{1dlg-E1.4})}$ given by Theorem $\rm{\ref{1dlg-T1.2}}$. Under the assumption of Theorem $\rm{\ref{1dlg-T1.3}}$, it holds that
\begin{equation}\label{1dlg-E6.2}
\begin{aligned}
\mathcal{X}_{M}(t)&\lesssim \delta_{0} (1+t)^{M-\frac{1}{2}(\frac{3}{2}-\sigma_{0})}
\end{aligned}
\end{equation}
for $M>1+\frac{1}{2}(\frac{3}{2}-\sigma_{0})$ and $t>0$, where $\delta_{0}\triangleq\|f_{0}\|^{\ell}_{\widetilde{L}^{2}_{\xi}(\dot{B}^{\sigma_{0}}_{2,\infty})}+\|f_{0}\|^{h}_{\widetilde{L}^{2}_{\xi}(\dot{B}^{3/2}_{2,1})}$.
\end{prop}

\begin{proof}
The proof is separated into several steps. 
\vspace{1mm}

\begin{itemize}
\item \emph{Step 1: Low-frequency estimates}
\end{itemize}

\vspace{1mm}
Let us begin with the Lyapunov type inequality (\ref{1dlg-E5.23}) in the low-frequency regime. Multiplying (\ref{1dlg-E5.23}) by $(1+t)^{2M}$ and using the fact that $(1+t)^{2M}\frac{d}{dt}\mathcal{L}_{1,q}(t)=\frac{d}{dt}\big((1+t)^{2M}\mathcal{L}_{1,q}(t)\big)-2M(1+t)^{2M-1}\mathcal{L}_{1,q}(t)$, we obtain
\begin{equation}\nonumber
\begin{aligned}
&\frac{d}{dt}\bigg((1+t)^{2M}\mathcal{L}_{1,q}(t)\bigg)+(1+t)^{2M}2^{2q}\mathcal{L}_{1,q}(t)+(1+t)^{2M}\|\sqrt{\nu(\xi)}\dot{\Delta}_{q}\{\mathbf{I-P}\}f\|^{2}_{L^{2}_{\xi}L^{2}_{x}}\\
&\quad\lesssim(1+t)^{2M-1}\mathcal{L}_{1,q}(t)+(1+t)^{2M}\Big|\Big(\dot{\Delta}_{q}\Gamma(f,f),\dot{\Delta}_{q}\{\mathbf{I-P}\}f\Big)_{\xi,x}\Big|\\
&\quad\quad+(1+t)^{2M}\sum_{i=1}^{3}\|\Lambda_{i}(\dot{\Delta}_{q}\mathbf{h})\|^{2}_{L^{2}_{x}}+(1+t)^{2M}\sum_{i,j=1}^{3}\|\Theta_{ij}(\dot{\Delta}_{q}\mathbf{h})\|^{2}_{L^{2}_{x}}
\end{aligned}
\end{equation} for
$q\leq0$, which together with (\ref{1dlg-E5.21}) and (\ref{1dlg-E5.22}) implies that
\begin{equation}\label{1dlg-E6.4}
\begin{aligned}
&(1+t)^{M}\|\dot{\Delta}_{q}f(t)\|_{L^{2}_{\xi}L^{2}_{x}}+2^{q}\bigg(\int_{0}^{t}\|(1+\tau)^{M}\dot{\Delta}_{q}\mathbf{P}f\|^{2}_{L^{2}_{\xi}L^{2}_{x}}d\tau\bigg)^{1/2}\\
&\quad\quad+\bigg(\int_{0}^{t}\|(1+\tau)^{M}\dot{\Delta}_{q}\{\mathbf{I-P}\}f\|^{2}_{L^{2}_{\xi,\nu}L^{2}_{x}}d\tau\bigg)^{1/2}\\
&\quad\lesssim\|\dot{\Delta}_{q}f_{0}\|_{L^{2}_{\xi}L^{2}_{x}}+\bigg(\int_{0}^{t}\|(1+\tau)^{M-1/2}\dot{\Delta}_{q}f\|^{2}_{L^{2}_{\xi}L^{2}_{x}}d\tau\bigg)^{1/2}\\
&\quad\quad+\bigg(\int_{0}^{t}(1+\tau)^{2M}\Big|\Big(\dot{\Delta}_{q}\Gamma(f,f),\dot{\Delta}_{q}\{\mathbf{I-P}\}f\Big)_{\xi,x}\Big|\,d\tau\bigg)^{1/2}\\
&\quad\quad+\sum_{i=1}^{3}\|(1+\tau)^{M}\Lambda_{i}(\dot{\Delta}_{q}\mathbf{h})\|_{L^{2}_{t}L^{2}_{x}}+\sum_{i,j=1}^{3}\|(1+\tau)^{M}\Theta_{ij}(\dot{\Delta}_{q}\mathbf{h})\|_{L^{2}_{t}L^{2}_{x}}.
\end{aligned}
\end{equation}
Then, we multiply (\ref{1dlg-E6.4}) by $2^{3q/2}$, take the supremum on $[0,t]$, and then sum over $q\leq0$ to get
\begin{equation}\label{1dlg-E6.5}
\begin{aligned}
&\|(1+\tau)^{M}f\|^{\ell}_{\widetilde{L}^{\infty}_{t}\widetilde{L}^{2}_{\xi}(\dot{B}^{3/2}_{2,1})}+\|(1+\tau)^{M}\mathbf{P}f\|^{\ell}_{\widetilde{L}^{2}_{t}\widetilde{L}^{2}_{\xi}(\dot{B}^{5/2}_{2,1})}\\
&\quad\quad+\|(1+\tau)^{M}\{\mathbf{I-P}\}f\|^{\ell}_{\widetilde{L}^{2}_{t}\widetilde{L}^{2}_{\xi,\nu}(\dot{B}^{3/2}_{2,1})}\\
&\quad\leq C\|f_{0}\|^{\ell}_{L^{2}_{\xi}(\dot{B}^{3/2}_{2,1})}+C\|(1+\tau)^{M-1/2}f\|^{\ell}_{\widetilde{L}^{2}_{t}\widetilde{L}^{2}_{\xi}(\dot{B}^{3/2}_{2,1})}\\
&\quad\quad+C\sum_{q\leq0}2^{\frac{3q}{2}}\bigg(\int_{0}^{t}(1+\tau)^{2M}\Big|\Big(\dot{\Delta}_{q}\Gamma(f,f),\dot{\Delta}_{q}\{\mathbf{I-P}\}f\Big)_{\xi,x}\Big|\,d\tau\bigg)^{1/2}\\
&\quad\quad+C\sum_{i=1}^{3}\sum_{q\leq0}2^{\frac{3q}{2}}\|(1+\tau)^{M}\Lambda_{i}(\dot{\Delta}_{q}\mathbf{h})\|_{L^{2}_{t}L^{2}_{x}}\\
&\quad\quad+C\sum_{i,j=1}^{3}\sum_{q\leq0}2^{\frac{3q}{2}}\|(1+\tau)^{M}\Theta_{ij}(\dot{\Delta}_{q}\mathbf{h})\|_{L^{2}_{t}L^{2}_{x}}.
\end{aligned}
\end{equation}
It is worth emphasizing that the second term on the right-hand side of (\ref{1dlg-E6.5}) plays a crucial role in the derivation of decay rates. To bound this term, we apply Young's inequality to deduce
\begin{equation}\label{1dlg-E6.6}
\begin{aligned}
&C\|(1+\tau)^{M-1/2}f\|_{\widetilde{L}^{2}_{t}\widetilde{L}^{2}_{\xi}(\dot{B}^{3/2}_{2,1})}^{\ell}\\
&\quad\leq C\|(1+\tau)^{M-1}f\|_{\widetilde{L}^{1}_{t}\widetilde{L}^{2}_{\xi}(\dot{B}^{3/2}_{2,1})}^{\ell}+\frac{1}{4}\|(1+\tau)^{M}f\|_{\widetilde{L}^{\infty}_{t}\widetilde{L}^{2}_{\xi}(\dot{B}^{3/2}_{2,1})}^{\ell},
\end{aligned}
\end{equation}
where the first term can be handled as follows:
\begin{equation}\label{1dlg-E6.9000}
\begin{aligned}
&\|(1+\tau)^{M-1}f\|_{\widetilde{L}^{1}_{t}\widetilde{L}^{2}_{\xi}(\dot{B}^{3/2}_{2,1})}^{\ell}\\
&\quad\lesssim \int_{0}^{t}(1+\tau)^{M-1}\|f^{\ell}\|_{\widetilde{L}^{2}_{\xi}(\dot{B}^{3/2}_{2,1})}d\tau+\int_{0}^{t}(1+\tau)^{M-1}\|f^{h}\|_{\widetilde{L}^{2}_{\xi}(\dot{B}^{3/2}_{2,1})}d\tau.
\end{aligned}
\end{equation}
Regarding the low-frequency term, we perform the real interpolation in Lemma \ref{1dlg-L2.4} with $s=\sigma_0, \widetilde{s}=5/2$, $p=2$ and $\theta=2/(5-2\sigma_{0})\in(0,1)$, obtaining:
\begin{equation}\label{1dlg-E6.7}
\begin{aligned}
&\int_{0}^{t}(1+\tau)^{M-1}\|f^{\ell}\|_{\widetilde{L}^{2}_{\xi}(\dot{B}^{3/2}_{2,1})}d\tau\\
&\quad\lesssim\int_{0}^{t}(1+\tau)^{M-1}\Big(\|f^{\ell}\|_{\widetilde{L}^{2}_{\xi}(\dot{B}^{\sigma_{0}}_{2,\infty})}\Big)^{\theta}\Big(\|f^{\ell}\|_{\widetilde{L}^{2}_{\xi}(\dot{B}^{5/2}_{2,\infty})}\Big)^{1-\theta}\,d\tau\\
&\quad\lesssim\|f^{\ell}\|^{\theta}_{\widetilde{L}^{\infty}_{t}\widetilde{L}^{2}_{\xi}(\dot{B}^{\sigma_{0}}_{2,\infty})}\int_{0}^{t}\|(1+\tau)^{M}f^{\ell}\|^{1-\theta}_{\widetilde{L}^{2}_{\xi}(\dot{B}^{5/2}_{2,1})}(1+\tau)^{M\theta-1}\,d\tau\\
&\quad\lesssim\bigg(\|(1+\tau)^{M}f\|^{\ell}_{\widetilde{L}^{2}_{t}\widetilde{L}^{2}_{\xi}(\dot{B}^{5/2}_{2,1})}\bigg)^{1-\theta}\bigg(\|f\|^{\ell}_{\widetilde{L}^{\infty}_{t}\widetilde{L}^{2}_{\xi}(\dot{B}^{\sigma_{0}}_{2,\infty})}\bigg)^{\theta}\|(1+\tau)^{M\theta-1}\|_{L_{t}^{\frac{2}{1+\theta}}}.
\end{aligned}
\end{equation}
On the other hand, by taking advantage of the dissipation property of $f$ in high frequencies, it is easy to get
\begin{equation}\label{1dlg-E6.9}
\begin{aligned}
&\int_{0}^{t}(1+\tau)^{M-1}\|f^{h}\|_{\widetilde{L}^{2}_{\xi}(\dot{B}^{3/2}_{2,1})}d\tau\\
&\quad\lesssim\bigg(\|f\|^{h}_{\widetilde{L}^{\infty}_{t}\widetilde{L}^{2}_{\xi}(\dot{B}^{3/2}_{2,1})}\bigg)^{\theta}\int_{0}^{t}\bigg(\|(1+\tau)^{M}f\|^{h}_{\widetilde{L}^{2}_{\xi}(\dot{B}^{3/2}_{2,1})}\bigg)^{1-\theta}(1+\tau)^{M\theta-1}d\tau\\
&\quad\lesssim\bigg(\|(1+\tau)^{M}f\|^{h}_{\widetilde{L}^{2}_{t}\widetilde{L}^{2}_{\xi}(\dot{B}^{3/2}_{2,1})}\bigg)^{1-\theta}\bigg(\|f\|^{h}_{\widetilde{L}^{\infty}_{t}\widetilde{L}^{2}_{\xi}(\dot{B}^{3/2}_{2,1})}\bigg)^{\theta}\|(1+\tau)^{M\theta-1}\|_{L_{t}^{\frac{2}{1+\theta}}}.
\end{aligned}
\end{equation}
By combining (\ref{1dlg-E6.6})-(\ref{1dlg-E6.9}), Young's inequality and the fact that $\|(1+\tau)^{M\theta-1}\|_{L_{t}^{\frac{2}{1+\theta}}}\lesssim\Big((1+t)^{M-\frac{1}{2}(\frac{3}{2}-\sigma_0)}\Big)^{\theta}$, we deduce 
\begin{equation}\label{1dlg-E6.10}
\begin{aligned}
&C\int_{0}^{t}\|(1+\tau)^{M-1}f\|^{\ell}_{\widetilde{L}^{2}_{\xi}(\dot{B}^{3/2}_{2,1})}d\tau\\
&\quad\leq \Big(\|f^{\ell}\|_{\widetilde{L}^{\infty}_{t}\widetilde{L}^{2}_{\xi}(\dot{B}^{\sigma_{0}}_{2,\infty})}+\|f^{h}\|_{\widetilde{L}^{\infty}_{t}\widetilde{L}^{2}_{\xi}(\dot{B}^{3/2}_{2,1})}\Big) (1+t)^{M-\frac{1}{2}(\frac{3}{2}-\sigma_{0})}\\
&\quad\quad+\frac{1}{4}\Big(\|(1+\tau)^{M}f\|_{\widetilde{L}^{2}_{t}\widetilde{L}^{2}_{\xi}(\dot{B}^{5/2}_{2,1})}^{\ell}+\|(1+\tau)^{M}f\|^{h}_{\widetilde{L}^{2}_{t}\widetilde{L}^{2}_{\xi}(\dot{B}^{3/2}_{2,1})}\Big).
\end{aligned}
\end{equation}
Now, let us deal with those nonlinear terms on the right-hand side of \eqref{1dlg-E6.5}. Taking advantage of Lemma \ref{1dlg-L2.7} and \eqref{2mm}, we obtain
\begin{equation}\label{1dlg-E6.151}
\begin{aligned}
&\sum_{q\in\mathbb{Z}}2^{\frac{3q}{2}}\bigg(\int_{0}^{t}\big|\big((1+\tau)^{2M}\dot{\Delta}_{q}\Gamma(f,f),\dot{\Delta}_{q}\{\mathbf{I-P}\}f\big)_{\xi,x}\big|\,d\tau\bigg)^{1/2}\\
&\quad\lesssim\|(1+\tau)^{M}f\|^{1/2}_{\widetilde{L}^{\infty}_{t}\widetilde{L}^{2}_{\xi}(\dot{B}^{3/2}_{2,1})}\|f\|^{1/2}_{\widetilde{L}^{2}_{t}\widetilde{L}^{2}_{\xi}(\dot{B}^{3/2}_{2,1})}\|(1+\tau)^{M}\{\mathbf{I-P}\}f\|^{1/2}_{\widetilde{L}^{2}_{t}\widetilde{L}^{2}_{\xi,\nu}(\dot{B}^{3/2}_{2,1})}\\
&\quad\lesssim\sqrt{\mathcal{D}(t)}\mathcal{X}_{M}(t).
\end{aligned}
\end{equation}
Employing Lemma \ref{1dlg-L2.9} directly gives rise to
\begin{equation}\label{1dlg-E6.15}
\begin{aligned}
&\sum_{i=1}^{3}\sum_{q\in\mathbb{Z}}2^{\frac{3q}{2}}\|(1+\tau)^{M}\Lambda_{i}(\dot{\Delta}_{q}\mathbf{h})\|_{L^{2}_{t}L^{2}_{x}}+\sum_{i,j=1}^{3}\sum_{q\in\mathbb{Z}}2^{\frac{3q}{2}}\|\tau^{M}\Theta_{ij}(\dot{\Delta}_{q}\mathbf{h})\|_{L^{2}_{t}L^{2}_{x}}\\
&\quad\lesssim\|(1+\tau)^{M}f\|_{\widetilde{L}^{\infty}_{t}\widetilde{L}^{2}_{\xi}(\dot{B}^{3/2}_{2,1})}\|f\|_{\widetilde{L}^{2}_{t}\widetilde{L}^{2}_{\xi}(\dot{B}^{3/2}_{2,1})}\\
&\quad\lesssim\mathcal{D}_t(f)\mathcal{X}_{M}(t).
\end{aligned}
\end{equation}
Substituting (\ref{1dlg-E6.6})-(\ref{1dlg-E6.15}) into (\ref{1dlg-E6.5}), we get
\begin{equation}\label{1dlg-E6.16}
\begin{aligned}
&\|(1+\tau)^{M}f\|^{\ell}_{\widetilde{L}^{\infty}_{t}\widetilde{L}^{2}_{\xi}(\dot{B}^{3/2}_{2,1})}+\frac{3}{4}\|(1+\tau)^{M}\mathbf{P}f\|^{\ell}_{\widetilde{L}^{2}_{t}\widetilde{L}^{2}_{\xi}(\dot{B}^{5/2}_{2,1})}\\
&\quad\quad+\|(1+\tau)^{M}\{\mathbf{I-P}\}f\|^{\ell}_{\widetilde{L}^{2}_{t}\widetilde{L}^{2}_{\xi,\nu}(\dot{B}^{3/2}_{2,1})}\\
&\quad\leq C\|f_{0}\|^{\ell}_{L^{2}_{\xi}(\dot{B}^{3/2}_{2,1})}+C\Big(\|f\|_{\widetilde{L}^{\infty}_{t}\widetilde{L}^{2}_{\xi}(\dot{B}^{\sigma_{0}}_{2,\infty})}^{\ell}+\|f\|^{h}_{\widetilde{L}^{\infty}_{t}\widetilde{L}^{2}_{\xi}(\dot{B}^{3/2}_{2,1})}\Big) (1+t)^{M-\frac{1}{2}(\frac{3}{2}-\sigma_{0})}\\
&\quad\quad+C\Big(\sqrt{\mathcal{D}_t(f)}+\mathcal{D}_t(f)\Big)\mathcal{X}_{M}(t)+\frac{1}{4}\|(1+\tau)^{M}f\|^{h}_{\widetilde{L}^{2}_{t}\widetilde{L}^{2}_{\xi}(\dot{B}^{3/2}_{2,1})}.
\end{aligned}
\end{equation}

\vspace{1mm}

\begin{itemize}
\item \emph{Step 2: High-frequency estimates}
\end{itemize}
For any $q\geq-1$, multiplying the Lyapunov inequality (\ref{1dlg-E5.35}) by $(1+t)^{2M}$ gives 
\begin{align}\nonumber
\begin{aligned}
&\frac{d}{dt}\big((1+t)^{2M}\mathcal{L}_{2,q}(t)\big)+(1+t)^{2M}\mathcal{L}_{2,q}(t)+(1+t)^{2M}\|\sqrt{\nu(\xi)}\dot{\Delta}_{q}\{\mathbf{I-P}\}f\|^{2}_{L^{2}_{\xi}L^{2}_{x}}\\
&\quad\lesssim(1+t)^{2M-1}\mathcal{L}_{2,q}(t)+(1+t)^{2M}\big|\big(\dot{\Delta}_{q}\Gamma(f,f),\dot{\Delta}_{q}\{\mathbf{I-P}\}f\big)_{\xi,x}\big|\\
&\quad\quad+2^{-2q}(1+t)^{2M}\sum_{i=1}^{3}\|\Lambda_{i}(\dot{\Delta}_{q}\mathbf{h})\|^{2}_{L^{2}_{x}}+2^{-2q}(1+t)^{2M}\sum_{i,j=1}^{3}\|\Theta_{ij}(\dot{\Delta}_{q}\mathbf{h})\|^{2}_{L^{2}_{x}}.
\end{aligned}
\end{align}
Furthermore, it is straightforward to obtain
\begin{equation}\label{1dlg-E6.17}
\begin{aligned}
&\|(1+\tau)^{M}f\|^{h}_{\widetilde{L}^{\infty}_{t}\widetilde{L}^{2}_{\xi}(\dot{B}^{3/2}_{2,1})}+\|(1+\tau)^{M}f\|^{h}_{\widetilde{L}^{2}_{t}\widetilde{L}^{2}_{\xi}(\dot{B}^{3/2}_{2,1})}\\
&\quad\quad+\|(1+\tau)^{M}\{\mathbf{I-P}\}f\|^{h}_{\widetilde{L}^{2}_{t}\widetilde{L}^{2}_{\xi,\nu}(\dot{B}^{3/2}_{2,1})}\\
&\quad\lesssim \|f_{0}\|^{h}_{\widetilde{L}^{2}_{\xi}(\dot{B}^{3/2}_{2,1})}+\|(1+\tau)^{M-1/2}f\|^{h}_{\widetilde{L}^{2}_{t}\widetilde{L}^{2}_{\xi}(\dot{B}^{3/2}_{2,1})}\\
&\quad\quad+\sum_{q\geq-1}2^{\frac{3q}{2}}\Big(\int_{0}^{t}\big|\big((1+\tau)^{2M}\dot{\Delta}_{q}\Gamma(f,f),\dot{\Delta}_{q}\{\mathbf{I-P}\}f\big)_{\xi,x}\big|\,d\tau\Big)^{1/2}\\
&\quad\quad+\sum_{i=1}^{3}\sum_{q\geq-1}2^{\frac{q}{2}}\|(1+\tau)^{M}\Lambda_{i}(\dot{\Delta}_{q}\mathbf{h})\|_{L^{2}_{t}L^{2}_{x}}\\
&\quad\quad+\sum_{i,j=1}^{3}\sum_{q\geq-1}2^{\frac{q}{2}}\|(1+\tau)^{M}\Theta_{ij}(\dot{\Delta}_{q}\mathbf{h})\|_{L^{2}_{t}L^{2}_{x}}.
\end{aligned}
\end{equation}
By employing the similar procedure leading to (\ref{1dlg-E6.6})-(\ref{1dlg-E6.10}), one can get
\begin{equation}\label{1dlg-E6.18}
\begin{aligned}
&\|(1+\tau)^{M-1/2}f\|^{h}_{\widetilde{L}^{2}_{t}\widetilde{L}^{2}_{\xi}(\dot{B}^{3/2}_{2,1})}\\
&\quad\leq C\int_{0}^{t}\|(1+\tau)^{M-1}f\|^{h}_{\widetilde{L}^{2}_{\xi}(\dot{B}^{3/2}_{2,1})}d\tau+\frac{1}{4}\|(1+\tau)^{M}f\|^{h}_{\widetilde{L}^{\infty}_{t}\widetilde{L}^{2}_{\xi}(\dot{B}^{3/2}_{2,1})}\\
&\quad\leq C\|f\|^{h}_{\widetilde{L}^{\infty}_{t}\widetilde{L}^{2}_{\xi}(\dot{B}^{3/2}_{2,1})} (1+t)^{M-\frac{1}{2}(\frac{3}{2}-\sigma_{0})}\\
&\quad\quad+\frac{1}{4}\|(1+\tau)^{M}f\|^{h}_{\widetilde{L}^{2}_{t}\widetilde{L}^{2}_{\xi}(\dot{B}^{3/2}_{2,1})}+\frac{1}{4}\|(1+\tau)^{M}f\|^{h}_{\widetilde{L}^{\infty}_{t}\widetilde{L}^{2}_{\xi}(\dot{B}^{3/2}_{2,1})}.
\end{aligned}
\end{equation}
Hence, in view of \eqref{1dlg-E6.151}-\eqref{1dlg-E6.15} and \eqref{1dlg-E6.17}-\eqref{1dlg-E6.18}, it holds that
\begin{equation}\label{1dlg-E6.19}
\begin{aligned}
&\|(1+\tau)^{M}f\|^{h}_{\widetilde{L}^{\infty}_{t}\widetilde{L}^{2}_{\xi}(\dot{B}^{3/2}_{2,1})}+\|(1+\tau)^{M}f\|^{h}_{\widetilde{L}^{2}_{t}\widetilde{L}^{2}_{\xi,\nu}(\dot{B}^{3/2}_{2,1})}\\
&\quad\lesssim \|f_{0}\|^{h}_{\widetilde{L}^{2}_{\xi}(\dot{B}^{3/2}_{2,1})}+\|f\|^{h}_{\widetilde{L}^{\infty}_{t}\widetilde{L}^{2}_{\xi}(\dot{B}^{3/2}_{2,1})} (1+t)^{M-\frac{1}{2}(\frac{3}{2}-\sigma_{0})}\\
&\quad\quad+\Big(\sqrt{\mathcal{D}_t(f)}+\mathcal{D}_t(f)\Big)\mathcal{X}_{M}(t).
\end{aligned}
\end{equation}

\vspace{1mm}

\begin{itemize}
\item \emph{Step 3: The gain of time-weighted estimates}
\end{itemize}
By combining (\ref{1dlg-E6.5}) and (\ref{1dlg-E6.19}), we conclude that 
\begin{equation}\nonumber
\begin{aligned}
\mathcal{X}_{M}(t)&\lesssim \|f_{0}\|_{\widetilde{L}^{2}_{\xi}(\dot{B}^{3/2}_{2,1})}^{\ell}+\|f_{0}\|^{h}_{\widetilde{L}^{2}_{\xi}(\dot{B}^{3/2}_{2,1})}\\
&\quad+\Big(\|f^{\ell}\|_{\widetilde{L}^{\infty}_{t}\widetilde{L}^{2}_{\xi}(\dot{B}^{\sigma_{0}}_{2,\infty})}+\|f\|^{h}_{\widetilde{L}^{\infty}_{t}\widetilde{L}^{2}_{\xi}(\dot{B}^{3/2}_{2,1})}\Big)(1+t)^{M-\frac{1}{2}(\frac{3}{2}-\sigma_{0})}\\
&\quad+(\sqrt{\mathcal{D}_t(f)}+\mathcal{D}_t(f))\mathcal{X}_{M}(t).
\end{aligned}
\end{equation}
We \textbf{claim} that
\begin{equation}\label{1dlg-E6.8}
\|f\|_{\widetilde{L}^{\infty}_{t}\widetilde{L}^{2}_{\xi}(\dot{B}^{\sigma_{0}}_{2,\infty})}^\ell\leq C\delta_{0}
\end{equation}
for all $t>0$ and some uniform constant $C>0$. The proof is left to the next subsection. Together with the global existence result (Theorem \ref{1dlg-T1.2}) that implies that $\mathcal{D}_t(f)\lesssim \var_0<<1$ and
\begin{equation}\nonumber
\begin{aligned}
\|f\|^{h}_{\widetilde{L}^{\infty}_{t}\widetilde{L}^{2}_{\xi}(\dot{B}^{3/2}_{2,1})}&\lesssim \|f_{0}\|^{\ell}_{\widetilde{L}^{2}_{\xi}(\dot{B}^{1/2}_{2,1})} +\|f_{0}\|^{h}_{\widetilde{L}^{2}_{\xi}(\dot{B}^{3/2}_{2,1})}\lesssim \var_0\lesssim \delta_0,
\end{aligned}
\end{equation}
we end up with \eqref{1dlg-E6.2}. Therefore, the proof of Proposition \ref{1dlg-R6.1} is finished.\end{proof}

\subsection{The evolution of low-frequency Besov regularity}
In this section, we establish the 
evolution of the $\widetilde{L}^{2}_{\xi}(\dot{B}^{\sigma_{0}}_{2,\infty})$ norm at low frequencies, say \eqref{1dlg-E6.8},  which plays a crucial role in deriving the weighted Lyapunov-type estimate \eqref{1dlg-E6.2}. 

\begin{lemma}\label{1dlg-R6.2}
If $f$ is the global solution to \eqref{1dlg-E1.4} given by Theorem \ref{1dlg-T1.2}, then for all $t>0$ the following inequality holds:
\begin{equation}\label{1dlg-E6.24}
\begin{aligned}
\|f\|_{\widetilde{L}^{\infty}_{t}\widetilde{L}^{2}_{\xi}(\dot{B}^{\sigma_{0}}_{2,\infty})}^{\ell}+\|\mathbf{P}f\|_{\widetilde{L}^{2}_{t}\widetilde{L}^{2}_{\xi}(\dot{B}^{\sigma_{0}+1}_{2,\infty})}^{\ell}+\|\{\mathbf{I-P}\}f\|_{\widetilde{L}^{2}_{t}\widetilde{L}^{2}_{\xi,\nu}(\dot{B}^{\sigma_{0}}_{2,\infty})}^{\ell}\leq C\delta_{0}.
\end{aligned}
\end{equation}
Here $C>0$ is a uniform constant, and $\delta_{0}\triangleq\|f_{0}\|^{\ell}_{\widetilde{L}^{2}_{\xi}(\dot{B}^{\sigma_{0}}_{2,\infty})}+\|f_{0}\|^{h}_{\widetilde{L}^{2}_{\xi}(\dot{B}^{3/2}_{2,1})}$.
\end{lemma}
\begin{proof}
We recall the low-frequency localized energy estimate \eqref{1dlg-E5.24}.  Multiplying (\ref{1dlg-E5.24}) by $2^{q\sigma_{0}}$ and taking the supremum over $[0,t]$ and $q\in\mathbb{Z}$ on both sides of the resulting inequality, we obtain
\begin{equation}\label{1dlg-E6.25}
\begin{aligned}
&\|f\|_{\widetilde{L}^{\infty}_{t}\widetilde{L}^{2}_{\xi}(\dot{B}^{\sigma_{0}}_{2,\infty})}^{\ell}+\|f\|_{\widetilde{L}^{2}_{t}\widetilde{L}^{2}_{\xi}(\dot{B}^{\sigma_{0}+1}_{2,\infty})}^{\ell}+\|\{\mathbf{I-P}\}f\|_{\widetilde{L}^{2}_{t}\widetilde{L}^{2}_{\xi,\nu}(\dot{B}^{\sigma_{0}}_{2,\infty})}^{\ell}\\
&\quad\leq C\|f_{0}\|_{\widetilde{L}^{2}_{\xi}(\dot{B}^{\sigma_{0}}_{2,\infty})}^{\ell}+C \sup_{q\leq 0}2^{q\sigma_{0}}\bigg(\int_{0}^{t}\big|\big(\dot{\Delta}_{q}\Gamma(f,f),\dot{\Delta}_{q}\{\mathbf{I-P}\}f\big)_{\xi,x}\big|\,d\tau\bigg)^{1/2}\\
&\quad\quad+C \sup_{q\leq 0}2^{q\sigma_{0}}\sum_{i=1}^{3}\|\Lambda_{i}(\dot{\Delta}_{q}\mathbf{h})\|_{L^{2}_{t}L^{2}_{x}}+C \sup_{q\leq 0}2^{q\sigma_{0}}\sum_{i,j=1}^{3}\|\Theta_{ij}(\dot{\Delta}_{q}\mathbf{h})\|_{L^{2}_{t}L^{2}_{x}}.
\end{aligned}
\end{equation}
Those nonlinear terms in \eqref{1dlg-E6.25} can be estimated as follows. By employing Lemma \ref{1dlg-L2.8} with $h=\{\mathbf{I-P}\}f,~g=f,$ $s_{1}=\sigma_{0}$ and $s_{2}=3/2$, we arrive at 
\begin{equation}\label{1dlg-E6.26}
\begin{aligned}
&\sup_{q\leq 0}2^{q\sigma_{0}}\bigg(\int_{0}^{t}\big|\big(\dot{\Delta}_{q}\Gamma(f,f),\dot{\Delta}_{q}\{\mathbf{I-P}\}f\big)_{\xi,x}\big|dt\bigg)^{1/2}\\
&\quad\leq C\|f\|^{1/2}_{\widetilde{L}^{\infty}_{t}\widetilde{L}^{2}_{\xi}(\dot{B}^{\sigma_{0}}_{2,\infty})}\|f\|^{1/2}_{\widetilde{L}^{2}_{t}\widetilde{L}^{2}_{\xi,\nu}(\dot{B}^{3/2}_{2,1})}\|\{\mathbf{I-P}\}f\|^{1/2}_{\widetilde{L}^{2}_{t}\widetilde{L}^{2}_{\xi,\nu}(\dot{B}^{\sigma_{0}}_{2,\infty})}\\
&\quad\leq C\|f\|_{\widetilde{L}^{\infty}_{t}\widetilde{L}^{2}_{\xi}(\dot{B}^{\sigma_{0}}_{2,\infty})}\|f\|_{\widetilde{L}^{2}_{t}\widetilde{L}^{2}_{\xi,\nu}(\dot{B}^{3/2}_{2,1})}+\frac{1}{4}\|\{\mathbf{I-P}\}f\|_{\widetilde{L}^{2}_{t}\widetilde{L}^{2}_{\xi,\nu}(\dot{B}^{\sigma_{0}}_{2,\infty})}\\
&\quad\leq C\|f\|_{\widetilde{L}^{2}_{t}\widetilde{L}^{2}_{\xi,\nu}(\dot{B}^{3/2}_{2,1})}\Big(\|f\|_{\widetilde{L}^{\infty}_{t}\widetilde{L}^{2}_{\xi}(\dot{B}^{\sigma_{0}}_{2,\infty})}^{\ell}+\|f\|^{h}_{\widetilde{L}^{\infty}_{t}\widetilde{L}^{2}_{\xi}(\dot{B}^{3/2}_{2,1})}\Big)\\
&\quad\quad+\frac{1}{4}\|\{\mathbf{I-P}\}f\|_{\widetilde{L}^{2}_{t}\widetilde{L}^{2}_{\xi,\nu}(\dot{B}^{\sigma_{0}}_{2,\infty})}^{\ell}+C\|\{\mathbf{I-P}\}f\|_{\widetilde{L}^{2}_{t}\widetilde{L}^{2}_{\xi,\nu}(\dot{B}^{3/2}_{2,1})}^h,
\end{aligned}
\end{equation}
where we have used Young's inequality and  \eqref{1dlg-E2.6}. In addition, it follows from Lemma \ref{1dlg-L2.10} that 
\begin{equation}\label{1dlg-E6.27}
\begin{aligned}
&\sup_{q\leq 0}2^{q\sigma_{0}}\sum_{i=1}^{3}\|\Lambda_{i}(\dot{\Delta}_{q}\mathbf{h})\|_{L^{2}_{t}L^{2}_{x}}+\sup_{q\leq 0}2^{q\sigma_{0}}\sum_{i,j=1}^{3}\|\Theta_{ij}(\dot{\Delta}_{q}\mathbf{h})\|_{L^{2}_{t}L^{2}_{x}}\\
&\quad\leq C\|f\|_{\widetilde{L}^{\infty}_{t}\widetilde{L}^{2}_{\xi}(\dot{B}^{\sigma_{0}}_{2,\infty})}\|f\|_{\widetilde{L}^{2}_{t}\widetilde{L}^{2}_{\xi,\nu}(\dot{B}^{3/2}_{2,1})}\\
&\quad\leq C\|f\|_{\widetilde{L}^{2}_{t}\widetilde{L}^{2}_{\xi,\nu}(\dot{B}^{3/2}_{2,1})}\Big(\|f\|_{\widetilde{L}^{\infty}_{t}\widetilde{L}^{2}_{\xi}(\dot{B}^{\sigma_{0}}_{2,\infty})}^{\ell}+\|f\|^{h}_{\widetilde{L}^{\infty}_{t}\widetilde{L}^{2}_{\xi}(\dot{B}^{3/2}_{2,1})}\Big).
\end{aligned}
\end{equation}
Hence, combining with \eqref{2mm}, (\ref{1dlg-E6.25}), (\ref{1dlg-E6.26}) and (\ref{1dlg-E6.27}), we deduce that 
\begin{equation}\label{1dlg-E6.2711}
\begin{aligned}
&\|f\|_{\widetilde{L}^{\infty}_{t}\widetilde{L}^{2}_{\xi}(\dot{B}^{\sigma_{0}}_{2,\infty})}^{\ell}+\|f^{\ell}\|_{\widetilde{L}^{2}_{t}\widetilde{L}^{2}_{\xi}(\dot{B}^{\sigma_{0}+1}_{2,\infty})}+\|\{\mathbf{I-P}\}f^{\ell}\|_{\widetilde{L}^{2}_{t}\widetilde{L}^{2}_{\xi,\nu}(\dot{B}^{\sigma_{0}}_{2,\infty})}\\
&\quad \leq C\|f_{0}^{\ell}\|_{\widetilde{L}^{2}_{\xi}(\dot{B}^{\sigma_{0}}_{2,\infty})}+C\mathcal{D}_t(f)\|f\|_{\widetilde{L}^{\infty}_{t}\widetilde{L}^{2}_{\xi}(\dot{B}^{\sigma_{0}}_{2,\infty})}^{\ell}+\Big(1+\mathcal{E}_t(f)\Big)\mathcal{D}_t(f).
\end{aligned}
\end{equation}
In light of \eqref{1dlg-E1.19}, (\ref{1dlg-E1.20}), \eqref{1dlg-E6.8},  $\eqref{1dlg-E6.2711}$ and the fact that $\mathcal{D}_t(f)\lesssim \var_0<<1$, the inequality \eqref{1dlg-E6.24} follows. The proof of Lemma \ref{1dlg-R6.1} is thus complete.
\end{proof}

\subsection{The optimal decay}
The last section is devoted to the proof of Theorem \ref{1dlg-T1.3}. It follows from Proposition \ref{1dlg-R6.1} 
that 
\begin{equation}\label{1dlg-E6.28}
(1+t)^{M}\|f(t)\|_{\widetilde{L}^{2}_{\xi}(\dot{B}^{3/2}_{2,1})}\lesssim \|(1+\tau)^{M}f\|_{\widetilde{L}^{\infty}_{t}\widetilde{L}^{2}_{\xi}(\dot{B}^{3/2}_{2,1})}\lesssim \delta_{0}(1+t)^{M-\frac{1}{2}(\frac{3}{2}-\sigma_{0})}
\end{equation}
for all $t>0$ and any suitably large $M$. Therefore, by dividing  (\ref{1dlg-E6.28}) by $(1+t)^{M}$, we get 
\begin{equation}\label{1dlg-E6.29}
\|f(t)\|_{\widetilde{L}^{2}_{\xi}(\dot{B}^{3/2}_{2,1})}\lesssim \delta_{0}(1+t)^{-\frac{1}{2}(\frac{3}{2}-\sigma_{0})},
\end{equation}
which implies that $f$ decays at the rate $\mathcal{O}(t^{-\frac{1}{2}(\frac{3}{2}-\sigma_0)})$ in the norm of $L^2_{\xi}(L^{\infty}_{x})$, owing to the embedding $\widetilde{L}^2_{\xi}(\dot{B}^{3/2}_{2,1})\hookrightarrow L^2_{\xi}L^{\infty}_{x}$. In addition, if $\sigma\in(\sigma_{0},3/2)$,  then employing  Lemma \ref{1dlg-L2.4} and $\widetilde{L}^2_{\xi}(\dot{B}^{s}_{2,1})\hookrightarrow \widetilde{L}^2_{\xi}(\dot{B}^{s}_{2,\infty})$ once again implies that
\begin{equation}\label{1dlg-E6.30}
\|f^{\ell}(t)\|_{\widetilde{L}^{2}_{\xi}(\dot{B}^{\sigma}_{2,1})}\lesssim\|f^{\ell}(t)\|^{\theta_{1}}_{\widetilde{L}^{2}_{\xi}(\dot{B}^{\sigma_{0}}_{2,\infty})}\|f^{\ell}(t)\|^{1-\theta_{1}}_{\widetilde{L}^{2}_{\xi}(\dot{B}^{3/2}_{2,1})}\lesssim \delta_{0}(1+t)^{-\frac{1}{2}(\sigma-\sigma_{0})},
\end{equation}
where $\theta_{1}=(3/2-\sigma)/(3/2-\sigma_0)\in (0,1)$. Regarding the corresponding high-frequency norm, one has  
\begin{equation}\label{1dlg-E6.3011}
\|f^{h}(t)\|_{\widetilde{L}^{2}_{\xi}(\dot{B}^{\sigma}_{2,1})}\lesssim \|f(t)\|_{\widetilde{L}^{2}_{\xi}(\dot{B}^{3/2}_{2,1})}^{h}\lesssim \delta_{0}(1+t)^{-\frac{1}{2}(\frac{3}{2}-\sigma_{0})}.
\end{equation}
Therefore,  we obtain (\ref{1dlg-E1.21}) by combining (\ref{1dlg-E6.28})-(\ref{1dlg-E6.3011}).

Next, we will establish the enhanced decay rate for the microscopic part $\{\mathbf{I-P}\}f$. Applying the microscopic projection $\{\mathbf{I-P}\}$ to \rm(\ref{1dlg-E1.4}) gives 
\begin{equation}\label{1dlg-E1.4000}
\begin{aligned}
&\partial_{t}\{\mathbf{I-P}\}f +\xi\cdot\nabla_{x}\{\mathbf{I-P}\}f +L\{\mathbf{I-P}\}f \\
&\quad=\Gamma(f,f) -\xi\cdot\nabla_{x}\mathbf{P}f +\mathbf{P}(\xi\cdot\nabla_x f) .
\end{aligned}
\end{equation} 
By employing $\dot{\Delta}_{q}$ to \eqref{1dlg-E1.4000}, taking the $L^{2}_{\xi,x}$ inner product of the resulting equation with $(1+t)^{2M^{\prime}}\dot{\Delta}_{q}\{\mathbf{I-P}\}f~(M^{\prime}\gg1)$, and then using Lemma \ref{1dlg-L2.2} and Young's inequality, we arrive at 
\begin{equation}\nonumber
\begin{aligned}
&\frac{1}{2}\frac{d}{dt}\bigg((1+t)^{2M^{\prime}}\|\dot{\Delta}_{q}\{\mathbf{I-P}\}f \|^{2}_{L^{2}_{\xi}L^{2}_{x}}\bigg)\\
&\quad\quad+\lambda_{0}(1+t)^{2M^{\prime}}\|\sqrt{\nu(\xi)}\dot{\Delta}_{q}\{\mathbf{I-P}\}f \|^{2}_{L^{2}_{\xi}L^{2}_{x}}\\
&\quad\leq 2M'(1+t)^{2M^{\prime}-1}\|\dot{\Delta}_{q}\{\mathbf{I-P}\}f \|^{2}_{L^{2}_{\xi}L^{2}_{x}}\\
&\quad\quad+(1+t)^{2M^{\prime}}\Big(\dot{\Delta}_{q}\Gamma(f,f) ,\dot{\Delta}_{q}\{\mathbf{I-P}\}f \Big)_{\xi,x}\\
&\quad\quad+(1+t)^{2M^{\prime}}\Big(-\xi\cdot\dot{\Delta}_{q}\nabla_{x}\mathbf{P}f +\dot{\Delta}_{q}\mathbf{P}(\xi\cdot\nabla_{x}f) ,\dot{\Delta}_{q}\{\mathbf{I-P}\}f\Big)_{\xi,x}\\
&\quad\leq 2M'(1+t)^{2M^{\prime}-1}\|\dot{\Delta}_{q}\{\mathbf{I-P}\}f \|^{2}_{L^{2}_{\xi}L^{2}_{x}}\\
&\quad\quad+(1+t)^{2M^{\prime}}\Big(\dot{\Delta}_{q}\Gamma(f,f) ,\dot{\Delta}_{q}\{\mathbf{I-P}\}f \Big)_{\xi,x}\\
&\quad\quad+C(1+t)^{2M^{\prime}}2^{2q}\|\dot{\Delta}_{q}f \|^{2}_{L^{2}_{\xi}L^{2}_{x}}+\frac{\lambda_{0}}{4}(1+t)^{2M^{\prime}}\|\dot{\Delta}_{q}\{\mathbf{I-P}\}f \|^{2}_{L^{2}_{\xi}L^{2}_{x}}.
\end{aligned}
\end{equation}
The use of Gronwall's inequality implies that 
\begin{equation}\label{fact}
\begin{aligned}
&(1+t)^{2M^{\prime}}\|\dot{\Delta}_{q}\{\mathbf{I-P}\}f \|_{L^{2}_{\xi}L^{2}_{x}}^2\\
&\quad\quad+\int_{0}^{t}e^{-\frac{\lambda_0}{2} (t-\tau)}(1+\tau)^{2M^{\prime}}\|\sqrt{\nu(\xi)}\dot{\Delta}_{q}\{\mathbf{I-P}\}f \|_{L^{2}_{\xi}L^{2}_{x}}^2d\tau\\
&\quad\lesssim e^{-\frac{\lambda_0}{2}t}
\|\dot{\Delta}_{q}\{\mathbf{I-P}\}f_{0} \|_{L^{2}_{\xi}L^{2}_{x}}^2\\
&\quad\quad+\int_{0}^{t}e^{-\frac{\lambda_0}{2}(t-\tau)}
(1+\tau)^{2M^{\prime}-1} \|\dot{\Delta}_{q}\{\mathbf{I-P}\}f \|^{2}_{L^{2}_{\xi}L^{2}_{x}}d\tau \\
&\quad\quad+2^{2q}\int_{0}^{t}
e^{-\frac{\lambda_0}{2}(t-\tau)}(1+\tau)^{2M^{\prime}}\|\dot{\Delta}_{q}f \|^{2}_{L^{2}_{\xi}L^{2}_{x}}d\tau\\
&\quad\quad+\int_{0}^{t}e^{-\frac{\lambda_0}{2}(t-\tau)}(1+\tau)^{2M'}\Big|\Big(\dot{\Delta}_{q}\Gamma(f,f) ,\dot{\Delta}_{q}\{\mathbf{I-P}\}f \Big)_{\xi,x}\Big|d\tau.
\end{aligned}
\end{equation}
Note that the second and third terms on the right-hand side of \eqref{fact} can be bounded by 
\begin{align}
&\int_{0}^{t}e^{-\frac{\lambda_0}{2}(t-\tau)}
(1+\tau)^{2M^{\prime}-1} \|\dot{\Delta}_{q}\{\mathbf{I-P}\}f \|^{2}_{L^{2}_{\xi}L^{2}_{x}}d\tau\\
&\quad \lesssim (1+t)^{-1} \|(1+\tau)^{M^{\prime}}\dot{\Delta}_{q}f \|^{2}_{L^{\infty}_{t}(L^{2}_{\xi}L^{2}_{x})}\nonumber
\end{align}
and
\begin{align}
2^{2q}\int_{0}^{t}
e^{-\frac{\lambda_0}{2}(t-\tau)}(1+\tau)^{2M^{\prime}}\|\dot{\Delta}_{q}f \|^{2}_{L^{2}_{\xi}L^{2}_{x}}d\tau&\lesssim 2^{2q}\|(1+\tau)^{M^{\prime}}\dot{\Delta}_{q}f \|^{2}_{L^{\infty}_{t}(L^{2}_{\xi}L^{2}_{x})}.\nonumber
\end{align}
We denote by $\|\cdot \|_{\widetilde{L}^{2}_{{\rm weight},t}\widetilde{L}^{2}_{\xi}(\dot{B}^{1/2}_{2,1})}$ the weighted norm 
$$
\|g\|_{\widetilde{L}^{2}_{{\rm weight},t}\widetilde{L}^{2}_{\xi}(\dot{B}^{1/2}_{2,1})}\triangleq\sum_{q\in\mathbb{Z}}2^{\frac{q}{2}}\bigg(\int_{0}^{t}e^{-\frac{\lambda_0}{2}(t-\tau)}(1+\tau)^{2M'}\|\sqrt{\nu(\xi)}\dot{\Delta}_{q}g\|_{L^2_{\xi}L^2_x}^2d\tau \bigg)^{1/2}.
$$
Hence, it follows from \eqref{fact} that
\begin{equation}\label{1dlg-E6.33}
\begin{aligned}
&\|(1+\tau)^{M^{\prime}}\{\mathbf{I-P}\}f \|_{\widetilde{L}^{\infty}_{t}\widetilde{L}^{2}_{\xi}(\dot{B}^{1/2}_{2,1})} +\|\{\mathbf{I-P}\}f \|_{\widetilde{L}^{2}_{{\rm weight},t}\widetilde{L}^{2}_{\xi}(\dot{B}^{1/2}_{2,1})}\\
&\quad\leq  Ce^{-\frac{\lambda_0}{4}t}\|\{\mathbf{I-P}\}f_{0} \|_{\widetilde{L}^{2}_{\xi}(\dot{B}^{1/2}_{2,1})}+
C(1+t)^{-\frac{1}{2}}\|(1+\tau)^{M^{\prime}}f \|_{\widetilde{L}^{\infty}_{t}\widetilde{L}^{2}_{\xi}(\dot{B}^{1/2}_{2,1})}\\
&\quad\quad +C\|(1+\tau)^{M^{\prime}}f \|_{\widetilde{L}^{\infty}_{t}\widetilde{L}^{2}_{\xi,\nu}(\dot{B}^{3/2}_{2,1})} +\mathcal{R}
\end{aligned}
\end{equation}
with
$$
\mathcal{R}:=C\sum_{q\in\mathbb{Z}}2^{\frac{q}{2}}\bigg(\int_{0}^{t}e^{-\frac{\lambda_0}{2}(t-\tau)}(1+\tau)^{2M^{\prime}}\Big|\Big(\dot{\Delta}_{q}\Gamma(f,f),  \dot{\Delta}_{q}\{\mathbf{I-P}\}f \big)_{\xi,x}\Big|d\tau\bigg)^{1/2}.
$$
Having the time-weighted estimate (\ref{1dlg-E6.28}) at hand with $M=M'$, we deduce that
\begin{equation}\label{1dlg-E6.34}
\|(1+\tau)^{M^{\prime}}f \|_{\widetilde{L}^{\infty}_{t}\widetilde{L}^{2}_{\xi}(\dot{B}^{3/2}_{2,1})}\lesssim \delta_{0}(1+t)^{M^{\prime}-\frac{1}{2}(\frac{3}{2}-\sigma_{0})}.
\end{equation}
It follows from  Lemma \ref{1dlg-L2.4}, \eqref{1dlg-E6.24} and (\ref{1dlg-E6.28}) with $M=M^{\prime}/(1-\theta_2)$ as well as $\theta_2=1/(3/2-\sigma_0)$ that
\begin{equation}\label{1dlg-E6.351}
\begin{aligned}
&\|(1+\tau)^{M^{\prime}}f\|_{\widetilde{L}^{\infty}_{t}\widetilde{L}^{2}_{\xi}(\dot{B}^{1/2}_{2,1})}\\
&\quad\lesssim \|f^{\ell}\|_{\widetilde{L}^{\infty}_{t}\widetilde{L}^{2}_{\xi}(\dot{B}^{\sigma_0}_{2,\infty})}^{\theta_2}\|(1+\tau)^{\frac{M^{\prime}}{1-\theta_2}}f^{\ell}\|_{\widetilde{L}^{\infty}_{t}\widetilde{L}^{2}_{\xi}(\dot{B}^{3/2}_{2,1})}^{1-\theta_2}+\|(1+\tau)^{M^{\prime}}f\|_{\widetilde{L}^{\infty}_{t}\widetilde{L}^{2}_{\xi}(\dot{B}^{3/2}_{2,1})}^{h}\\
&\quad\lesssim \delta_0(1+t)^{M'-\frac{1}{2}(\frac{1}{2}-\sigma_0)}.
\end{aligned}
\end{equation}

In what follows, we turn to the nonlinear term $\mathcal{R}$ in \eqref{1dlg-E6.33}. To begin with, we decompose $\Gamma(f,f)$ as
$$
\Gamma(f,f)=\Gamma(\mathbf{P}f,\mathbf{P}f)+\Gamma(\{\mathbf{I-P}\}f,\mathbf{P}f)+\Gamma(f,\{\mathbf{I-P}\}f).
$$
Regarding the term $\Gamma(\mathbf{P}f,\mathbf{P}f)$, it follows from Lemma \ref{1dlg-L2.7} that
\begin{equation}\nonumber
\begin{aligned}
&\sum_{q\in\mathbb{Z}}2^{\frac{q}{2}}\Big(\int_{0}^{t}e^{-\frac{\lambda_0}{2}(t-\tau)}(1+\tau)^{2M^{\prime}}\big|\big(\dot{\Delta}_{q}\Gamma(\mathbf{P}f,\mathbf{P}f) ,\dot{\Delta}_{q}\{\mathbf{I-P}\}f \big)_{\xi,x}\big|d\tau\Big)^{1/2}\\
&\quad\leq 
C\|e^{-\frac{\lambda_0}{8}(t-\tau)}(1+\tau)^{\frac{M^{\prime}}{2}}\mathbf{P}f\|^{1/2}_{\widetilde{L}^{\infty}_{t}\widetilde{L}^{2}_{\xi}(\dot{B}^{1/2}_{2,1})}\|e^{-\frac{\lambda_0}{8}(t-\tau)}(1+\tau)^{\frac{M^{\prime}}{2}}\mathbf{P}f\|^{1/2}_{\widetilde{L}^{2}_{t}\widetilde{L}^{2}_{\xi,\nu}(\dot{B}^{3/2}_{2,1})}\\
&\quad\quad\times\|\{\mathbf{I-P}\}f \|_{\widetilde{L}^{2}_{{\rm weight},t}\widetilde{L}^{2}_{\xi}(\dot{B}^{1/2}_{2,1})}^{1/2}.
\end{aligned}
\end{equation}
By using (\ref{1dlg-E6.2}) with $M=M'/2$ and arguing similarly as in \eqref{1dlg-E6.351}, we have
\begin{equation}\nonumber
\begin{aligned}
&\|e^{-\frac{\lambda_0}{8}(t-\tau)}(1+\tau)^{\frac{M^{\prime}}{2}}\mathbf{P}f\|_{\widetilde{L}^{\infty}_{t}\widetilde{L}^{2}_{\xi}(\dot{B}^{1/2}_{2,1})}\\
&\quad\lesssim \|(1+\tau)^{\frac{M^{\prime}}{2}}f \|_{\widetilde{L}^{\infty}_{t}\widetilde{L}^{2}_{\xi}(\dot{B}^{1/2}_{2,1})}\lesssim \delta_0 (1+t)^{\frac{M^{\prime}}{2}-\frac{1}{2}(\frac{1}{2}-\sigma_0)}
\end{aligned}
\end{equation}
and
\begin{equation}\nonumber
\begin{aligned}
&\|e^{-\frac{\lambda_0}{8}(t-\tau)}(1+\tau)^{\frac{M^{\prime}}{2}}\mathbf{P}f\|_{\widetilde{L}^{2}_{t}\widetilde{L}^{2}_{\xi,\nu}(\dot{B}^{3/2}_{2,1})}\\
&\quad \lesssim \bigg(\int_{0}^{t}e^{-\frac{\lambda_0}{4}(t-\tau)}d\tau\bigg)^{1/2}\|(1+\tau)^{\frac{M^{\prime}}{2}}f\|_{\widetilde{L}^{\infty}_{t}\widetilde{L}^{2}_{\xi}(\dot{B}^{3/2}_{2,1})} \lesssim \delta_0(1+t)^{\frac{M^{\prime}}{2}-\frac{1}{2}(\frac{3}{2}-\sigma_0)}.
\end{aligned}
\end{equation}
Thus, it holds that
\begin{equation*}
\begin{aligned}
&\sum_{q\in\mathbb{Z}}2^{\frac{q}{2}}\bigg(\int_{0}^{t}e^{-\frac{\lambda_0}{2}(t-\tau)}(1+\tau)^{2M^{\prime}}\Big|\Big(\dot{\Delta}_{q}\Gamma(\mathbf{P}f,\mathbf{P}f) ,\dot{\Delta}_{q}\{\mathbf{I-P}\}f \big)_{\xi,x}\Big|d\tau\bigg)^{1/2}\\
&\quad\leq \frac{1}{8}\|\{\mathbf{I-P}\}f \|_{\widetilde{L}^{2}_{{\rm weight},t}\widetilde{L}^{2}_{\xi}(\dot{B}^{1/2}_{2,1})}+C\delta_0^2 (1+t)^{M'-(1-\sigma_0)}.
\end{aligned}
\end{equation*}
Recall that (\ref{1dlg-E6.2}) ensures that
$$
\|(1+\tau)^{M^{\prime}}\{\mathbf{I-P}\}f\|_{\widetilde{L}^{2}_{t}\widetilde{L}^{2}_{\xi,\nu}(\dot{B}^{3/2}_{2,1})}\lesssim \delta_0(1+t)^{M^{\prime}-\frac{1}{2}(\frac{3}{2}-\sigma_0)}.
$$
Thus, for $\Gamma(\{\mathbf{I-P}\}f,\mathbf{P}f)$, employing Lemma \ref{1dlg-L2.7} guarantees that 
\begin{equation*}
\begin{aligned}
&\sum_{q\in\mathbb{Z}}2^{\frac{q}{2}}\bigg(\int_{0}^{t}e^{-\frac{\lambda_0}{2}(t-\tau)}(1+\tau)^{2M^{\prime}}\big|\big(\dot{\Delta}_{q}\Gamma(\{\mathbf{I-P}\}f,\mathbf{P}f) ,\dot{\Delta}_{q}\{\mathbf{I-P}\}f \big)_{\xi,x}\big|d\tau\bigg)^{1/2}\\
&\quad\leq C\bigg(\|(1+\tau)^{M^{\prime}}\{\mathbf{I-P}\}f\|^{1/2}_{\widetilde{L}^{\infty}_{t}\widetilde{L}^{2}_{\xi}(\dot{B}^{1/2}_{2,1})}\|\mathbf{P}f\|^{1/2}_{\widetilde{L}^{2}_{t}\widetilde{L}^{2}_{\xi,\nu}(\dot{B}^{3/2}_{2,1})}\\
&\quad\quad+\|(1+\tau)^{M^{\prime}}\{\mathbf{I-P}\}f\|^{1/2}_{\widetilde{L}^{2}_{t}\widetilde{L}^{2}_{\xi,\nu}(\dot{B}^{3/2}_{2,1})}\|\mathbf{P}f\|^{1/2}_{\widetilde{L}^{\infty}_{t}\widetilde{L}^{2}_{\xi}(\dot{B}^{1/2}_{2,1})}\bigg)\|\{\mathbf{I-P}\}f \|_{\widetilde{L}^{2}_{{\rm weight},t}\widetilde{L}^{2}_{\xi}(\dot{B}^{1/2}_{2,1})}^{1/2}\\
&\quad\leq \frac{1}{8}\|\{\mathbf{I-P}\}f \|_{\widetilde{L}^{2}_{{\rm weight},t}\widetilde{L}^{2}_{\xi}(\dot{B}^{1/2}_{2,1})}\\
&\quad\quad+C\|\mathbf{P}f\|_{\widetilde{L}^{2}_{t}\widetilde{L}^{2}_{\xi}(\dot{B}^{3/2}_{2,1})}\|(1+\tau)^{M^{\prime}}\{\mathbf{I-P}\}f \|_{\widetilde{L}^{\infty}_{t}\widetilde{L}^{2}_{\xi}(\dot{B}^{1/2}_{2,1})}+C\delta_0^2(1+t)^{M'-\frac{1}{2}(\frac{3}{2}-\sigma_0)}.
\end{aligned}
\end{equation*}
Similarly, it holds that
\begin{equation*}
\begin{aligned}
&\sum_{q\in\mathbb{Z}}2^{\frac{q}{2}}\bigg(\int_{0}^{t}e^{-\frac{\lambda_0}{2}(t-\tau)}(1+\tau)^{2M^{\prime}}\big|\big(\dot{\Delta}_{q}\Gamma(f,\{\mathbf{I-P}\}f) ,\dot{\Delta}_{q}\{\mathbf{I-P}\}f \big)_{\xi,x}\big|d\tau\bigg)^{1/2}\\
&\quad\leq C\bigg(\|f\|^{1/2}_{\widetilde{L}^{\infty}_{t}\widetilde{L}^{2}_{\xi}(\dot{B}^{1/2}_{2,1})}\|(1+\tau)^{M^{\prime}}\{\mathbf{I-P}\}f\|^{1/2}_{\widetilde{L}^{2}_{t}\widetilde{L}^{2}_{\xi,\nu}(\dot{B}^{3/2}_{2,1})}\\
&\quad\quad+\|f\|^{1/2}_{\widetilde{L}^{2}_{t}\widetilde{L}^{2}_{\xi,\nu}(\dot{B}^{3/2}_{2,1})}\|(1+\tau)^{M^{\prime}}\{\mathbf{I-P}\}f\|^{1/2}_{\widetilde{L}^{\infty}_{t}\widetilde{L}^{2}_{\xi}(\dot{B}^{1/2}_{2,1})}\bigg)\|\{\mathbf{I-P}\}f \|_{\widetilde{L}^{2}_{{\rm weight},t}\widetilde{L}^{2}_{\xi}(\dot{B}^{1/2}_{2,1})}^{1/2}\\
&\quad\leq \frac{1}{8}\|\{\mathbf{I-P}\}f \|_{\widetilde{L}^{2}_{{\rm weight},t}\widetilde{L}^{2}_{\xi}(\dot{B}^{1/2}_{2,1})}+C\|f\|_{\widetilde{L}^{2}_{t}\widetilde{L}^{2}_{\xi,\nu}(\dot{B}^{3/2}_{2,1})}\|(1+\tau)^{M^{\prime}}\{\mathbf{I-P}\}f \|_{\widetilde{L}^{\infty}_{t}\widetilde{L}^{2}_{\xi}(\dot{B}^{1/2}_{2,1})}\\
&\quad\quad+C\delta^2_0(1+t)^{M'-\frac{1}{2}(\frac{3}{2}-\sigma_0)}.
\end{aligned}
\end{equation*}
Therefore, combining the above estimates yields
\begin{equation}\label{R}
\begin{aligned}
\mathcal{R}&\leq \frac{3}{8}\|\{\mathbf{I-P}\}f \|_{\widetilde{L}^{2}_{{\rm weight},t}\widetilde{L}^{2}_{\xi}(\dot{B}^{1/2}_{2,1})}\\
&\quad+C\|f\|_{\widetilde{L}^{2}_{t}\widetilde{L}^{2}_{\xi,\nu}(\dot{B}^{3/2}_{2,1})}\|(1+\tau)^{M^{\prime}}\{\mathbf{I-P}\}f \|_{\widetilde{L}^{\infty}_{t}\widetilde{L}^{2}_{\xi}(\dot{B}^{1/2}_{2,1})}\\
&\quad+C\delta^2_0(1+t)^{M'-\frac{1}{2}(\frac{3}{2}-\sigma_0)}+C\delta_0^2 (1+t)^{M'-(1-\sigma_0)}.
\end{aligned}
\end{equation}
To proceed, we substitute \eqref{R} into \eqref{1dlg-E6.33} to get
\begin{equation}\nonumber
\begin{aligned}
&\|(1+\tau)^{M^{\prime}}\{\mathbf{I-P}\}f \|_{\widetilde{L}^{\infty}_{t}\widetilde{L}^{2}_{\xi}(\dot{B}^{1/2}_{2,1})} +\|\{\mathbf{I-P}\}f \|_{\widetilde{L}^{2}_{{\rm weight},t}\widetilde{L}^{2}_{\xi}(\dot{B}^{1/2}_{2,1})}\\
&\quad\leq  Ce^{-\frac{\lambda_0}{4}t}\|\{\mathbf{I-P}\}f_{0} \|_{\widetilde{L}^{2}_{\xi}(\dot{B}^{1/2}_{2,1})}+(\delta_{0}+\delta_0^2) (1+t)^{M^{\prime}-\frac{1}{2}(\frac{3}{2}-\sigma_{0})}+\delta^2_{0}(1+t)^{M^{\prime}-(1-\sigma_{0})}\\
&\quad\quad+\|f\|_{\widetilde{L}^{2}_{t}\widetilde{L}^{2}_{\xi,\nu}(\dot{B}^{3/2}_{2,1})}\|(1+\tau)^{M^{\prime}}\{\mathbf{I-P}\}f \|_{\widetilde{L}^{\infty}_{t}\widetilde{L}^{2}_{\xi}(\dot{B}^{1/2}_{2,1})}.
\end{aligned}
\end{equation}
Keeping in mind that $\|f\|_{\widetilde{L}^{2}_{t}\widetilde{L}^{2}_{\xi,\nu}(\dot{B}^{3/2}_{2,1})}\lesssim \var_0<<1$ and $\frac{1}{2}(\frac{3}{2}-\sigma_{0})< 1-\sigma_{0}$ for any $\sigma_0\in [-3/2,1/2)$, we conclude that 
\begin{equation}\nonumber
\begin{aligned}
&\|(1+\tau)^{M^{\prime}}\{\mathbf{I-P}\}f \|_{\widetilde{L}^{\infty}_{t}\widetilde{L}^{2}_{\xi}(\dot{B}^{1/2}_{2,1})} +\|\{\mathbf{I-P}\}f \|_{\widetilde{L}^{2}_{{\rm weight},t}\widetilde{L}^{2}_{\xi}(\dot{B}^{1/2}_{2,1})}\\
&\quad\lesssim (\delta_0+\delta_0^2)(1+t)^{M'-\frac{1}{2}(\frac{3}{2}-\sigma_0)},
\end{aligned}
\end{equation}
which leads to
\begin{equation}\label{1dlg-E6.38}
\|\{\mathbf{I-P}\}f(t)\|_{\widetilde{L}^{2}_{\xi}(\dot{B}^{1/2}_{2,1})}\lesssim (1+\delta_0)\delta_0 (1+t)^{-\frac{1}{2}(\frac{3}{2}-\sigma_{0})}.
\end{equation}
Similarly, for all $j\in\mathbb{Z}$ and any $M'>>1$, one deduces from \eqref{fact} that
\begin{equation*}
\begin{aligned}
&(1+t)^{2M^{\prime}}2^{j\sigma_0}\|\dot{\Delta}_{q}\{\mathbf{I-P}\}f \|_{L^{2}_{\xi}L^{2}_{x}}^2\\
&\quad\quad+\int_{0}^{t}e^{-\frac{\lambda_0}{2} (t-\tau)}(1+\tau)^{2M^{\prime}}2^{j\sigma_0}\|\sqrt{\nu(\xi)}\dot{\Delta}_{q}\{\mathbf{I-P}\}f \|_{L^{2}_{\xi}L^{2}_{x}}^2d\tau\\
&\quad\lesssim e^{-\frac{\lambda_0}{2}t}2^{j\sigma_0}\|\dot{\Delta}_{q}\{\mathbf{I-P}\}f \|_{L^{2}_{\xi}L^{2}_{x}}^2+\int_0^t e^{-\frac{\lambda_0(t-\tau)}{2}}(1+\tau)^{2M^{\prime}-1}d\tau \|f\|_{L^{\infty}_t \widetilde{L}^2_{\xi}(\dot{B}^{\sigma_0}_{2,\infty})}^2\\
&\quad\quad+\int_0^t e^{-\frac{\lambda_0(t-\tau)}{2}}(1+\tau)^{2M^{\prime}-1}d\tau \|(1+\tau)f\|_{L^{\infty}_t \widetilde{L}^2_{\xi}(\dot{B}^{\sigma_0+1}_{2,\infty}) }^2\\
&\quad\quad+ \|f\|_{L^{\infty}_t \widetilde{L}^2_\xi(\dot{B}^{\sigma_0}_{2,1}) }\|(1+t)^{M'} f\|_{L^2_t \widetilde{L}^2_{\xi,\nu}(\dot{B}^{\frac{3}{2}}_{2,1})}  \sup_{q\in\mathbb{Z}}\bigg(\int_{0}^{t}e^{-\frac{\lambda_0}{2} (t-\tau)}(1+\tau)^{2M^{\prime}}\|\sqrt{\nu(\xi)}\dot{\Delta}_q\{\mathbf{I-P}\}f \|_{L^2_{\xi}L^2_{x}}^2 d\tau \bigg)^{\frac{1}{2}}\\
&\quad\lesssim (1+t)^{2M^{\prime}-1}\bigg( \|\dot{\Delta}_{q}\{\mathbf{I-P}\}f \|_{L^{2}_{\xi}L^{2}_{x}}^2+\|f\|_{L^{\infty}_t \widetilde{L}^2_{\xi}(\dot{B}^{\sigma_0}_{2,\infty})}^2+\|(1+\tau)f\|_{L^{\infty}_t \widetilde{L}^2_{\xi}(\dot{B}^{\sigma_0+1}_{2,\infty}) }^2\\
&\quad\quad+\|f\|_{L^{\infty}_t \widetilde{L}^2_\xi(\dot{B}^{\sigma_0}_{2,1}) }^2 \|(1+t)^{M'} f\|_{L^2_t \widetilde{L}^2_{\xi,\nu}(\dot{B}^{\frac{3}{2}}_{2,1})}^2 \bigg)\\
&\quad+\frac{1}{2}\sup_{q\in\mathbb{Z}}\int_{0}^{t}e^{-\frac{\lambda_0}{2} (t-\tau)}(1+\tau)^{2M^{\prime}}\|\sqrt{\nu(\xi)}\dot{\Delta}_q\{\mathbf{I-P}\}f \|_{\widetilde{L}^2_{\xi,\nu}(\dot{B}^{\sigma_0}_{2,\infty})}^2d\tau.
\end{aligned}
\end{equation*}
This, together with \eqref{1dlg-E6.2}, \eqref{1dlg-E6.24} and \eqref{1dlg-E6.30}, yields
\begin{equation}\label{733}
\begin{aligned}
\|\{\mathbf{I-P}\}f(t)\|_{\widetilde{L}^2_{\xi}(\dot{B}^{\sigma_0}_{2,\infty})}\lesssim (1+\delta_0)\delta_0(1+t)^{-\frac{1}{2}}.
\end{aligned}
\end{equation}
Furthermore, if $\sigma\in(\sigma_0,1/2)$, then the real interpolation inequality, \eqref{1dlg-E6.38} and \eqref{733}
enable us to get
\begin{equation}\label{1dlg-E6.39}
\begin{aligned}
&\|\{\mathbf{I-P}\}f^{\ell}(t)\|_{\widetilde{L}^{2}_{\xi}(\dot{B}^{\sigma}_{2,1})}\\
&\quad \lesssim\|\{\mathbf{I-P}\}f^{\ell}(t)\|^{\theta_{2}}_{\widetilde{L}^{2}_{\xi}(\dot{B}^{\sigma_{0}}_{2,\infty})}\|\{\mathbf{I-P}\}f^{\ell}(t)\|^{1-\theta_{2}}_{\widetilde{L}^{2}_{\xi}(\dot{B}^{1/2}_{2,1})}\\
&\quad\lesssim(1+\delta_0)\delta_0  (1+t)^{-\frac{1}{2}(\sigma-\sigma_{0}+1)}
\end{aligned}
\end{equation}
for $\theta_{2}=(1/2-\sigma)/(1/2-\sigma_{0})\in (0,1)$. In addition, it follows from (\ref{1dlg-E6.28}) that
\begin{equation}\label{1dlg-E6.40}
\begin{aligned}
\|\{\mathbf{I-P}\}f^{h}(t)\|_{\widetilde{L}^{2}_{\xi}(\dot{B}^{\sigma}_{2,1})}&\lesssim\|f(t)\|^{h}_{\widetilde{L}^{2}_{\xi}(\dot{B}^{3/2}_{2,1})}\lesssim \delta_{0} (1+t)^{-\frac{1}{2}(\frac{3}{2}-\sigma_{0})}.
\end{aligned}
\end{equation}
Hence, \eqref{1dlg-E1.22}
is followed by (\ref{1dlg-E6.39})-(\ref{1dlg-E6.40}) directly.  The proof of Theorem \ref{1dlg-T1.3} is complete.   \hfill $\Box$

\bigbreak
\bigbreak
\textbf{Acknowledgments}
The authors would like to thank anonymous referees for their valuable comments. L.-Y. Shou is supported by the National Natural Science Foundation of China (12301275). J. Xu is partially supported by the National Natural Science Foundation of China (12271250, 12031006) and the Fundamental Research Funds for the Central Universities, NO. NP2024105.

\vspace{2mm}

\textbf{Conflict of interest.} The authors have no possible conflicts of interest.

\vspace{2mm}

\textbf{Data availability statement.}
Data sharing is not applicable to this article as no data sets were generated or analyzed during the current study.

\newpage

\vspace{5ex}

\bigbreak

(J. Liu)\par\nopagebreak
\noindent\textsc{School of Mathematics and Key Laboratory of Mathematical MIIT, Nanjing University of Aeronautics and
Astronautics, Nanjing, 211106, P. R. China}

Email address: {liujing\_18@nuaa.edu.cn}

\vspace{3ex}

\vspace{3ex}

(L.-Y. Shou)\par\nopagebreak
\noindent\textsc{School of Mathematics and Key Laboratory of Mathematical MIIT, Nanjing University of Aeronautics and
Astronautics, Nanjing, 211106, P. R. China}\\
School of Mathematical Sciences and Ministry of Education Key Laboratory of NSLSCS, Nanjing Normal University, Nanjing 210023, China

Email address: {shoulingyun11@gmail.com}

\vspace{3ex}

\vspace{3ex}

(J. Xu)\par\nopagebreak
\noindent\textsc{School of Mathematics and Key Laboratory of Mathematical MIIT, Nanjing University of Aeronautics and
Astronautics, Nanjing, 211106, P. R. China}

Email address: {jiangxu\_79@nuaa.edu.cn}
\end{document}